\documentclass[12pt]{article}
\pdfoutput=1

{\begin{Sbox}\begin{minipage}}%
{\end{minipage}\end{Sbox}\fbox{\TheSbox}}%

\usepackage[psamsfonts]{amssymb}
\usepackage{amsmath, amsthm, anysize, enumerate, color, url}
\usepackage{graphicx}
\usepackage{subcaption}
\usepackage{multirow}
\usepackage{footnote}
\usepackage{epstopdf}
\usepackage{epsfig}
\usepackage[ruled,vlined]{algorithm2e}
\usepackage[sectionbib]{natbib}

\usepackage{threeparttable}

\usepackage[colorlinks, breaklinks,
                   linkcolor=red,
                  citecolor=blue
                  ]{hyperref}

\usepackage{breakurl}

\newtheorem{theorem}{Theorem} 

\newtheorem{lemma}{Lemma} 

\theoremstyle{definition}
\newtheorem{remark}{Remark}  

\newtheorem{condition}{Condition}

\newcommand{\R}{\mathbb{R}}

\newcommand{\PP}{\mathbb{P}}

\renewcommand{\hat}{\widehat}
\renewcommand{\tilde}{\widetilde}

\usepackage{setspace}

\begin{document}
\title{A degree-corrected Cox model for dynamic networks\footnote{The authors are listed in the alphabetical order.}}

\makeatother

\author{Yuguo Chen\footnote{Department of Statistics, University of Illinois at Urbana-Champaign, Champaign, IL 61820.
\texttt{Email:} yuguo@illinois.edu.}, 
Lianqiang Qu\footnote{School of Mathematics and Statistics, Central China Normal University,  Wuhan, Hubei, 430079, P.R.China.
\texttt{Email:} qulianq@ccnu.edu.cn.}
, 
Jinfeng Xu\footnote{Department of Biostatistics, City University of Hong Kong, Tat Chee Avenue, Kowloon, Hong Kong.
\texttt{Email:} jinfengxu@gmail.com.}
, 
Ting Yan\footnote{Department of Statistics, Central China Normal University, Wuhan, P.R. China, 430079.
\texttt{Email:} tingyanty@mail.ccnu.edu.cn.}
, 
Yunpeng Zhou\footnote{Department of Statistics and Actuarial Science, The University of Hong Kong, Hong Kong.
\texttt{Email:} u3514104@connect.hku.hk.}
\\
\footnotemark[2] University of Illinois at Urbana-Champaign,  \\
\footnotemark[3] \footnotemark[5]\ \ Central China Normal University, \\
\footnotemark[4] City University of Hong Kong, \\
\footnotemark[6] The University of Hong Kong
}
\date{}

\maketitle
\begin{abstract}{
Continuous time network data have been successfully modeled by 
multivariate counting processes,
in which the intensity function is characterized by covariate information.
However, degree heterogeneity has not been incorporated into the model 
which may lead to large biases for the estimation of homophily effects. 
In this paper, we propose a degree-corrected Cox network model 
to simultaneously analyze the dynamic degree heterogeneity and homophily effects for continuous time directed network data. 
Since each node has individual-specific in- and out-degree effects in the model,
the dimension of the time-varying parameter vector grows with the number of nodes, 
which makes the estimation problem non-standard. 
We develop a local estimating equations approach to estimate unknown time-varying parameters,
and establish consistency and asymptotic normality of 
the proposed estimators by using the powerful martingale process theories.
We further propose test statistics to test for trend and degree heterogeneity in dynamic networks. 
Simulation studies are provided to assess the finite sample performance of the proposed method and
a real data analysis is used to illustrate its practical utility.
}


\end{abstract}

\vskip 5pt


\section{Introduction}
\label{section-introduction}

Networks are very common in a wide variety of fields, 
including social sciences, biological sciences, transportation systems, and power grids.
In networks, nodes are used to represent the entities of interest, 
and edges are used to represent interactions among the nodes.
For instance, in an email network, nodes represent users and edges represent email communications between users.
Statistical models are useful tools to analyze the interactions in networks 
(e.g., \citealp{Goldenberg2010, Fienberg2012}). See \cite{Kolaczyk2009} for a comprehensive review
on the statistical analysis of network data.

In many networks, the interactions (such as emails and phone calls) between nodes vary over time.
Modelling and inferring from such dynamic network data has attracted great interests in recent years.
One common approach is to aggregate network data on predefined time intervals to obtain a sequence of discrete time-stamped snapshots of random graphs with unweighted or weighted edges.
See for example, temporal exponential random graph models
\citep{hanneke2010discrete,krivitsky2014separable}, dynamic stochastic block models \citep{yang2011detecting,matias2017statistical,pensky2019}, and dynamic latent space models \citep{sewell2015analysis,sewell2015latent}.
However, observation/interaction times are continuous and irregular in many scenarios, 
including email networks \citep{perry2013point}, Twitter direct messages networks \citep{DBS2013}, and bike share networks \citep{MRV2018}. 
As discussed in \cite{perry2013point}, inference 
based on transforming the interaction counts into binary edges depends on the choice of the threshold, 
which may lead to dramatically different networks and conclusions \citep{Choudhury2010}. 
In addition, the aggregation of data also depends crucially on the choice of the time intervals which may lead to different networks and inference as well.
Therefore, it is more desirable to develop continuous-time network models to make full use of the data.

One natural way to model continuous time interactions amongst nodes is using counting processes (e.g., \citealp{butts20084,
	Vu-hunter2011dynamic,DBS2013, MRV2018}). 
For example, \cite{perry2013point} proposed a Cox-type regression model for intensity function
to analyze dynamic directed interactions, and developed partial likelihood inference for their model.
Instead of using intensity functions, \cite{Sit2020} proposed to use the rate function to model directed interactions.
However, both \cite{perry2013point} and \cite{Sit2020} assumed that the regression parameters are constant over time.
Since the regression parameters often change with time, it is important to know the
time-varying effects of covariates on interactions.
Recently, \cite{kreib2019} extended \citeauthor{perry2013point}'s work to time-varying coefficients Cox model 
that can characterize temporal effects of covariates on interactions, 
and established pointwise consistency and asymptotic normality of the local maximum likelihood estimator.

The above literature on counting processes for dynamic networks mainly focus on 
accounting for the effects of covariates. They do not consider degree heterogeneity, which means individuals exhibit substantial variation in their interactions with others. 
This is another important feature of many real-world networks.  
In our simulation studies, 
we found that the estimator of homophily parameters has a large bias when the degree heterogeneity exists but is not considered in the model; see Figures \ref{fig:het}.
This phenomenon was also observed by \cite{Graham2017} in static undirected networks, where the presence of ``hub" nodes (i.e., nodes with high degrees),
who form many interactions with nodes of different kinds,
effectively attenuates measured network homophily. 
Here homophily means individuals with similar covariate values are easier to form connections with each other.


In this article, we aim to model dynamic degree heterogeneity and homophily effects simultaneously for continuous time directed network data.
Our contributions are three-folds.
First, we propose a novel model, called degree-corrected Cox network model in \eqref{eq2.1}, to address 
dynamic degree heterogeneity 
across the $n$ nodes in the network.
The new model contains a set of $2n$ time-varying degree parameters measuring dynamic degree heterogeneity 
and a $p$-dimensional time-varying regression coefficient for dynamic homophily effect.
A local estimating equations method is developed to estimate $2n+p$ unknown time-varying parameters. 
Second, we establish consistency and asymptotic normality of the proposed estimators by combining
the martingale process theories and kernel smoothing methods. 
The main challenge is that the dimension of the parameter vector grows with the number of nodes in the network, 
that is, our setting is in the high-dimensional paradigm. This is different from \cite{perry2013point} and \cite{ kreib2019}, 
where the number of unknown parameters of interest is fixed, 
so their methods cannot be applied here.
Third, in practice it is not always clear whether a network has degree heterogeneity, especially when the network is sparse.
Thus, we further propose test statistics to check whether there is degree heterogeneity in a dynamic work.
To the best of our knowledge, such a test has not yet been explored in existing literature.

The rest of this paper is organized as follows.  
Section \ref{section:model} introduces our proposed degree-corrected Cox network model. 
Section \ref{section-approach} develops a local estimating equations approach to estimate the $2n+p$ parameter functions.
Section \ref{section:asymptotic} establishes 
uniform consistency of the estimators and their point-wise central limit theorems, and constructs confidence intervals  for parameters. 
Section \ref{section:test} develops testing statistics to test for trend and degree heterogeneity.
Section \ref{section:simulation} provides simulation studies and presents an application to a real data analysis. 
All the proofs are presented in the supplementary material.

\section{Degree-corrected Cox network model}
\label{section:model}

We consider a network with $n$ nodes labelled as ``$1, \ldots, n$", and study directed interaction processes amongst nodes. A directed interaction from head node $i$ to tail node $j$ 
means a direct edge from $i$ to $j$.
We use $[n]$ and $[n-1]_n$ to denote the integer sets $\{1,\dots,n\}$ and $\{n+1,\dots,2n-1\},$ respectively. 
For any two nodes $i\neq j\in[n]$, define 
$N_{ij}(s,t)$ as the number of directed interactions from head node $i$ to tail node $j$ in the time interval $(s,t].$ 
Write $N_{ij}(t)=N_{ij}(0,t)$. Without loss of generality, we assume that the interaction process
starts at $t = 0$ with $N_{ij}(0) = 0$
for $1\le i\not=j \le n.$
The counting process $\{N_{ij}(t): t\ge 0\}$ encodes occurrences of directed interactions  from head node $i$ to tail node $j$.

Let $Z_{ij}(t)$ be a $p$-dimensional external covariate process for the node pair $(i,j).$
The covariate $Z_{ij}(t)$ can be used to
measure the similarity or dissimilarity of individual-level characteristics. 
For example, if individual $i$ has a $d$-dimensional characteristic $X_i(t)$,
the pairwise covariate $Z_{ij}(t)$ can be constructed by setting $Z_{ij}(t)=X_i(t)^\top X_j(t)$,
$Z_{ij}(t)=\|X_i(t)-X_j(t)\|_2$, or $Z_{ij}(t)=X_i(t)\otimes X_j(t)$. Here $\|x\|_2$ denotes the $\ell_2$-norm of any vector $x\in\mathbb{R}^d$ and
$x\otimes y$ denotes the Kronecker product of the vector $x$ and the vector $y$. 

The observations consist of $\{N_{ij}(t), Z_{ij}(t): i\neq j\in[n],~ t\in [0,\tau]\},$
where $\tau$ is the termination time of the observation period.
Let $\mathcal{F}_t$ denote the $\sigma$-filtration that represents everything that happened up to time $t$. 
Define 
\[\lambda_{ij}(t|\mathcal{F}_t)=\lim_{\Delta t\rightarrow 0^+}P\big(N_{ij}((t+\Delta t)^-)-N_{ij}(t^-)=1|\mathcal{F}_t\big)/\Delta t\] 
as the intensity function of $N_{ij}(t),$
where $N_{ij}(t^-)$ denotes the number of interactions before time $t$.
We propose the following model for the intensity function of $N_{ij}(t)$:
\begin{align}
	\label{eq2.1}
	\lambda_{ij}(t|\mathcal{F}_{t})
	&=\exp\{\alpha_i(t)+\beta_j(t)+Z_{ij}(t)^\top\gamma(t)\}, \quad 1\le i\not= j \le n,
\end{align}
where $\alpha_i(t)$ denotes the outgoingness of node $i$, $\beta_j(t)$ denotes the popularity of node $j$,
and $\gamma(t)$ is a common regression coefficient function.
As discussed in \cite{perry2013point} and \cite{yan2019},  $\gamma(t)$ concisely captures and estimates the homophily effects of covariates.
The larger the term $Z_{ij}(t)^\top\gamma(t)$ is, the more likely homophilous nodes interact with each other.

Similar to the explanations of model parameters in the $p_1$ model for static networks \citep{Paul:Leinhardt1981}, 
the parameters $\alpha_i(t)$ and  $\beta_j(t)$ measure degree heterogeneity.
To see this clearly in our model, we consider the special case 
that $N_{ij}(t)$ is a Poisson process.
Note that the out- and in-degrees for node $i$ in time interval $(s,t]$
are $\sum_{k\neq i}N_{ik}(s,t)$ and $\sum_{i\neq k}N_{ik}(s,t),$ respectively.
In the case of the Poisson process, we have
\begin{align}
	\sum_{k=1,k\neq i}^nE\{N_{ik}(s,t)\}=&\int_{s}^t\exp\{\alpha_i(u)\}\sum_{k\neq i}^n\exp\{\beta_k(u)
	+Z_{ik}(u)^\top\gamma(u)\}du, \label{eq2.2}\\
	\sum_{k=1,k\neq i}^nE\{N_{ki}(s,t)\}=&\int_{s}^t\exp\{\beta_i(u)\}\sum_{k\neq i}^n\exp\{\alpha_k(u)
	+Z_{ki}(u)^\top\gamma(u)\}du. \label{eq2.3}
\end{align}
As we can see, the larger $\alpha_i(t)$ is, the more likely node $i$ interacts with others.
Therefore, $\alpha_i(t)~(i\in[n])$  accommodates the out-degree heterogeneity across different nodes.
On the other hand,  the larger $\beta_i(t)$ is, the more likely node $i$ receives  interactions from other nodes,
that is, $\beta_i(t)$ reflects the in-degree heterogeneity.
As a result, the effect of degree heterogeneity can be clearly delineated by estimating 
the node-specific parameters $\alpha_i(t)$ and $\beta_j(t)$.

In model \eqref{eq2.1}, if we treat 
$\lambda_{0,ij}(t):=\exp\{\alpha_i(t)+\beta_j(t)\}$ as the baseline intensity function,
it reduces to a network version of the well known Cox regression model. 
Because of this fact, we call model \eqref{eq2.1} the \emph{degree-corrected Cox network model}.
Note that if one transforms $\alpha_i(t)+\beta_j(t)$ to $[\alpha_i(t)+c(t)]+[\beta_j(t)-c(t)]$ by a unknown function $c(t),$
then model \eqref{eq2.1} does not change.
For the identifiability of model \eqref{eq2.1}, 
in what follows we set $\beta_n(t)=0$ as in \cite{Yan:Leng:Zhu:2016}.

Now, we discuss the differences between our model and those proposed by \cite{perry2013point}, \citet{kreib2019} and \citet{Alexander2021}.
In model \eqref{eq2.1}, if we multiply by an indicator function of the receiver set of sender $i$ and set $\beta_j(t)=0~(j\in[n])$ and $\gamma(t)=\gamma$,
then it reduces to \citeauthor{perry2013point}'s model. 
There are two major differences between model \eqref{eq2.1} and \citeauthor{perry2013point}'s model: 
(1) the regression coefficient $\gamma$ is a constant over time in \citeauthor{perry2013point}'s model while it is a unknown function $\gamma(t)$ in model \eqref{eq2.1}.
Our estimates of $\gamma(t)$ can be used for statistical inference on time changes in the effects of covariates;
(2) only out-degree  heterogeneity is taken into account in \citeauthor{perry2013point}'s model while both out- and in-degree heterogeneity are characterized in our model. 
The node popularity measured by in-degree parameter $\beta_j(t)$ is another important network feature 
\citep{Sengupta-Chen-2017}. To see how  model \eqref{eq2.1} captures this feature, we compute the log-ratio of
$u_i^{(1)}(0,t)$ to $u_j^{(1)}(0,t)$, where $u_i^{(1)}(0,t)$ denotes the first derivative of $u_i(0,t)=\sum_{k\neq i}^nE\{N_{ki}(0,t)\}$ with respect to $t$. 
Then we have
{\begin{eqnarray*}
		\log\bigg(\frac{u_{i}^{(1)}(0,t)}{u_{j}^{(1)}(0,t)}\bigg)
		=\beta_i(t)-\beta_j(t)+\log\bigg(\frac{\sum_{k\neq i}\exp\{\alpha_k(t)+Z_{ki}(t)^\top\gamma(t)\}}
		{\sum_{k\neq j}\exp\{\alpha_k(t)+Z_{kj}(t)^\top\gamma(t)\}}\bigg).
\end{eqnarray*}}
As we can see, node $i$ tends to be more popular in a network 
than node $j$ if $\beta_i(t)>\beta_j(t)$ at time $t$.
Under \citeauthor{perry2013point}'s model, we note that the ratio does not contain the difference $\beta_i(t)-\beta_j(t)$.
In addition, \cite{perry2013point} treated the baseline intensity function as a nuisance parameter and focused on estimating $\gamma$.

In  model \eqref{eq2.1}, if $\alpha_i(t)=\alpha(t)$, $i\in[n]$, and $\beta_j(t)=\beta(t)$, $j\in[n]$,
then after transforming the baseline function to $\lambda_{0,ij}(t)=\exp\{\alpha(t)+\beta(t)\}$,
it becomes the model proposed by
\citet{kreib2019} and \citet{Alexander2021}. 
One drawback of the model in \citet{kreib2019} and \citet{Alexander2021} is that
they set the same degree parameters for all nodes which neglects the effect of degree heterogeneity in real-world networks. 
In model \eqref{eq2.1}, our interest is on estimating not only $\gamma(t)$, but also the $2n$ node-specified parameters $\alpha_i(t)$ and $\beta_j(t)$.  
The dimension of parameters increases with the number of nodes in model \eqref{eq2.1}, 
which is more difficult than the case with fixed dimensional parameters in \citet{kreib2019} and \citet{Alexander2021}.   
It requires the development of new  methods for statistical inference; see Sections \ref{section-approach}-\ref{section:test} below. 

Finally, our model allows the network to have different edge densities. In particular, it allows the network to be sparse.
We call a dynamic network sparse if the ratio of the expected number of edges to $n^2$ over any time interval $(s,t]\subset(0,\tau)$ 
goes to zero as $n\rightarrow\infty.$
Consider the case that $\sup_{t\in[0,\tau]}[\alpha_i(t)+\beta_j(t)+Z_{ij}(t)^\top\gamma(t)]=-q_n,$ 
where $q_n$ is a positive constant depending on $n.$  
Then, by \eqref{eq2.2} and \eqref{eq2.3}, the expected number of edges over time $(s,t]$ is   
\begin{align}\label{eq2.4}
	\sum_{i=1}^n\sum_{j\neq i}\int_s^t\exp\{\alpha_i(u)+\beta_j(u)+Z_{ij}(u)^\top\gamma(u)\}du \leq (t-s)n^2e^{-q_n}.
\end{align}
Therefore, $q_n$ determines the sparsity level of the networks.
The larger  $q_n$ is, the sparser the network becomes.
If we set $q_n=c_0\log(n)$ with a constant $c_0\in (0,1]$,
then the order of the right hand side of \eqref{eq2.4} is $n^{2-c_0},$
which is less than $n^2$, and the dynamic network is sparse.

\section{Local estimating equation approach}
\label{section-approach}

Let  $\alpha(t)=(\alpha_1(t),\dots,\alpha_n(t))^\top$ and
$\beta(t)=(\beta_1(t),\dots,\beta_{n-1}(t))^\top$.
Define $\theta(t)=(\gamma(t)^\top,\alpha(t)^\top,$ $\beta(t)^\top)^\top$.
We use the superscript * to denote the true value (e.g., $\theta^*(t)$ is 
the true value of $\theta(t)$). 
For convenience, define  $N=n(n-1)$. 
Let $\mathcal{K}_h(u)=h^{-1}\mathcal{K}(u/h),$ where $\mathcal{K}(\cdot)$ is a kernel function and $h$ denotes the bandwidth.
Further, we define
\[
\mathcal{M}_{ij}(t; \alpha_i(t),\beta_j(t),\gamma(t))=N_{ij}(t)-\int_0^t\exp\{\alpha_i(s)+\beta_j(s)+Z_{ij}(s)^\top\gamma(s)\}ds.
\]
Under model \eqref{eq2.1}, $\mathcal{M}_{ij}(t)=\mathcal{M}_{ij}(t; \alpha_i^*(t),\beta_j^*(t),\gamma^*(t))$ is a zero-mean martingale process 
(see Lemma 2.3.2 in \citealp{FH1991}).
Assume that $\alpha_i^*(t),$ $\beta_j^*(t)$ and $\gamma^*(t)$ are sufficiently smooth in the sense that as $s\to t$,
\begin{align*}
	\max_{i\in[n]} |\alpha_i^*(s)- \alpha_i^*(t)| \to 0, \quad
	\max_{j\in[n-1]} |\beta_j^*(s) - \beta_j^*(t) | \to 0, \quad
	\max_{k\in[p]} |\gamma_k^*(s) - \gamma_k^*(t)| \to 0.
\end{align*}
When $s$ is close to $t,$ this leads to
\begin{align*}
	d\mathcal{M}_{ij}(s)=&dN_{ij}(s)-\exp\{\alpha_i(s)+\beta_j(s)+Z_{ij}(s)^\top\gamma(s)\}ds\\
	\approx & dN_{ij}(s)-\exp\{\alpha_i^*(t)+\beta_j^*(t)+Z_{ij}(s)^\top\gamma^*(t)\}ds,
\end{align*}
where $dN_{ij}(t)=\lim_{\Delta t\rightarrow 0}[N_{ij}((t+\Delta t)^-)-N_{ij}(t^{-})].$
For simplicity, let $d\mathcal{M}_{ij}(s,t)=dN_{ij}(s)-\exp\{\alpha_i^*(t)+\beta_j^*(t)+Z_{ij}(s)^\top\gamma^*(t)\}ds.$
Based on these facts, we propose the following local estimating equation
\begin{align}\label{eq2.5}
	(F(\theta(t))^\top, Q(\theta(t))^\top)^\top=0, 
\end{align}
where $F(\theta(t))=(F_1(\theta(t)),\dots,F_{2n-1}(\theta(t)))^\top$ with
\begin{align*}
	F_i(\theta(t))=&\frac{1}{n-1}\sum_{j\neq i}\int_{0}^{\tau}\mathcal{K}_{h_1}(s-t)d\mathcal{M}_{ij}(s,t),~~i\in[n],
	\\
	F_{n+j}(\theta(t))=&\frac{1}{n-1}\sum_{i\neq j}\int_{0}^{\tau}\mathcal{K}_{h_1}(s-t)d\mathcal{M}_{ij}(s,t),~~j\in[n-1],
\end{align*}
and
$$Q(\theta(t))=\frac{1}{N}\sum_{i=1}^n \sum_{j\neq i}
\int_{0}^{\tau}Z_{ij}(s)\mathcal{K}_{h_2}(s-t)d\mathcal{M}_{ij}(s,t).
$$
We estimate $\theta^*(t)$ by the solution to equation (\ref{eq2.5}), denoted by $\widehat{\theta}(t)=(\widehat{\gamma}(t)^\top,\widehat{\alpha}(t)^\top, \widehat{\beta}(t)^\top)^{\top}$.  

Now, we discuss the algorithm for solving \eqref{eq2.5}.
We adopt a combination of the fixed point iterative method and the Newton-Raphson method 
by alternatively  solving $F(\theta(t))=0$ and $Q(\theta(t))=0$. 
This is implemented in Algorithm \ref{algorithm-a}, where {\bf Step 1} is about
solving $F(\theta(t))=0$ with a given $\gamma(t)$ via the fixed point iterative method,
and {\bf Step 2} is about solving $Q(\theta(t))=0$ with given $\alpha(t)$ and $\beta(t)$
via the Newton-Raphson method. 
The stopping criterion in {\bf Step 3} is 
\begin{equation}
	\label{stop_cri}
	\epsilon^{(k)}=\max_{i\in[n]}|\alpha_i^{[k]}(t)-\alpha_i^{[k-1]}(t)|+
	\max_{i\in[n-1]}|\beta_i^{[k]}(t)-\beta_i^{[k-1]}(t)|+
	\max_{j\in[p]}|\gamma_j^{[k]}-\gamma_j^{[k-1]}|\le 10^{-3},
\end{equation}
which has good performances in simulation studies and real data analyses in Section \ref{section:simulation}.

\begin{algorithm}
	\caption{An algorithm for solving local estimating equations}
	\label{algorithm-a}
	{\bf Input}: $k=0,~\alpha_i^{[0]}(t)=0,~\beta_j^{[0]}(t)=0,~\gamma^{(0)}(t)=0$ and $\epsilon^{(0)}=1$\;
	\While{$\epsilon^{(k)}>10^{-3}$}
	{Set $k\leftarrow k+1$\;
		{\bf Step 1}: Calculate $\alpha_i^{[k]}(t)$ and $\beta_j^{[k]}(t)$ by
		\begin{align*}
			\alpha_i^{[k]}(t)=&\log\Bigg\{\frac{\sum_{j\neq i}\int_0^{\tau}\mathcal{K}_{h_1}(s-t)dN_{ij}(t)}
			{\sum_{j\neq i}\int_0^{\tau}\mathcal{K}_{h_1}(s-t)\exp\{\beta_j^{[k-1]}(t)+Z_{ij}(s)^\top\gamma^{[k-1]}(t)\}ds}\Bigg\},\\
			\beta_j^{[k]}(t)=&\log\Bigg\{\frac{\sum_{i\neq j}\int_0^{\tau}\mathcal{K}_{h_1}(s-t)dN_{ij}(t)}
			{\sum_{i\neq j}\int_0^{\tau}\mathcal{K}_{h_1}(s-t)\exp\{\alpha_i^{[k-1]}(t)+Z_{ij}(s)^\top\gamma^{[k-1]}(t)\}ds}\Bigg\}.
		\end{align*}
		
		{\bf Step 2}: Update $\gamma^{[k]}(t)$ by the solution to
		\[
		\frac{1}{N}\sum_{i=1}^n \sum_{j\neq i}
		\int_{0}^{\tau}Z_{ij}(s)\mathcal{K}_{h_2}(s-t)\big[dN_{ij}(s)- \exp\big\{\alpha_i^{[k-1]}(t)+\beta_j^{[k-1]}(t)+Z_{ij}(s)^\top\gamma(t)\big\}ds\big]=0.
		\]
		
		{\bf Step 3}: Calculate $\epsilon^{(k)}$ according to (\ref{stop_cri})
	}
	
	{\bf Output}: $\widehat\alpha_i(t)$, $\widehat\beta_j(t)$ and $\widehat\gamma(t).$
\end{algorithm}

\section{Theoretical properties}
\label{section:asymptotic}

In this section, we present consistency and asymptotic normality of the estimator.
To obtain consistency for $\widehat{\theta}(t)$, we adopt a two-stage procedure.
Let $\eta(t)=(\alpha(t)^\top,\beta(t)^\top)^\top$
and $\widehat{\eta}_{\gamma}(t)$ be the estimator obtained by solving $F(\theta(t))=0$ with a given $\gamma(t)$. Define $\widehat Q_c(\gamma(t))$
as the profiled function of $Q(\theta(t))$ obtained by replacing $\eta(t)$
with $\widehat{\eta}_\gamma(t)$. It is clear that $\widehat Q_c(\widehat{\gamma}(t))=0$.
In the first stage, we establish the existence of $\widehat{\eta}_{\gamma}(t)$
and derive its consistency rate uniformly in $t\in[a,b]$ (see Lemma 6 in the supplementary material).
In the second stage, we derive the upper bound of the error between $\widehat{\gamma}(t)$ and $\gamma^*(t)$ by using the profiled function $\widehat Q_c(\gamma(t))$.

For convenience, define  $\pi_{ij}(t)=\alpha_i(t)+\beta_j(t)+Z_{ij}(t)^\top\gamma(t)$ and
$\pi_{ij}^*(t)=\alpha_i^*(t)+\beta_j^*(t)+Z_{ij}(t)^\top\gamma^*(t)$.
Let $\mathbb{V}(t;\eta(t),\gamma(t))$ denote the Hessian  matrix of \eqref{eq2.5} at $\theta(t)$,
and write
\[
\mathbb{V}(t;\eta(t),\gamma(t))=
\begin{pmatrix}
	V_{\eta,\eta}(t,\gamma(t)), & V_{\eta,\gamma}(t)\\
	V_{\gamma,\eta}(t), & V_{\gamma,\gamma}(t)
\end{pmatrix}.
\]
Furthermore, define $V^*(t)=-\mathbb{E}\{V_{\eta^*,\eta^*}(t,\gamma^*(t))\}$
and $v_{i,j}^*(t)$ be the $(i,j)$th element of $V^*(t),$
where
\begin{align*}
	v_{i,i}^*(t)=&(n-1)^{-1}\sum_{j\neq i}\mathbb{E}\{e^{\pi_{ij}^*(t)}\},\quad~ i\in[n],\\
	v_{n+j,n+j}^*(t)=&(n-1)^{-1}\sum_{i\neq j}\mathbb{E}\{e^{\pi_{ij}^*(t)}\},\quad~ j\in[n-1],\\
	v_{i,n+j}^*(t)=&v_{n+j,i}(t)=(n-1)^{-1}\mathbb{E}\{e^{\pi_{ij}^*(t)}\},~i\in[n],~j\in[n-1],~i\neq j,
\end{align*}
and $v_{i,j}^*(t)=v_{j,i}^*(t)=0$ otherwise.
Let $v_{2n,2n}^*(t)=\sum_{i=1}^nv_{ii}^*(t)-\sum_{i=1}^n\sum_{j=1,j\neq i}^{2n-1}v_{ij}^*(t).$
Then, define $S^*(t)=(s_{ij}^*(t))\in \mathbb{R}^{(2n-1)\times (2n-1)},$
where
\begin{align*}
	s_{ij}^*(t)=
	\begin{cases}
		\frac{\delta_{ij}}{v_{i,i}^*(t)} + \frac{1}{v_{2n,2n}^*(t)}, & i, j\in [n]~~
		\text{or}~~ i,j=[n-1]_n,\\
		\frac{1}{v_{2n,2n}^*(t)}, &~~~\text{otherwise}.
	\end{cases}
\end{align*}
In the above equation, $\delta_{ij}$ is an indicator function, i.e., 
$\delta_{ij}=1$ if $i=j$ and $\delta_{ij}=0$ otherwise.
Let $\|a\|_{\infty}=\max_{1\le i\le n}|a_i|$ denote the $\ell_{\infty}$-norm for any vector $a=(a_1,\dots,a_n)^\top.$ 

Before presenting consistency and asymptotic normality of the estimator,
we introduce the following conditions.

\begin{condition}\label{condition:CB1}
	$\max_{i,j} \|Z_{ij}(t) \|_\infty\le \kappa_n$ almost surely,
	where $\kappa_n$ could diverge with $n$.
\end{condition}

\begin{condition}\label{condition:PB2}
	The parameter functions $\alpha_i^*(t),~\beta_j^*(t)$ and $\gamma^*(t)$ are twice continuously differentiable.
	In addition, each element of $\gamma^{(l)}(t)~(l=0,1,2)$ is a bounded function
	and
	\begin{align*}
		\theta^*(t)\in\mathcal{B}=\bigg\{\theta(t):
		\max_{i,j}\sup_{t\in[a,b]}\big|\alpha_i^{(l)}(t)+\beta_j^{(l)}(t)+Z_{ij}(t)^\top\gamma^{(l)}(t)\big|\le q_n~(l=0,1,2)\bigg\},
	\end{align*}
	where $q_n$ could diverge with $n$.
\end{condition}

\begin{condition}\label{condition:H3}
	$H_Q(\gamma(t))=\lim_{n\rightarrow\infty}\mathbb{E}\{V_{\gamma,\gamma}(t)-V_{\gamma,\eta^*}(t)
	V_{\eta^*,\eta^*}(t,\gamma(t))^{-1}V_{\eta^*,\gamma}(t)\}$
	is continuous in a neighbourhood of $\gamma^*(t)$ and
	\begin{align*}
		\inf_{t\in[a,b]}\rho_{\min}\big\{H_Q(\gamma^*(t))\big\}> \varsigma>0,
	\end{align*}
	where $\rho_{\min}(M)$ denotes the smallest eigenvalue of the matrix $M,$
	and $\varsigma$ is some constant.
\end{condition}

\begin{condition}\label{condition:kernel4}
	$\mathcal{K}(x)$ is a symmetric density function with support $[-1,1].$
\end{condition}

\begin{condition}\label{condition:bandwidth5}
	$nh_1^2\rightarrow\infty,$ $nh_2\rightarrow\infty,$ $ne^{2q_n}h_1^5\rightarrow 0$
	and $ne^{q_n}h_2^{5/2}\rightarrow 0$ as $n\rightarrow \infty.$
	Moreover, $h_2=O(h_1^2).$
\end{condition}

Condition \ref{condition:CB1} assumes that covariates are bounded above by $\kappa_n$ uniformly.
If $Z_{ij}(t)$ is a binary predictive variable,  Condition \ref{condition:CB1} automatically holds.
If $Z_{ij}(t)$'s are generated from normal distributions with variances bounded above by a constant, 
Condition \ref{condition:CB1} still holds with $\kappa_n=O( \log n)$.
Condition \ref{condition:PB2} requires that parameter functions are sufficiently smooth 
and the sum of parameter functions and the sum of their first and second derivatives are bounded above by $q_n$.
Condition \ref{condition:H3} is the assumption to ensure the identifiability of $\gamma^*(t)$.
Condition \ref{condition:kernel4} is a standard assumption in nonparametric statistics.
The kernel function affects the convergence rate of the estimators only by multiplicative constants and thus has little impact on the rate of convergence 
(e.g., \citealp{FG1996}). 
In Condition \ref{condition:bandwidth5}, 
$h_1$ and $h_2$ are the bandwidths for estimating $(\alpha(t), \beta(t))$ and $\gamma(t)$ in \eqref{eq2.5}
respectively. 
The orders of the bandwidths $h_1$ and $h_2$ are determined by the sample size $n$ and the sparsity parameter $q_n.$
If $q_n$ is an absolute constant,  the bandwidths can be chosen as $h_1=o(n^{-1/5})$ and $h_2=o(n^{-2/5}).$

\subsection{Consistency and asymptotic normality}
We first present consistency of $\widehat\theta(t)$, whose proof is given in the supplementary material.

\begin{theorem}\label{theorem:consistency}
	Suppose that Conditions \ref{condition:CB1}-\ref{condition:bandwidth5} hold.
	If $(q_n+1)e^{19q_n} \kappa_n^3\sqrt{\log(nh_1)/(nh_1)}\rightarrow 0$
	as $n\rightarrow \infty$,
	then the estimator $\widehat\theta(t)=(\widehat\eta(t)^\top,\widehat\gamma(t)^\top)^\top$, the solution to equation \eqref{eq2.2}, exists
	and satisfies
	\begin{align}
		&\sup_{t\in[a,b]}\|\widehat{\eta}(t)-\eta^*(t)\|_\infty=O_p\left((q_n+1)e^{13q_n}\kappa_n^3\sqrt{\frac{\log nh_1}{n h_1}}\right), \label{eq2.6}\\
		&\sup_{t\in[a,b]}\|\widehat{\gamma}(t)-\gamma^*(t)\|_\infty=O_p\left((q_n+1) e^{7q_n} \kappa_n^2 \sqrt{\frac{\log nh_1}{nh_1}}+\frac{1}{n^2h_2} \right). \label{eq2.7}
	\end{align}
\end{theorem}

\begin{remark}\label{Rem1}
	The condition in Theorem \ref{theorem:consistency} implies that
	$q_n$ can be chosen as $c\log(nh_1)$ with $c\in(0,1/40)$.
	Therefore, the term on the right-hand side of \eqref{eq2.4} 
	is of order $n^{\tilde c}$ with $\tilde c\in(79/40,2)$,
	whose ratio to $n^2$ tends to zero.
	The condition on $q_n$ in Theorem \ref{theorem:consistency} appears stronger than what is needed to 
	guarantee the right-hand side of \eqref{eq2.6} and \eqref{eq2.7} to go to zero.
	The reason is that it establishes not only the uniform consistency, but also the existence of the solution to \eqref{eq2.5}.  
	The existence of the solution requires a more stringent condition on $q_n$ while this point is not explicitly reflected in \eqref{eq2.6} and \eqref{eq2.7}.
	This phenomenon also exists in other works (e.g., \citealp{Chatterjee:Diaconis:Sly:2011}; \citealp{yan2016asymptotics}). 
\end{remark}

\begin{remark}
	Indeed the uniform convergence rate in Theorem \ref{theorem:consistency} has the familiar bias-variance trade-off
	in the kernel smoothing literature (e.g., \citealp{FG1996}).
	Specifically, in the proof of Theorem \ref{theorem:consistency}, 
	the bias terms for $\widehat\eta(t)$
	and $\widehat\gamma(t)$ are of order 
	$O(e^{q_n}h_1^2)$ and $O(e^{q_n}h_2^2),$ respectively.
	This fact suggests that the bandwidths should be carefully selected  to balance the bias and variance. 
	When Condition \ref{condition:bandwidth5} holds, 
	the biases in Theorem \ref{theorem:consistency} are dominated uniformly by $O_p\big((q_n+1) e^{13q_n} \kappa_n^3 \sqrt{\log( nh_1)/(nh_1)}\big)$ and $O_p\big((q_n+1) e^{7q_n} \kappa_n^2 \sqrt{\log( nh_1)/(nh_1)}+1/(n^2h_2)\big).$
\end{remark}

Define $\mu_{0}=\int \mathcal{K}^2(u)du,$ and write $H_Q(t)=H_Q(\gamma^*(t)).$ Next, we present the asymptotic normality of $\widehat\gamma(t)$.

\begin{theorem}
	\label{theorem-central-gamma}
	Suppose that Conditions \ref{condition:CB1}-\ref{condition:bandwidth5} hold.
	If $(q_n+1)^{3/2}e^{15q_n}\kappa_n^2[\log(nh_1)]^{3/4}/(nh_1)^{1/4}\rightarrow 0$,
	then $(Nh_2)^{1/2}\Psi(t)^{-1/2}\{\widehat{\gamma}(t)-\gamma^*(t)- \{H_Q(t)\}^{-1}b_*(t)\}$ converges in distribution to a $p$-dimensional standard normal distribution,
	where $\Psi(t)=\{H_Q(t)\}^{-1}\Sigma(t) \{H_Q(t)\}^{-1}$.
	Here, $\Sigma(t)$ and $b_*(t)$ are
	\begin{align*}
		\Sigma(t)=&\frac{\mu_0}{N}\sum_{i=1}^n\sum_{j=1,j\neq i}^n\mathbb{E}\bigg[\Big(Z_{ij}(t)
		-V_{\gamma^*\eta^*}(t)S^*(t)\iota_{ij}\Big)^{\otimes 2}e^{\pi_{ij}^*(t)}\bigg],\\
		b_*(t)=&\frac{\mu_{0}}{2Nh_1}\bigg[ \sum_{i=1}^n \frac{\sum_{j\neq i}\mathbb{E}(Z_{ij}(t) \exp\{\pi_{ij}^*(t)\})}
		{\sum_{j\neq i} \mathbb{E}(\exp\{\pi_{ij}^*(t)\})}+\sum_{j=1}^n \frac{  \sum_{i\neq j}\mathbb{E}(Z_{ij}(t) \exp\{\pi_{ij}^*(t)\})}
		{\sum_{i\neq j} \mathbb{E}(\exp\{\pi_{ij}^*(t)\})}\bigg],
	\end{align*}
	where $a^{\otimes2}=aa^\top$ for any vector $a,$
	and $\iota_{ij}$ is a $(2n-1)$-dimensional vector with the $i$th and $(n+j)$th elements being one
	and others being zero.
\end{theorem}

\begin{remark}
	The limiting distribution of $\widehat\gamma(t)$ involves
	a bias term $\{H_Q(t)\}^{-1}b_*(t).$
	This is referred to as the so-called incidental parameter problem in econometric literature \citep{NS1948, FW2016, Dzemski2017}. 
	This phenomenon also appears in the network literature 
	\citep{Graham2017, yan2019}.
	The 
	bias is due to 
	the appearance of the estimator $\widehat\eta(t)$
	in the profiled function  $Q(\theta),$
	and the dimension of $\widehat\eta(t)$ diverges as $n\rightarrow \infty.$
	If $q_n$ is an absolute constant,
	then $\{H_Q(t)\}^{-1}b_*(t)=O(1/(nh_1))$, which is asymptotically negligible as $nh_1\rightarrow \infty$.
\end{remark}

The asymptotic normality of $\widehat\eta(t)$ is presented in the following theorem.

\begin{theorem}\label{theorem-central-degree}
	Under Conditions \ref{condition:CB1}-\ref{condition:bandwidth5}, if
	$(q_n+1)\kappa_n^3 e^{23q_n}(\log nh_1)^{1/2}/(nh_1)^{1/4}\rightarrow 0$
	as $n\rightarrow \infty$,
	then for any fixed positive integer $k$,
	$(nh_1)^{1/2}\big\{(\mu_{0}S^*(t))^{-1/2}[\widehat{\eta}(t)- \eta^*(t)]\big\}_{1:k}$
	converges in distribution to a $k$-dimensional standard normal distribution.
\end{theorem}
\begin{remark}
	By Theorem \ref{theorem-central-degree}, 
	the covariance matrix of $(nh_1)^{1/2}[\widehat{\eta}(t)- \eta^*(t)]_{1:k}$ is  given by the upper left $k\times k$ block of $\mu_0S^*(t).$
	In addition, for any fixed $i,$ as $nh_1\rightarrow\infty$, the
	convergence rate of $\widehat\eta_i(t)$ is $O((nh_1)^{-1/2}e^{q_n/2}).$
\end{remark}


\subsection{Confidence intervals for $\eta^*(t)$}

We next construct the pointwise confidence interval for $\eta^*(t)$. 
Since the asymptotic covariance matrix for $\eta^*(t)$ involves
unknown 
$S^*(t)$ for approximating $[V^*(t)]^{-1}$, we use  $\widehat{S}(t)=(\widehat{s}_{ij}(t))_{i,j\in [2n-1]}$
to estimate it, where the unknown parameters $\eta^*(t)$ and $\gamma^*(t)$ are replaced by their respective estimators
$\widehat{\eta}(t)$ and $\widehat{\gamma}(t)$, i.e.,
\[
\widehat{s}_{ij}(t)= 
\begin{cases}
	\frac{\delta_{ij}}{\widehat v_{i,i}(t)} + \frac{1}{\widehat v_{2n,2n}(t)},  & i, j\in [n],
	\text{or}~~ i,j\in [n-1]_n, \\
	-\frac{1}{\widehat v_{2n,2n}(t)}, & ~~~\text{otherwise}.
\end{cases}
\]
In the above equation,  $\widehat{v}_{2n,2n}=\sum_{i=1}^n \widehat{v}_{ii}(t)-\sum_{i=1}^n\sum_{j=1,j\neq i}^{2n-1} \widehat{v}_{ij}(t)$.

By Theorem \ref{theorem-central-degree}, the distribution of $(nh_1)^{1/2}\big\{\widehat{\eta}(t)-\eta^*(t)\big\}$ is asymptotically  equivalent to
\[
S^*(t)\left(\frac{h_1}{n}\right)^{1/2} \int_{0}^{\tau}\mathcal{K}_{h_1}(s-t)d\widetilde{\mathcal{M}}(s),
\]
where $\widetilde{\mathcal{M}}(t)=(\widetilde{M}_{1}(s),\dots,\widetilde{M}_{2n-1}(t))^\top$ with
\begin{align*}
	\widetilde M_{i}(t)=\sum_{k\ne i} \mathcal{M}_{ik}(t)~ (i\in[n])~~
	\text{and}~~\widetilde M_{n+i}(t)=\sum_{k\ne i} \mathcal{M}_{ki}(t)~(i\in[n-1]).
\end{align*}
Note that $\widetilde M_{i}(t) $ is the sum of local square-integerable martingales.
Therefore, by the martingale properties, we can estimate the covariance $\Omega(t)$ of $\sqrt{h_1/n}\int_{0}^{\tau}\mathcal{K}_{h_1}(s-t)d\widetilde{\mathcal{M}}(s)$
by $\widehat \Omega(t)=(\widehat\omega_{ij}(t): i,j\in[2n-1])$ with
\begin{align*}
	\widehat\omega_{ii}(t)=&\frac{h_1}{n}\sum_{j\neq i}\int_0^{\tau}\mathcal{K}_{h_1}^2(s-t)dN_{ij}(s),~i\in[n],\\
	\widehat\omega_{n+j,n+j}(t)=&\frac{h_1}{n}\sum_{i\neq j}\int_0^{\tau}\mathcal{K}_{h_1}^2(s-t)dN_{ij}(s),~j\in[n-1],\\
	\widehat\omega_{i,n+j}(t)=\widehat\omega_{n+j,i}(t)=&\frac{h_1}{n}\int_0^{\tau}\mathcal{K}_{h_1}^2(s-t)dN_{ij}(s), \qquad i\in[n],~j\in[n-1],\\
	\widehat\omega_{ij}(t)=\widehat\omega_{ji}(t)=&0,\quad\quad\text{otherwise}.
\end{align*}
Thus, we estimate the variance of $\sqrt{nh_1}\{\widehat\eta_i(t)-\eta_i^*(t)\}$
by the $i$th diagonal element $\widehat\sigma_{ii}(t)$ of $\widehat S(t)\widehat \Omega(t)\widehat S(t)$. 

By arguments similar to Lemma 7 in the supplementary material, we can show
\begin{align*}
	\|\widehat S(t)\widehat \Omega(t)\widehat S(t)-\mu_{0}S^*(t)\|_{\max}=o_p(1),
\end{align*}
where $\|M\|_{\max}=\max_{i,j} |m_{i,j}|$ 
denotes the maximum absolute entry-wise norm for any matrix $M=(a_{i,j}).$
Let $z_{\alpha}$ be the $100(1-\alpha)$th percentile of the standard normal distribution.
Then the $(1-\alpha)$-confidence interval  for $\eta_i^*(t)$ is given by
\begin{align*}
	&\widehat \eta_i(t)\pm (nh_1)^{-1/2}z_{\alpha/2}\widehat\sigma_{ii}^{1/2}(t),~~i\in[2n-1].
\end{align*}

\begin{remark}
	Since the dimension of $\eta^*(t)$ diverges with the sample size $n,$ directly using $\mu_0\widehat S(t)$ to estimate $\mu_0S^*(t)$ may result in a large bias for local smoothing estimators with finite sample.
	Therefore, instead of $\mu_0\widehat S(t),$
	we consider a sandwich-type estimator $\widehat S(t)\widehat \Omega(t)\widehat S(t)$ for
	the variance of $\sqrt{nh_1}\{\widehat\eta_i(t)-\eta_i^*(t)\}.$
	This is different from the covariance estimator developed by \cite{yan2019} for static networks, where the maximum likelihood estimator was used.
\end{remark}

\subsection{Confidence intervals for $\gamma^*(t)$}
Based on Theorem \ref{theorem-central-gamma}, $\widehat\gamma(t)$ has a non-negligible bias term $H_Q(t)^{-1}b_*(t).$
Therefore, bias-correction is necessary. For this, define
\begin{align*}
	\widehat b(t)=&\frac{h_1}{2N}\bigg[\sum_{i=1}^n\frac{\sum_{j\neq i}Z_{ij}(t)\int_0^\tau \mathcal{K}_{h_1}^2(s-t)dN_{ij}(s)}
	{\sum_{j\neq i} \exp\{\widehat\pi_{ij}(t)\}}+\sum_{j=1}^n\frac{\sum_{i\neq j}Z_{ij}(t)\int_0^\tau \mathcal{K}_{h_1}^2(s-t)dN_{ij}(s)}
	{\sum_{i\neq j} \exp\{\widehat\pi_{ij}(t)\}}\bigg],\\
	\widehat H_Q(t)=&\frac{1}{N}\sum_{i=1}^n\sum_{j\neq i}Z_{ij}(t)^{\otimes2}e^{\widehat\pi_{ij}(t)}
	-\widehat  V_{\hat\gamma,\hat\eta}(t)\widehat S(t)\widehat V_{\hat\eta,\hat\gamma}(t),
\end{align*}
where $\widehat V_{\hat \gamma,\hat\eta}(t)=N^{-1}(\widehat u_{1}(t),\dots,\widehat u_{2n-1}(t))$
and $\widehat  V_{\hat\eta,\hat\gamma}(t)=n\widehat V_{\hat \gamma,\hat\eta}(t)^\top$
with
$$
\widehat u_{i}(t)=\sum_{j\neq i}Z_{ij}(t)e^{\widehat\pi_{ij}(t)},~i\in[n]~~ \text{and}~~
\widehat u_{n+j}(t)=\sum_{i\neq j}Z_{ij}(t)e^{\widehat\pi_{ij}(t)},~j\in[n-1].
$$
In addition, by the martingale properties, $\Sigma(t)$ can be estimated  by
\begin{align*}
	\widehat\Sigma(t)=&\frac{h_2}{N}\sum_{i=1}^n\sum_{j\neq i}\Big(Z_{ij}(t)-\widehat V_{\hat\gamma,\hat\eta}(t)\widehat S(t)\iota_{ij}\Big)^{\otimes 2}\int_0^{\tau}\mathcal{K}_{h_2}^2(u-t)dN_{ij}(u).
\end{align*}
Finally, we estimate the bias $\{H_Q(t)\}^{-1}b_*(t)$ by $\{\widehat H_Q(t)\}^{-1}\widehat b(t)$
and the covariance $\Psi(t)$
by $\widehat\Psi(t)=\{\widehat H_Q(t)\}^{-1}\widehat\Sigma(t)\{\widehat H_Q(t)\}^{-1}$.

By arguments similar to Lemma 7 in the supplementary material, the absolute entry-wise error tends to zero with probability, that is,
\begin{align*}
	\|\{\widehat H_Q(t)\}^{-1}\widehat b(t)-\{H_Q(t)\}^{-1} b_*(t)\|_{\infty}=o_p(1)~~~ \text{and}~~~
	\|\widehat \Psi(t)-\Psi(t)\|_{\max}=o_p(1).
\end{align*}
Let $\widetilde b_j(t)$ be the $j$th element of $\{\widehat H_Q(t)\}^{-1}\widehat b(t)$,
and $\widehat \psi_{jj}(t)$ be the $j$th diagonal element of $\widehat \Psi(t).$
Then the $(1-\alpha)$-confidence interval  for $\gamma_j^*(t)$ is given by
$$
\widehat \gamma_j(t)-\widetilde b_j(t)\pm (Nh_2)^{-1/2}z_{\alpha/2}\widehat \psi_{jj}^{1/2}(t),~~j\in[p].
$$

\section{Hypothesis testing}\label{section:test}
\subsection{Tests for Trend}\label{sec:Trt}
\cite{perry2013point} and \cite{Sit2020} assumed that $\gamma^*(t)$ is constant over time, while \cite{kreib2019}, 
\cite{Alexander2021} and our proposed model \eqref{eq2.1} assume that $\gamma^*(t)$ is time-varying.
Whether the effects of covariates on interactions change with time is usually unknown.
If both $\eta^*(t)$ and $\gamma^*(t)$ are time-invariant, the network may be static.
Therefore, it is of interest to test if $\eta^*(t)$ and $\gamma^*(t)$ have time-varying trends.
We call this the trend testing problem.
In this section, we consider testing the following hypotheses:
\begin{align*}
	& H_{0\eta}: \eta^*(t)=\eta^*~~ \text{for~all}~~t\in [a,b]\\
	\text{versus } & H_{1\eta}: \eta^*(t)\neq \eta^*~~\text{for some} \ t\in [a,b],
\end{align*}
and
\begin{align*}
	& H_{0\gamma}: \gamma^*(t)=\gamma^*~~ \text{for~all}~~ t\in [a,b]\\
	\text{ versus } & H_{1\gamma}: \gamma^*(t)\neq \gamma^*~~\text{for some}\ t\in [a,b],
\end{align*}
where $\eta^*$ and $\gamma^*$ are some unspecified vectors.
We see that
\begin{itemize}
	\item if either of $H_{0\eta}$ and $H_{0\gamma}$ holds, then model \eqref{eq2.1} is a semi-parametric model;
	\item if both $H_{0\eta}$ and $H_{0\gamma}$ hold, then model \eqref{eq2.1} becomes a completely parametric model.
\end{itemize}
\cite{Alexander2021} proposed a test statistic which compares the completely parametric
and the non-parametric estimator 
using the $\ell_2$-distance
to test $H_{0\eta}$ and $H_{0\gamma}.$
However, since the dimension of parameters grows with the nodes under model \eqref{eq2.1},
the method developed for fixed dimension 
in \cite{Alexander2021} can not be directly applied here.
We consider the following test statistics to test $H_{0\eta}$ and $H_{0\gamma},$ respectively:
\begin{align*}
	\mathcal{T}_{\eta}&=\max_{i\in[2n-1]}\sup_{a\le t_1<t_2\le b}\sqrt{nh_1}\big|\widehat{\eta}_i(t_1)-\widehat{\eta}_i(t_2)\big|/\widehat\vartheta_{i,\eta}^{1/2}(t_1,t_2),\\
	\mathcal{T}_{\gamma}&=\max_{j\in[p]}\sup_{a\le t_1<t_2\le b}\sqrt{Nh_2}\big|\widehat{\gamma}_j(t_1)-\widetilde{b}_j(t_1)-\widehat{\gamma}_j(t_2)+\widetilde b_{j}(t_2)\big|/\widehat\vartheta_{j,\gamma}^{1/2}(t_1,t_2),
\end{align*}
where $\widehat\vartheta_{i,\eta}(t_1,t_2)=\widehat\sigma_{ii}(t_1)+\widehat\sigma_{ii}(t_2)$
and $\widehat\vartheta_{j,\gamma}(t_1,t_2)=\widehat \psi_{jj}(t_1)+\widehat \psi_{jj}(t_2).$

The test statistics $\mathcal{T}_{\eta}$ and $\mathcal{T}_\gamma$ are close to zero under the nulls
$H_{0\eta}$ and $H_{0\gamma},$ respectively.
Hence we will reject $H_{0\eta}$ if $\mathcal{T}_\eta>c_\eta(\nu),$
and reject $H_{0\gamma}$ if $\mathcal{T}_\gamma>c_\gamma(\nu),$
where $c_\eta(\nu)$ and $c_\gamma(\nu)$ are the critical values.
To obtain these critical values,
we consider a resampling approach.
Using arguments similar to the proof of Theorem \ref{theorem-central-degree},
we can show that under the null $H_{0\eta},$
the distribution of $\sqrt{nh_1}[\widehat{\eta}(t_1)-\widehat{\eta}(t_2)]$ is asymptotically equivalent to
\begin{align*}
	\sqrt{\frac{h_1}{n}}\bigg[\widehat S(t_1)\int_0^{\tau} \mathcal{K}_{h_1}(u-t_1)d\widetilde{\mathcal{M}}(u)
	-\widehat S(t_2)\int_0^{\tau} \mathcal{K}_{h_1}(u-t_2)d\widetilde{\mathcal{M}}(u)\bigg],
\end{align*}
and  $\sqrt{nh_1}[\widehat{\eta}(t_1)-\eta^*] $ and $\sqrt{nh_1}[\widehat{\eta}(t_2)-\eta^*]$
are asymptotically independent with $t_1\neq t_2.$
In addition, using arguments similar to the proof of Theorem \ref{theorem-central-gamma},
we have that under the null $H_{0\gamma},$
the distribution of $\sqrt{Nh_2}[\widehat{\gamma}(t_1)-\widehat{b}(t_1)-\widehat{\gamma}(t_2)+\widehat b(t_2)]$ is asymptotically equivalent to
\begin{align*}
	&\sqrt{\frac{h_2}{N}}\sum_{i=1}^n\sum_{j\neq i} \bigg[\Big(Z_{ij}(t_1)-\widehat V_{\hat\gamma,\hat\eta}(t_1)\widehat S(t_1)\iota_{is}\Big)\int_0^{\tau}\mathcal{K}_{h_2}(u-t_1)d\mathcal{M}_{ij}(u)\\
	& \hspace{1.2in}-\Big(Z_{ij}(t_2)-\widehat V_{\hat\gamma,\hat\eta}(t_2)\widehat S(t_2)\iota_{ij}\Big)\int_0^{\tau}\mathcal{K}_{h_2}(u-t_2)d\mathcal{M}_{ij}(u)\bigg],
\end{align*}
and $\sqrt{Nh_2}[\widehat{\gamma}(t_1)-\widehat{b}(t_1)-\gamma^*]$ and $\sqrt{Nh_2}[\widehat{\gamma}(t_2)-\widehat b(t_2)-\gamma^*]$ are asymptotically independent with $t_1\neq t_2.$
A direct calculation yields that  the variance function of $\mathcal{M}_{ij}(u)$ is $E\{N_{ij}(u)\}$
(see Theorem 2.5.3 in \citealp{FH1991}).
Motivated by the work of \cite{LFW1994},
we replace $\mathcal{M}_{ij}(u)$ with $N_{ij}(u)G_{ij},$ that is,
\begin{align*}
	\widetilde{T}_{\eta}(t_1,t_2)=&\sqrt{\frac{h_1}{n}}\bigg[\widehat S(t_1)\int_0^{\tau} \mathcal{K}_{h_1}(u-t_1)d\widetilde{\mathcal{N}}(u)
	-\widehat S(t_2)\int_0^{\tau} \mathcal{K}_{h_1}(u-t_2)d\widetilde{\mathcal{N}}(u)\bigg],\\
	\widetilde{T}_{\gamma}(t_1,t_2)=&\sqrt{\frac{h_2}{N}}\sum_{i=1}^n\sum_{j\neq i} \bigg[\widehat H_Q^{-1}(t_1)\Big(Z_{ij}(t_1)-\widehat V_{\hat\gamma,\hat\eta}(t_1)\widehat S(t_1)\iota_{ij}\Big)\int_0^{\tau}\mathcal{K}_{h_2}(u-t_1)dN_{ij}(u)G_{ij}\\
	& \hspace{1in}-\widehat H_Q^{-1}(t_2)\Big(Z_{ij}(t_2)-\widehat V_{\hat\gamma,\hat\eta}(t_2)\widehat S(t_2)\iota_{ij}\Big)\int_0^{\tau}\mathcal{K}_{h_2}(u-t_2)dN_{ij}(u)G_{ij}\bigg],
\end{align*}
where $\widetilde{\mathcal{N}}(u)=(\widetilde N_1(u),\dots,\widetilde N_{2n-1}(u))^\top$ with
\begin{align*}
	\widetilde N_i(u)=\sum_{k\ne i} N_{ik}(u)G_{ik}~ (i\in[n])~~
	\text{and}~~\widetilde N_{n+i}(u)=\sum_{k\ne i} N_{ki}(u)G_{ki}~(i=[n-1]),
\end{align*}
and $G_{ij}\ (i\neq j\in[n])$ are independent standard normal variables which are independent of
the observed data and $G_{ii}=0.$
Define $\widehat\vartheta_{\eta}(t_1,t_2)=\text{diag}\{\widehat\vartheta_{1,\eta}(t_1,t_2),\dots,\widehat\vartheta_{2n-1,\eta}(t_1,t_2)\}$
and $\widehat\vartheta_{\gamma}(t_1,t_2)=\text{diag}\{\widehat\vartheta_{1,\gamma}(t_1,t_2),\dots,\widehat\vartheta_{p,\gamma}(t_1,t_2)\},$
where $\text{diag}\{a_1,\dots,a_n\}\in\mathbb{R}^{n\times n}$ denotes a diagonal matrix
with $a_i$ as its $i$th diagonal element.
By repeatedly generating the normal random sample $G_{ij},$
the distribution of $\mathcal{T}_{\eta}$ and $\mathcal{T}_{\gamma}$
can be respectively approximated by the conditional distributions of $\widetilde{\mathcal{T}}_{\eta}$ and $\widetilde{\mathcal{T}}_{\gamma}$ given the observed data, where
\begin{align*}
	\widetilde{\mathcal{T}}_{\eta}=&\sup_{a\le t_1<t_2\le b} \|\widehat\vartheta_{\eta}^{-1/2}(t_1,t_2)\widetilde{T}_{\eta}(t_1,t_2)\|_{\infty},\\
	\widetilde{\mathcal{T}}_{\gamma}=&\sup_{a\le t_1<t_2\le b} \|\widehat\vartheta_{\gamma}^{-1/2}(t_1,t_2)\widetilde{T}_{\gamma}(t_1,t_2)\|_{\infty}.
\end{align*}
Then, the critical values $c_\eta(\nu)$ and $c_\gamma(\nu)$ can be obtained by the upper $(1-\nu)$-percentile of the conditional distribution of $\widetilde{\mathcal{T}}_{\eta}$ and $\widetilde{\mathcal{T}}_{\gamma},$ respectively.

\subsection{Tests for degree heterogeneity}\label{sec:dht}
Degree heterogeneity is an important feature in real-world networks, 
but it is not always clear whether a network has degree heterogeneity, especially when the network is sparse.
In this section, we consider the test for degree heterogeneity.
As mentioned in Section \ref{section:model},
it 
is equivalent 
to test the following hypotheses:
\begin{align*}
	& H_{01}: \alpha_{i}^*(t)=\alpha^*(t)~~ \text{for~all}~~ i\in[n]\\
	\text{ versus } & H_{11}: \text{There exists some}\ i\in[n]\ \text{such that}\
	\alpha_{i}^*(t)\neq \alpha^*(t),
\end{align*}
and
\begin{align*}
	& H_{02}: \beta_{i}^*(t)=\beta^*(t)~~ \text{for~all}~~  i\in[n-1]\\
	\text{ versus } & H_{12}: \text{There exists some}\ i\in[n-1]\ \text{such that}\ \beta_{i}^*(t)\neq \beta^*(t),
\end{align*}
where $\alpha^*(t)$ and $\beta^*(t)$ are some unspecified functions.
We see that
\begin{itemize}
	\item If $H_{01}$ holds but $H_{02}$ does not, then the network has only in-degree heterogeneity.
	\item If $H_{02}$ holds but $H_{01}$ does not, then the network has only out-degree heterogeneity \citep{perry2013point}.
	\item If both $H_{01}$ and $H_{02}$ hold, then the network has no degree heterogeneity \citep{kreib2019, Alexander2021}.
\end{itemize}
Let $e_{i,j}$ be a $(2n-1)$-dimensional vector, in which its $i$th element is $1,$
$j$th element is $-1$ and other elements are zeros.
We consider the following test statistics for $H_{01}$ and $H_{02},$ respectively:
\begin{align*}
	\mathcal{D}_\alpha=&\max_{i\neq j\in[n]}\sup_{t\in[a,b]}
	\sqrt{nh_1}|\widehat \alpha_{i}(t)-\widehat \alpha_{j}(t)|/\widehat\zeta_{ij,\alpha}^{1/2}(t),\\
	\mathcal{D}_\beta=&\max_{i\neq j\in[n-1]}\sup_{t\in[a,b]}
	\sqrt{nh_1}|\widehat \beta_{i}(t)-\widehat \beta_{j}(t)|/\widehat\zeta_{ij,\beta}^{1/2}(t),
\end{align*}
where 
\begin{align*}
	\widehat\zeta_{ij,\alpha}(t)=&e_{i,j}^\top\widehat S(t)\widehat\Omega(t)\widehat S(t) e_{i,j},\\
	\widehat\zeta_{ij,\beta}(t)
	=&e_{n+i,n+j}^\top\widehat S(t)\widehat\Omega(t)\widehat S(t)e_{n+i,n+j}.
\end{align*}

The test statistics $\mathcal{D}_{\alpha}$ and $\mathcal{D}_\beta$ are close to zero under the nulls
$H_{01}$ and $H_{02},$ respectively.
Hence we reject $H_{01}$ if $\mathcal{D}_\alpha>c_1(\nu),$
and reject $H_{02}$ if $\mathcal{D}_\beta>c_2(\nu),$
where $c_1(\nu)$ and $c_2(\nu)$ are the critical values.
We consider a resampling approach to obtain these critical values. 
Here we only focus on how to obtain $c_1(\nu)$, and
$c_2(\nu)$ can be obtained similarly.
Using arguments similar to the proof of Theorem \ref{theorem-central-degree},
we can show that under the null $H_{01},$
the distribution of $\sqrt{nh_1}\{\widehat \alpha_{i}(t)-\widehat \alpha_{j}(t)\}~(i\neq j\in[n])$ is asymptotically equivalent to
$\sqrt{h_1/n}e_{i,j}^\top \widehat S(t)\int_{0}^\tau \mathcal{K}_{h_1}(u-t) d\widetilde{\mathcal{M}}(u).$
As in Section \ref{sec:Trt}, we replace $\mathcal{M}_{is}(u)$ with $N_{is}(u)G_{is}$ in $\widetilde{\mathcal{M}}(u),$
that is,
\begin{align*}
	\widetilde{\mathcal{D}}_{ij,\alpha}(t)=\sqrt{\frac{h_1}{n}}e_{i,j}^\top\widehat S(t)\int_0^{\tau} \mathcal{K}_{h_1}(u-t)d\widetilde N(u),
\end{align*}
By repeatedly generating the normal random sample $G_{ij},$
the distribution of $\mathcal{D}_{\alpha}$
can be approximated by the conditional distribution of $\widetilde{\mathcal{D}}_{\alpha}$ given the observed data, where
\begin{align*}
	\widetilde{\mathcal{D}}_{\alpha}=&\max_{i\neq j\in[n]}\sup_{t\in[a,b]} |\widetilde{\mathcal{D}}_{ij,\alpha}(t)|/\widehat\zeta_{ij,\alpha}^{1/2}(t).
\end{align*}
Then, the critical value $c_1(\nu)$ can be obtained by the upper $(1-\nu)$-percentile of the conditional distribution of $\widetilde{\mathcal{D}}_{\alpha}.$ 

\section{Numerical studies} 
\label{section:simulation}
\subsection{Simulation studies} 
\label{section:5.1}
In this section, we carry out simulation studies to evaluate the finite sample performance of the proposed method. 
The time-varying degree parameters $\alpha_i^*(t)$ and $\beta_j^*(t)$ are
\begin{equation*}
	\alpha_i^*(t)= \begin{cases}
		-c_0\log(n)+(2.5+\sin(2\pi t)), & \text{if} \quad i<\frac{n}{2}, \\
		-c_0\log(n)+(1.5+t/2), & \text{if} \quad i\geq\frac{n}{2},
	\end{cases}
\end{equation*}
and
\begin{equation*}
	~~~~~~\beta_j^*(t)= \begin{cases}
		-c_0\log(n)+(2.5+\cos(2\pi t)), & \text{if} \quad j<\frac{n}{2}, \\
		-c_0\log(n)+(1.5+t/2), & \text{if} \quad \frac{n}{2}\leq j < n, \\
		0 &  \text{if} \quad j = n,
	\end{cases}
\end{equation*}
where $c_0$ is used to specify sparse regimes.
We take $c_0=0.5$ and hence $q_n\approx\log(n).$
The network sparsity level defined by
$\tau n^2e^{-q_n}$ is $O(n),$
which is less than $n^2$ and a moderately sparse network is generated.
For the homophily term, we set $\gamma^*(t)=(\gamma_1^*(t),\gamma_2^*(t))^\top$ with $\gamma_1^*(t)=\gamma_2^*(t)=\sin(2\pi t)/3$. 
The covariates $Z_{ij}$ are independently generated from the standard normal distribution. 
We set $\tau=1$ and the numbers of nodes as $n=100$, $200$ and $500.$ 
The Gaussian kernel $\mathcal{K}(x)=\exp(-x^2/2)/(2\pi)^{1/2}$ is used 
and the bandwidths are chosen by the rule of thumb: $h_1 = 0.1n^{-{1}/{5}}$ and $h_2 = 0.015n^{-{2}/{5}}$. 
All of the results are based on $1000$ replications.
To measure the error of the estimators, we use the mean integrated squared error (MISE), which is defined by
$$
\text{MISE}=\frac{1}{1000}\sum_{k=1}^{1000}\int_{0}^{\tau}[\hat{f}_{k}(t)-f_k^*(t)]^2dt.
$$
Here, $f_k^*(t)$ denotes the true value 
and $\hat f_{k}(t)$ is its estimate in the $k$th replication.

The MISE for the estimators of $\alpha_1^*(t), ~\alpha_{n/2+1}^*(t),~\beta_1^*(t),~\beta_{n/2+1}^*(t)$
and $\gamma_1^*(t)$ are reported in Table \ref{MISEnew}.
The results for other parameters are similar and are omitted. 
From Table \ref{MISEnew}, we can see that all MISEs are small and less than $0.2.$
The MISE decreases as the sample size $n$ increases, as we expected. 
The MISE for $\gamma_1^*(t)$ is much smaller (up to two orders of magnitude) than that for degree parameters,
which is due to the fact that the dimension of regression coefficients is fixed while the number of degree parameters is of order $n$.

The averages of the 1000 estimated coefficient curves for $\alpha_1^*(t),~\beta_1^*(t)$ and $\gamma_1^*(t),$ 
and their pointwise $95\%$ confidence bands are given in Figure \ref{fig:sendnew}.  
We can see that as the number of nodes increases, 
the estimated curves become closer to their true curves, 
and the confidence bands tend to cover the entire true curves. 
Table \ref{estimates} reports the coverage probabilities of the pointwise $95\%$ confidence intervals
and the average lengths of the confidence intervals for  $\alpha_1^*(t),~\beta_1^*(t)$ and $\gamma_1^*(t).$
We see that the coverage 
probabilities are close to the nominal level 
and the lengths of the confidence intervals decrease as $n$ increases.
Figure S1 in supplementary material further displays the asymptotic distributions
of standardized $\widehat\alpha_1(t)$, $\widehat\beta_1(t)$ and $\widehat\gamma_1(t)~(t=0.6~\text{and}~0.8)$ with $n=500,$
which can be well approximated by the standard
normal distribution.
This confirms the theoretical results in Theorems  \ref{theorem-central-gamma} and \ref{theorem-central-degree}.

\emph{Comparison with \cite{kreib2019}}.
We now compare our method with \cite{kreib2019} on the performance of estimating homophily parameters.
Since \cite{perry2013point} assumed that the effects of covariates are constant over time,
their method is not compared here.
In this simulation, the homophily parameter is set as $\gamma^*(t)=\sin(2\pi t)/3$ and the covariates $Z_{ij}$ 
are set to be 1 if $i\leq 4$ and $j\leq n/3$, and $0$ otherwise.
We set $\alpha_i^*(t)$ and $\beta_i^*(t)$ as
\begin{align*}
	\alpha_i^{*}(t)=\beta_i^{*}(t)=\begin{cases}
		b\big[-0.5\log(n)+(3+t/2)\big], & \text{if} \quad i<\frac{n}{2}, \\
		0, & \text{otherwise}.
	\end{cases}
\end{align*}
When $b=0,$ the simulated network does not have degree heterogeneity, 
and hence both methods yield consistent estimators. 
As $b$ increases, the method of \cite{kreib2019} may give 
biased estimates for homophily parameters 
due to the presence of degree heterogeneity.
We choose $b$ to be $0$, $1/3$, $1/2$ and $1$, and set $n=200$. 

The results based on $1000$ replications are shown in Figure \ref{fig:het}.
We see that when $b=0$, the performance of the two methods are comparable in terms of bias.
However, our method leads to a wider confidence band, which is not surprising
because
there are $2n+p-1$ unknown parameters 
in our model,
while there are only $p$ unknown parameters 
in \citeauthor{kreib2019}'s model.
On the other hand, our model still performs well in estimating $\gamma^*(t)$ with $b=1/3$, $1/2$  and $1$, 
but the method of \cite{kreib2019} yields a biased estimate for $\gamma^*(t)$.
The bias increases as $b$ increases from $1/3$ to $1$,
and when $b=1$, the $95\%$ pointwise confidence band even fails to cover the entire true curve  of $\gamma^*(t)$.
This indicates that when degree heterogeneity exists in a network, 
neglecting this feature may result in a biased estimate for homophily effects.

\emph{Tests for trends}.
To examine the performance of the tests for trends,
we set $\alpha_{i}^*(t)=\beta^*_j(t)=-0.5\log(n)+(2.5+\tilde c_1\sin(2\pi t))$ for all $i\in[n]$ and $j\in[n-1]$, and $\gamma^*(t)=\tilde c_2\sin(2\pi t)/3$. The covariates $Z_{ij}$ are independently generated from the standard normal distribution. 
The parameters $\tilde c_1$ and $ \tilde c_2$ indicate the trending level.
We can see that $H_{0\eta}$ holds if $\tilde c_1=0$, and the departure from $H_{0\eta}$ increases as $\tilde c_1$ increases.
Similarly, $H_{0\gamma}$ holds if $\tilde c_2=0,$ and
the departure from $H_{0\gamma}$ increases as $\tilde c_2$ increases.
For simplicity, we choose $t_1$ and $t_2$ as $0.1, 0.2, \ldots, 0.9$ to calculate $\mathcal{T}_{\eta}$ and $\mathcal{T}_{\gamma}.$
The kernel function and bandwidth are chosen to be the same as before.
The sample size $n=100$ and $200,$ and the level $\nu$ is chosen as $0.05$.
The critical values are calculated using the resampling method with 1000 simulated realizations.

Figure \ref{fig:con} depicts the size and power of the statistics $\mathcal{T}_\eta$ and $\mathcal{T}_{\gamma}.$
Note that the estimated sizes of $\mathcal{T}_\eta$ and $\mathcal{T}_{\gamma}$ are around 0.05, and the empirical powers of both test statistics increase 
as $\tilde c_1$ and $\tilde c_2$ increase. 
The powers also increase with the sample size.
The results show that the statistics ${\mathcal{T}}_{\eta}$ and $\mathcal{T}_{\gamma}$ perform well under the null hypothesis, 
and can also successfully detect the time-varying trends of parameters under the alternative hypothesis.

\emph{Tests for degree heterogeneity}.
We now examine the performance of the test statistics proposed in Section \ref{sec:dht} to test degree heterogeneity.
We set $\gamma^*(t)=\sin(2\pi t)/3$, and the covariates $Z_{ij}$ are independently generated from the standard normal distribution. 
Let $\alpha_{i}^*(t)=\beta^*_i(t)=\alpha(t)~\text{for~all}~ i\in[n-1],$ $\alpha_n^*(t)=\alpha(t) + \tilde c$
and $\beta_n^*(t)=0$, where $\tilde c$ indicates the level of degree heterogeneity. 
When $\tilde c=0,$ the null hypothesis $H_{01}$ holds, and
the departure from  the null $H_{01}$ increases as $\tilde c$ increases.
Here we set $\alpha(t)=t/2$.
We choose $t$ as $0.1,0.2,\dots,0.9$ to calculate $\mathcal{D}_{\alpha}.$
The kernel function and bandwidth are chosen to be the same as in Section \ref{section:5.1}.
The sample size $n=100$ and $200,$ and the level $\nu$ is chosen as $0.05$.
The critical values are calculated using the resampling method with 1000 simulated realizations.

The estimated size and power of ${\mathcal{D}}_{\alpha}$ are presented in Figure \ref{fig:power}.
We see that the estimated size is around 0.05 when $\tilde c=0,$ 
and the power of the proposed test increases as $\tilde c$ tends to 1.5.
In addition, the powers also increase when the sample size increases from $100$ to $200$.
The results show that the statistic ${\mathcal{D}}_{\alpha}$ 
performs well under the null hypothesis, and can also successfully detect the existence of degree heterogeneity under the alternative hypothesis.
The performance of ${\mathcal{D}}_{\beta}$ is similar to that of ${\mathcal{D}}_{\alpha}$ and it is omitted here.

\subsection{Real data analysis} 
\label{section:application}

In this section, we apply the proposed method to analyze a collaborative network data,
which can be retrieved from the Web of Science database (\url{https://www.webofscience.com/wos/woscc/basic-search}). 
From the field of machine learning, we retrieved the information for a total of 30,000 most cited papers from Jan. 2000 to Apr. 2022. 
The original dataset includes  key information including author names, article titles, source titiles, keywords, abstracts, addresses, 
email addresses and other publication information.
Since machine learning is becoming popular in recent years, 
it is of interest to know the developing trend of this field in different countries 
and how the collaboration network evolved between different regions.
Therefore, we extracted the collaboration network between countries from the original dataset.
The nodes of the network represent different countries or regions.
For each article, if the first author and other authors come from different countries, 
then these countries constitute the collaboration relationship, and a directed edge 
from the country of the first author to the country of each collaborator is added to the network. 
Multiple edges between two countries are allowed, but self-loops (links within the same country) are omitted.
There are 74 countries (labelled from 1 to 74 as nodes) with 19,679 collaborating records (edges) in the collaboration network.

Figure \ref{fig:struc} depicts the snapshots of the networks during four periods:
Jan. 2000 - Apr. 2002, May 2002 - Aug. 2004, Sep. 2004 - Jan. 2007, Feb. 2007 - May. 2009.
The size of the node corresponds to its in-degree and the countries are color coded by continent.
We can see that during Jan. 2000 - Apr. 2002, the collaborations concentrate on a few countries from Europe and America. 
However, more and more links involve Asian countries during Feb. 2007 - May 2009.
In Figure S2 of the supplementary material, we further plot the in-degrees and out-degrees of five selected countries from Jan. 2000 to Apr. 2022,
and the figure shows that both degrees grow significantly over time.
In addition, Figure S3 in the supplementary material depicts the total in-degrees and out-degrees of 74 countries during Jan. 2000 to Apr. 2022, 
and it shows that the degrees vary a lot from country to country.
For example, the highest in-degree is 3743 while the lowest is only 11.
These results imply that the collaboration network may have degree heterogeneity 
and its structure be time evolving.

From Figure \ref{fig:struc}, we also observe that there may exist some hub nodes (i.e., nodes with high degrees) in the collaboration network, 
e.g., USA which is the largest orange node.
The hub nodes tend to form many interactions with other countries, irrespective of the continent the country is from.
If we consider countries from the same continent (i.e., nodes with the same color) as homophilous, 
and countries from different continents (i.e., nodes with different colors) as heterophilous, 
the existence of hub nodes leads to even more interactions between heterophilous nodes than those between homophilous nodes.
For example, during Jan. 2000 - Apr. 2002, 77 out of 111 edges are heterophilic. 
A classical dynamic model of link formation may conclude that the preferences are not homophilic in the collaboration network, 
but our analysis later shows that there are still homophily effects in this data.

Our interests are to estimate the time-varying trend of in- and out-degrees 
and to examine whether homophilic effects exist in the collaboration network.
For this, we consider covariates: $M_{i}$, $A_{i}$ and $E_i$, $i\in\{1,\dots,74\}$, 
which are indicators of whether country $i$ is from America, Asia and Europe, respectively. Let $X_i=(M_{i},A_i,E_i)^T$. 
We further define $Z_{ij}=X_i\otimes X_j$, i.e. 
$$
Z_{ij}=(M_{i}M_{j},M_{i}A_{j},M_{i}E_j,A_{i}M_{j},A_{i}A_{j}, A_{i}E_j,E_iM_{j},E_iA_{j},E_iE_j)^\top.
$$
Here $M_{i}A_{j}$ denotes whether the first author is from America and the collaborator from Asia.
Other terms are defined similarly. 
As in simulation studies, the Gaussian kernel $\mathcal{K}(x)=\exp(-x^2/2)/(2\pi)^{1/2}$ is used 
and the bandwidths are chosen by the rule of thumb as $h_1 = 0.1n^{-{1}/{5}}$ and $h_2 = 0.015n^{-{2}/{5}}$. 

We carried out the testing procedures described in Section \ref{sec:Trt} to examine 
whether the parameters $\alpha_i^*(t)$ and $\beta_j^*(t)$ are time-varying,
and the p-value is less than 0.001,
which indicates that the trends of in- and out-degrees are time-varying.
Figure \ref{fig:realcountryID} shows the estimates of $\alpha_i^*(t)$ and $\beta_j^*(t)$ for some countries.
To avoid selecting the baseline for comparison,
Figure \ref{fig:realcountryID} gives 
the curves $\widehat{\alpha}_i^{'}(t)=\widehat{\alpha}_i(t) - {\overline{\alpha}}(t)$ and 
$\widehat{\beta}_i^{'}(t)=\widehat{\beta}_i(t) - {\overline{\beta}}(t),$ 
where ${\overline{\alpha}}(t)=n^{-1}\sum_{i=1}^n\widehat\alpha_i(t)$ and ${\overline{\beta}}(t)=n^{-1}\sum_{i=1}^n\widehat\beta_i(t).$ 
We observe that different countries have different trends of collaboration activities.
For example, we see a decreasing trend (comparing to the average) in USA for collaborative work, 
both as the first author or other authors. This seems to indicate that USA's role 
in leading global collaboration in machine learning research became less dominant over time, 
although their output level is still above the average in recent years.
The plots for Chad show different collaboration patterns and their role in the collaboration. 
Most authors from Chad tend to be the first author in their collaborative research with $\hat{\alpha}(t)$ above average 
and $\hat{\beta}(t)$ below average in the beginning period, 
but a few years later, they start to take the role of the collaborator more frequently.
The estimated curves for Nigeria are very different from those of USA and Chad. 
The curves increase rapidly after the year 2007, but the effects are negative on the collaboration activities over the entire time interval.
We also implement the testing procedure described in Section \ref{sec:dht} to test degree heterogeneity.
The p-values for testing both $\alpha_i^*(t)$ and $\beta_i^*(t)$
are less than 0.001, which indicates that the collaboration network has degree heterogeneity.

For homophily effects, we first carry out the testing procedure developed in Section \ref{sec:Trt} to test the existence of time-varying trend.
The p-value is less than 0.001,
which indicates that the homophily effects have significant time evolving trends.
We then present the estimated curves of the homophily parameters in Figure \ref{fig:realcountryHT}. 
It can be seen that if both countries are from America or Europe, 
they tend to collaborate more frequently. 
However, two Asian countries are more likely to collaborate in the beginning,
but have a lower tendency to collaborate in recent years.
Our results differ from those obtained by the method of \cite{kreib2019}.
For example, their method shows that the homophily effects on the collaboration between Asian countries are very close to zero,
but our method reveals that the homophily parameter is significantly positive at the very beginning. 
This can possibly be attributed to the consideration of degree heterogeneity in our method.

\section{Discussions}
In this article, we proposed a new degree-corrected Cox network model for the analysis of network recurrent data. We developed a kernel smoothing method to estimate the homophily and individual-specific parameters, and established consistency and asymptotic normality of the proposed estimators. We also proposed testing procedures to test for trend and degree heterogeneity in dynamic networks. 
Although we focus on directed networks, the methods developed here can be 
easily adapted into undirected networks. 
Numerical studies demonstrated that the proposed method performed well in practice.

There are several directions for future research. 
First, the degree heterogeneity we addressed in this paper may be incorporated into other models as well. For example, to model the dependence structure in network settings, the concept of asymptotic uncorrelation was proposed by \citet{kreib2019}, and further extended to momentary-$m$-dependence and $\beta$-mixing in \cite{Alexander2021}. See more related work in \cite{Sit2020} and \cite{Alexander2021}. However, the degree heterogeneity, which has not been considered in these papers, 
makes the mathematics behind these types of analysis significantly more challenging with $2n$ parameters. 
It would be of interest to explore this in the future. 

Second, we only used the local constant fitting to construct the estimating equation for simplicity. Indeed, it can be extended to 
local linear fitting and general local polynomial fitting. However, the resulting inference procedures would be much more complicated and need future research.
Finally, we chose the bandwidth by the rule of thumb.
Developing data-driven methods, such as the $K$-fold cross-validation, to select the optimal bandwidth is certainly of interest.
However, it poses challenges with $2n$ individual-specific parameters under our model, which requires more effort to study in the future.

\bibliographystyle{apalike-alt}
\bibliography{reference2}
\addtolength{\itemsep}{-2 em}

\clearpage
\newpage

\begin{table}[h] {\rm \  \hspace*{-0.3cm}{}
		\caption{The MISEs for parameters $\alpha_1^*(t), \alpha_{n/2+1}^*(t), \beta_1^*(t), \beta_{n/2+1}^*(t)$ and	$\gamma_1^*(t)$.}
		\begin{center}
			\renewcommand{\arraystretch}{1.15} \tabcolsep 0.16in  \doublerulesep 1.0pt
			\begin{tabular}{ccccccc} \hline   
				\multirow{2}{*}{\centering$n$}	 &&   \multicolumn{5}{c}{MISE}       \\  \cline{3-7} 		
				&&  $\alpha_1^*(t)$ & $\alpha_{n/2+1}^*(t)$  & $\beta_1^*(t)$   &
				$\beta_{n/2+1}^*(t)$ &	$\gamma_1^*(t)$     \\ 
				\hline
				$100$	 &&	0.129   & 0.190   & 0.130   & 0.169 & 0.008\\ 
				$200$	 &&	0.111   & 0.173   & 0.107   & 0.153 & 0.004\\ 	 
				$500$	 &&	0.104   & 0.169   & 0.096   & 0.149 & 0.002\\ 			 
				\hline
			\end{tabular}
		\end{center}
		\label{MISEnew} 
}\end{table}

\begin{table}[h] {\rm \  \hspace*{-0.3cm}{}
		\caption{The coverage probability for parameters $\times100$ (the length of 95\% confidence interval) for $\alpha_1^*(t), \alpha_{n/2+1}^*(t), \beta_1^*(t), \beta_{n/2+1}^*(t)$ and $\gamma_1^*(t)$ at different time points.}
		\begin{center}
			\renewcommand{\arraystretch}{1.15} \tabcolsep 0.26in  \doublerulesep 1.0pt
			\begin{tabular}{clrrr} \hline   
				$n$&&  \multicolumn{1}{c}{$t=0.4$} & \multicolumn{1}{c}{$t=0.6$}  & \multicolumn{1}{c}{$t=0.8$}      \\ 
				\hline
				$100$	  &$\alpha_1^*(t)$          &	 95.5 (1.00)   & 92.3 (1.61)   & 95.9 (1.45)\\
				&$\alpha_{n/2+1}^*(t)$    &	93.7 (1.64)   & 93.8 (1.69)   & 95.4 (1.39)\\ 
				&$\beta_1^*(t)$           &	92.8 (1.12)   & 95.8 (1.60)   & 92.7 (1.30)\\ 
				&$\beta_{n/2+1}^*(t)$     &	92.6 (1.18)   & 94.7 (1.65)   & 93.9 (1.66)\\ 
				&$\gamma_1^*(t)$          &	93.8 (0.31)   & 95.8 (0.44)   & 94.7 (0.37)\\ 	 
				$200$	  &$\alpha_1^*(t)$          &	 96.5 (0.95)   & 97.8 (1.54)   & 91.1 (1.36)\\
				&$\alpha_{n/2+1}^*(t)$    &	95.8 (1.65)   & 95.3 (1.63)   & 95.9 (1.29)\\ 
				&$\beta_1^*(t)$           &	90.7 (1.08)   & 92.3 (1.57)   & 94.4 (1.20)\\ 
				&$\beta_{n/2+1}^*(t)$     &	96.1 (1.15)   & 96.3 (1.61)   & 95.5 (1.60)\\ 
				&$\gamma_1^*(t)$          &	95.3 (0.23)   & 95.3 (0.33)   & 95.3 (0.27)\\ 
				$500$	  &$\alpha_1^*(t)$          &	 95.5 (0.92)   & 97.1 (1.53)   & 94.8 (1.33)\\
				&$\alpha_{n/2+1}^*(t)$    &	92.9 (1.73)   & 94.5 (1.65)   & 93.7 (1.24)\\ 
				&$\beta_1^*(t)$           &	93.7 (1.08)   & 94.9 (1.60)   & 93.8 (1.14)\\ 
				&$\beta_{n/2+1}^*(t)$     &	95.5 (1.15)   & 97.5 (1.54)   & 95.4 (1.59)\\ 
				&$\gamma_1^*(t)$          &	94.6 (0.16)   & 95.3 (0.23)   & 95.2 (0.19)\\ 	 
				\hline
			\end{tabular}
		\end{center}
		\label{estimates} 
}\end{table}

\begin{figure}[h]
	\centering
	\subfloat[Subfigure 1 list of figures text][$n=100$, $\hat{\alpha}_1(t)$]{
		\includegraphics[width=0.33\textwidth]{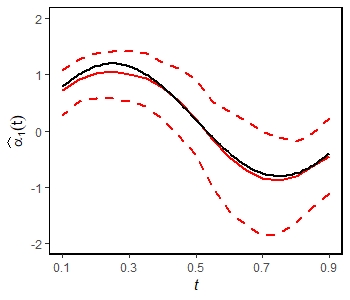}}
	\subfloat[Subfigure 2 list of figures text][$n=100$, $\hat{\beta}_1(t)$]{
		\includegraphics[width=0.33\textwidth]{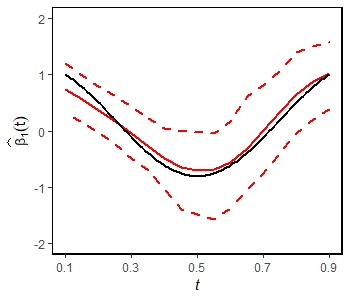}}
	\subfloat[Subfigure 3 list of figures text][$n=100$, $\hat{\gamma}_1(t)$]{
		\includegraphics[width=0.33\textwidth]{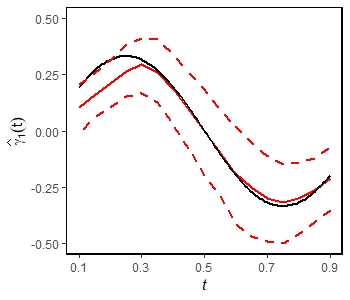}}
	\qquad
	\subfloat[Subfigure 1 list of figures text][$n=200$, $\hat{\alpha}_1(t)$]{
		\includegraphics[width=0.33\textwidth]{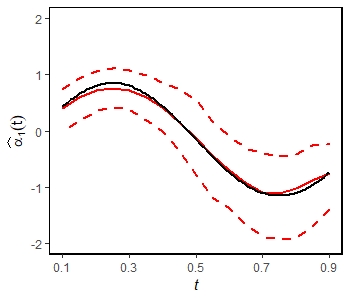}}
	\subfloat[Subfigure 2 list of figures text][$n=200$, $\hat{\beta}_1(t)$]{
		\includegraphics[width=0.33\textwidth]{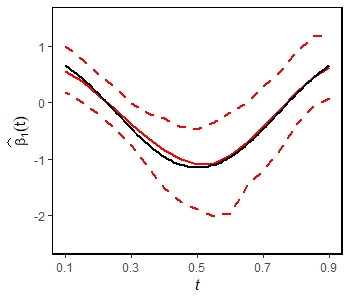}}
	\subfloat[Subfigure 3 list of figures text][$n=200$, $\hat{\gamma}_1(t)$]{
		\includegraphics[width=0.33\textwidth]{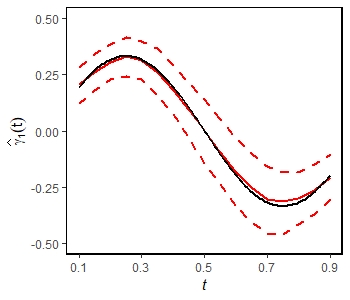}}
	\qquad
	\subfloat[Subfigure 1 list of figures text][$n=500$, $\hat{\alpha}_1(t)$]{
		\includegraphics[width=0.33\textwidth]{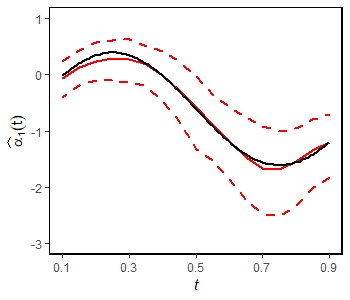}}
	\subfloat[Subfigure 2 list of figures text][$n=500$, $\hat{\beta}_1(t)$]{
		\includegraphics[width=0.33\textwidth]{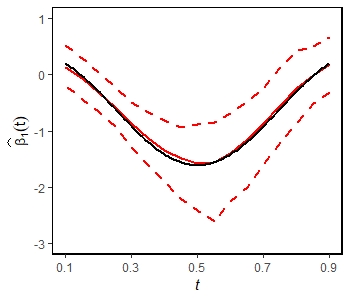}}
	\subfloat[Subfigure 3 list of figures text][$n=500$, $\hat{\gamma}_1(t)$]{
		\includegraphics[width=0.33\textwidth]{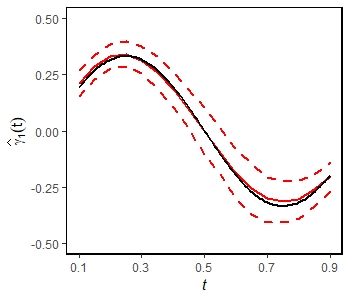}}
	\qquad
	\caption{The estimated curves for $\alpha_1^*(t),~\beta_1^*(t)$ and $\gamma_1^*(t).$ The black solid lines denote the true curves of $\alpha_1^*(t), ~\beta_1^*(t)$ and $\gamma_1^*(t).$ The red solid lines are the averages (over 1000 replications) of the proposed estimators $\widehat\alpha_1(t),~\widehat\beta_1(t)$ and $\widehat\gamma_1(t)$, while the red dashed lines represent the pointwise $95\%$ confidence intervals.}
	\label{fig:sendnew}
\end{figure}

\begin{figure}[h]
	\centering
	\subfloat[Subfigure 1 list of figures text][$b=0$]{
		\includegraphics[width=0.4\textwidth]{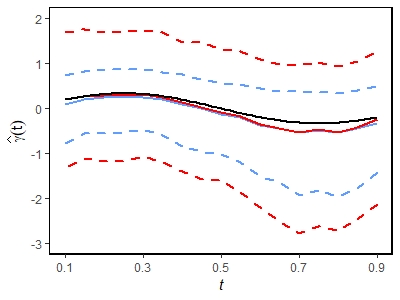}}
	\subfloat[Subfigure 3 list of figures text][$b=1/3$]{
		\includegraphics[width=0.4\textwidth]{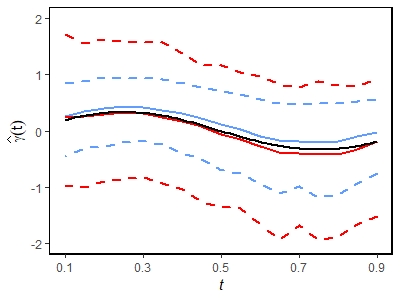}}\\
	\subfloat[Subfigure 3 list of figures text][$b=1/2$]{
		\includegraphics[width=0.4\textwidth]{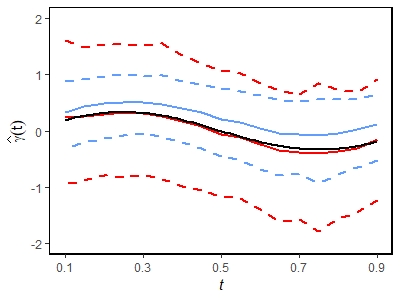}}
	\subfloat[Subfigure 3 list of figures text][$b=1$]{
		\includegraphics[width=0.4\textwidth]{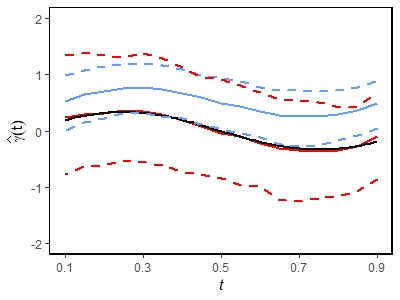}}
	\qquad
	\caption{The red solid lines are the average of the proposed estimators  of $\gamma^*(t)$ over 1000 replications,
		and the red dashed lines are the pointwise $95\%$ confidence intervals. The blue solid lines are the average of the estimators obtained by the method of \cite{kreib2019} over 1000 replications, and the blue dashed lines are the pointwise $95\%$ confidence intervals. The black solid lines represent the true curves of $\gamma^*(t).$}
	\label{fig:het}
\end{figure}

\begin{figure}[h]
	\centering
	\subfloat[Subfigure 1 list of figures text][Size and power of $\mathcal{T}_{\eta}$]{
		\includegraphics[width=0.43\textwidth]{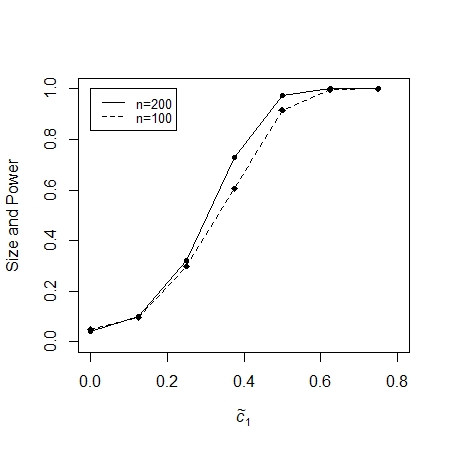}}
	\subfloat[Subfigure 3 list of figures text][Size and power of $\mathcal{T}_{\gamma}$]{
		\includegraphics[width=0.43\textwidth]{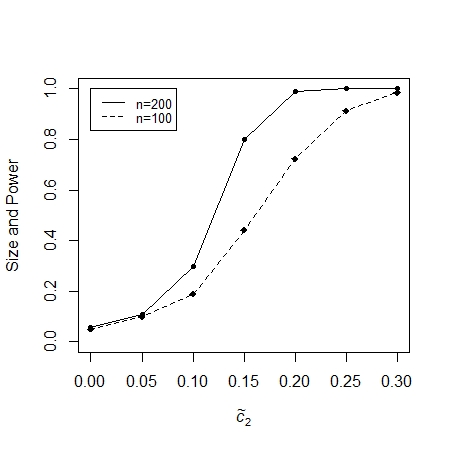}}
	\caption{The size and power of the test statistics $\mathcal{T}_{\eta}$ and $\mathcal{T}_{\gamma}$ as trending levels $\tilde{c}_1$ and $\tilde{c}_2$ increase.}
	\label{fig:con}
\end{figure}

\begin{figure}[h]
	\centering
	\subfloat{
		\includegraphics[width=0.45\textwidth]{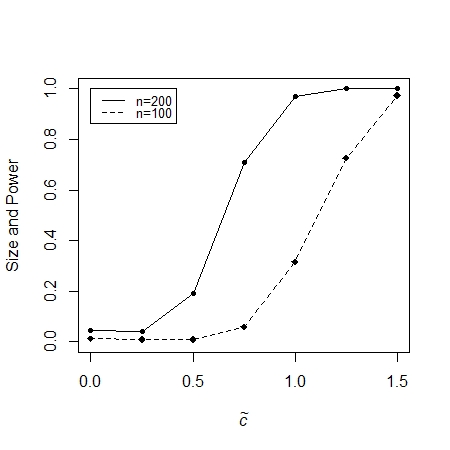}}
	\caption{The size and power of the test statistic $\mathcal{D}_{\alpha}$ as the heterogeneity level $\tilde{c}$ varies from 0 to 1.5.}
	\label{fig:power}
\end{figure}


\begin{figure}[h]
	\centering
	\subfloat[Subfigure 1 list of figures text][Jan. 2000 - Apr. 2002 (P1)]{
		\includegraphics[width=0.49\textwidth]{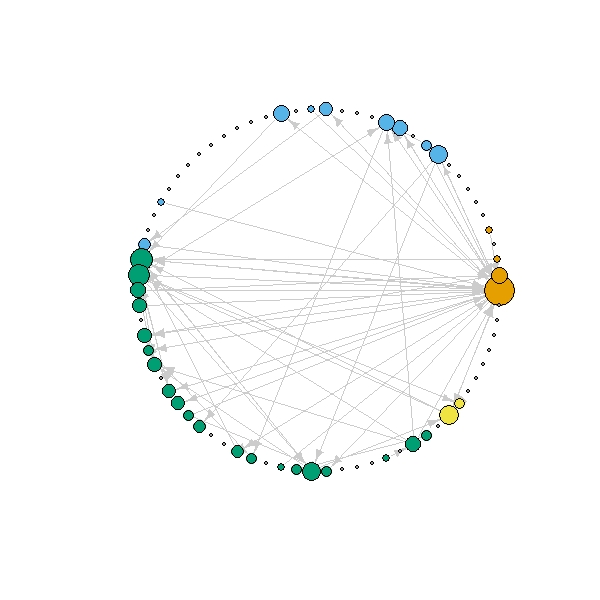}}
	\subfloat[Subfigure 3 list of figures text][May. 2002 - Aug. 2004 (P2)]{
		\includegraphics[width=0.49\textwidth]{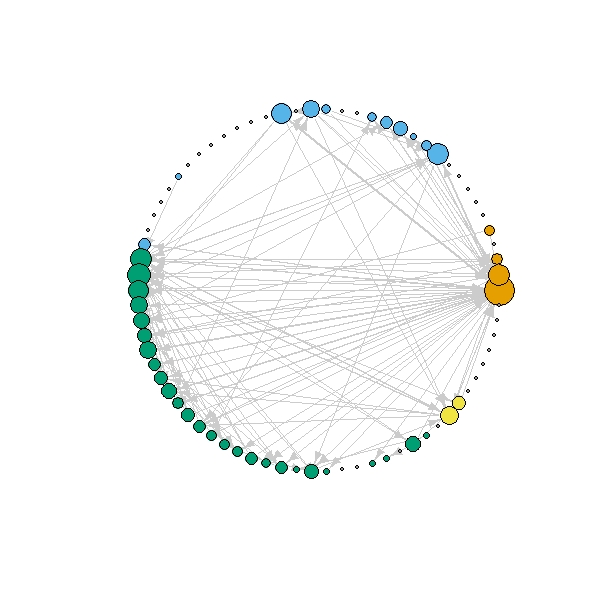}} \\
	\subfloat[Subfigure 1 list of figures text][Sep. 2004 - Jan. 2007 (P3)]{
		\includegraphics[width=0.49\textwidth]{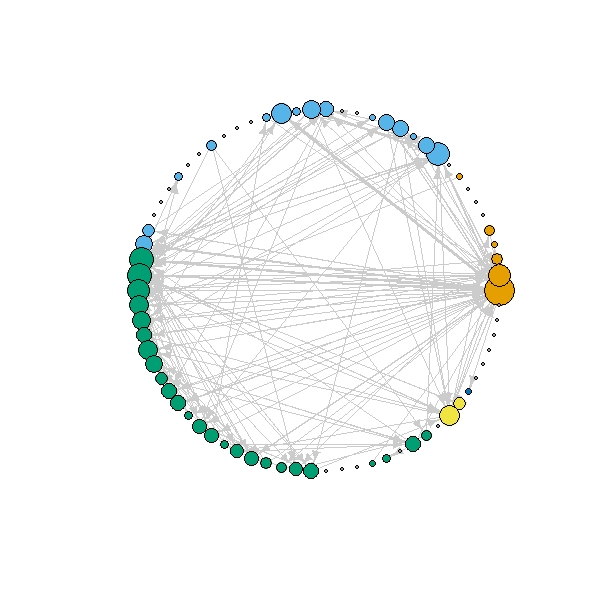}}
	\subfloat[Subfigure 3 list of figures text][Feb. 2007 - May. 2009 (P4)]{
		\includegraphics[width=0.49\textwidth]{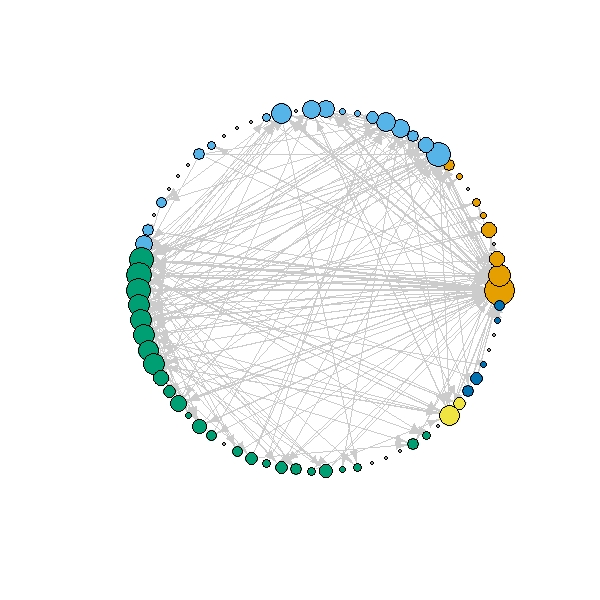}} \\
	\caption{The collaboration links of 74 countries from Jan. 2000 to May. 2009. The size of nodes corresponds to their in-degrees. The countries are color coded by continent: America (orange), Asia (cyan), Europe (green), Oceania (yellow) and Africa (dark blue).}
	\label{fig:struc}
\end{figure}

		

\begin{figure}[h]
	\centering
	\subfloat[Subfigure 1 list of figures text][Chad, $\hat{\alpha}^{'}(t)$]{
		\includegraphics[width=0.38\textwidth]{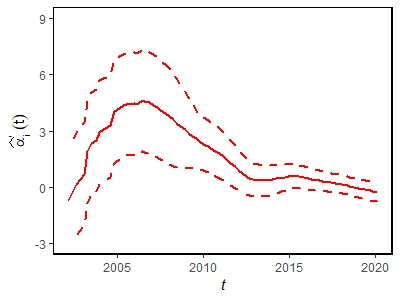}}
	\subfloat[Subfigure 2 list of figures text][Chad, $\hat{\beta}^{'}(t)$]{
		\includegraphics[width=0.38\textwidth]{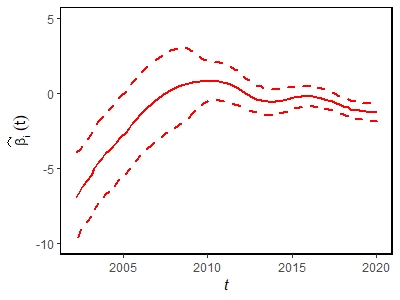}}
	\qquad
	\subfloat[Subfigure 1 list of figures text][USA, $\hat{\alpha}^{'}(t)$]{
		\includegraphics[width=0.38\textwidth]{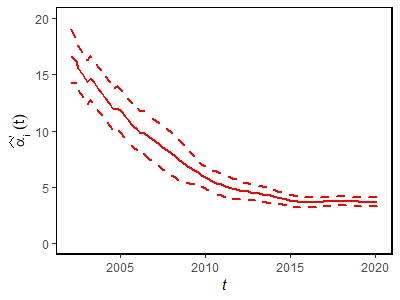}}
	\subfloat[Subfigure 2 list of figures text][USA, $\hat{\beta}^{'}(t)$]{
		\includegraphics[width=0.38\textwidth]{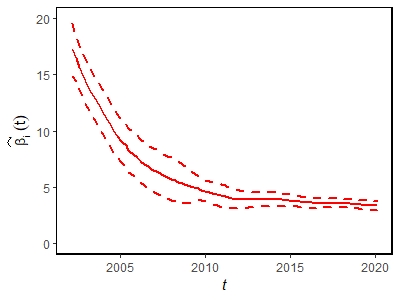}}
	\qquad
	\subfloat[Subfigure 1 list of figures text][Nigeria, $\hat{\alpha}^{'}(t)$]{
		\includegraphics[width=0.38\textwidth]{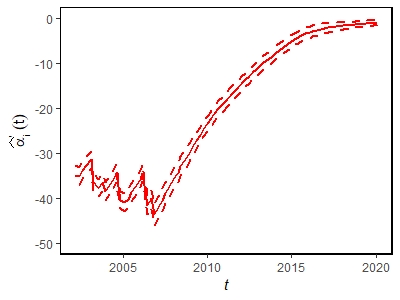}}
	\subfloat[Subfigure 2 list of figures text][Nigeria, $\hat{\beta}^{'}(t)$]{
		\includegraphics[width=0.38\textwidth]{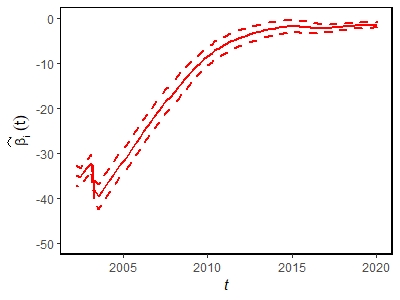}}
	\caption{Real data analysis: Estimation results of degree heterogeneity parameters. 
		The solid lines are the estimated curves, and the dashed lines denote their $95\%$ pointwise confidence bands.}
	\label{fig:realcountryID}
\end{figure}

\begin{figure}[h]
	\captionsetup[subfloat]{margin=10pt,format=hang,singlelinecheck=false}
	\centering
	\subfloat[Subfigure 1 list of figures text][$\hat{\gamma}_1(t)$ for covariate $M_iM_j$]{
		\includegraphics[width=0.33\textwidth]{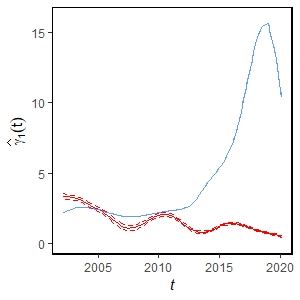}}
	\subfloat[Subfigure 2 list of figures text][$\hat{\gamma}_2(t)$ for covariate $M_iA_j$]{
		\includegraphics[width=0.33\textwidth]{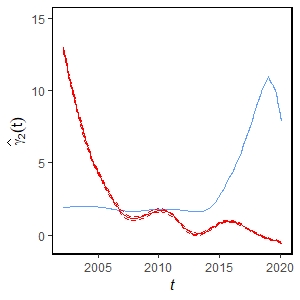}}
	\subfloat[Subfigure 2 list of figures text][$\hat{\gamma}_3(t)$ for covariate $M_iE_j$]{
		\includegraphics[width=0.33\textwidth]{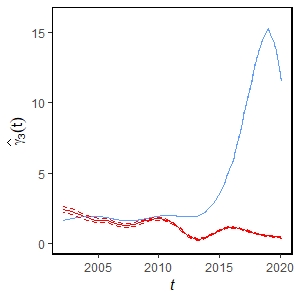}}
	\qquad
	\subfloat[Subfigure 1 list of figures text][$\hat{\gamma}_4(t)$ for covariate $A_iM_j$]{
		\includegraphics[width=0.33\textwidth]{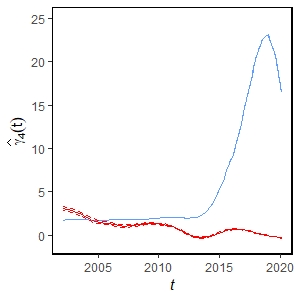}}
	\subfloat[Subfigure 2 list of figures text][$\hat{\gamma}_5(t)$ for covariate $A_iA_j$]{
		\includegraphics[width=0.33\textwidth]{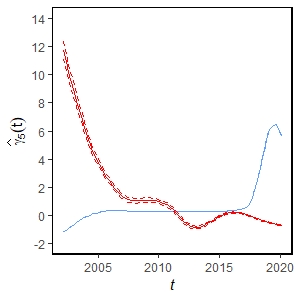}}
	\subfloat[Subfigure 2 list of figures text][$\hat{\gamma}_6(t)$ for covariate $A_iE_j$]{
		\includegraphics[width=0.33\textwidth]{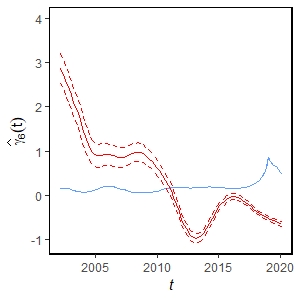}}
	\qquad
	\subfloat[Subfigure 1 list of figures text][$\hat{\gamma}_7(t)$ for covariate $E_iM_j$]{
		\includegraphics[width=0.33\textwidth]{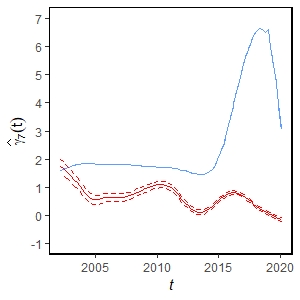}}
	\subfloat[Subfigure 2 list of figures text][$\hat{\gamma}_8(t)$ for covariate $E_iA_j$]{
		\includegraphics[width=0.33\textwidth]{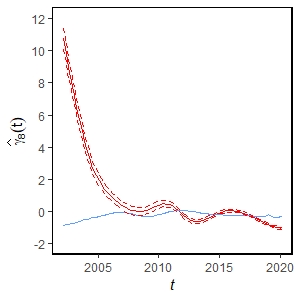}}
	\subfloat[Subfigure 2 list of figures text][$\hat{\gamma}_9(t)$ for covariate $E_iE_j$]{
		\includegraphics[width=0.33\textwidth]{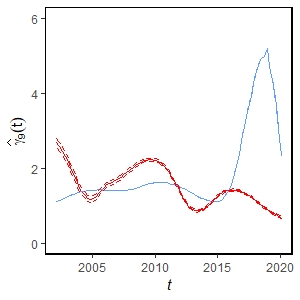}}
	\caption{Real data analysis: Estimation results of homophily parameters. The red solid lines represent our estimates and red dashed lines are their $95\%$ pointwise confidence bands. The blue lines are estimates obtained by the method of \cite{kreib2019}.}
	\label{fig:realcountryHT}
\end{figure}

%
%

\clearpage
\newpage

\appendix

\newpage
\setcounter{page}{1}

\begin{center}
{\Large\bf \textsf{
Supplementary Material for ``A degree-corrected Cox model for dynamic
networks"}}
\end{center}

\section{Preliminaries}

In this section, we present some results that will be used in the proofs and state them as lemmas.
Given $b_L$, $b_U>0$, we say $M=(m_{i,j})_{n\times n}$ belongs to the matrix class $\mathcal{L}_{n}(b_L, b_U)$ if
$M$ satisfies
\begin{equation*}\label{eq1}
\begin{array}{l}
m_{i,j}=m_{j,i}=0, ~~ i,j=1,\dots,n,~i\neq j,\\
m_{i,j}=m_{j,i}=0, ~~ i,j=n+1,\dots,2n-1, ~i\neq j,\\
m_{n+i,i}=m_{i,n+i}=0, ~~ i,j=1,\dots,n-1, \tag{A.1}\\
b_L\le m_{i,i}-\sum_{j=n+1}^{2n-1} m_{i,j}\le b_U,~~ i=1,\dots,n,\\
m_{i,i}=\sum_{k=1}^{n}m_{k,i}=\sum_{k=1}^{n}m_{i,k}, ~~i=n+1,\dots,2n-1,\\
b_L\le m_{i,j}=m_{j,i} \le b_U, ~~ i=1,\ldots,n,~ j=n+1,\dots,2n-1;~ j\neq n+i.
\end{array}
\end{equation*}
Clearly, if $M\in \mathcal{L}_{n}(b_L, b_U)$, then $M$ is a $(2n-1)\times (2n-1)$ diagonally dominant, symmetric nonnegative matrix and $M$ has the following structure:
\[
M= \left(\begin{array}{ll} M_{11} & M_{12} \\
M_{12}^\top  & M_{22}
\end{array}\right) ,
\]
where $M_{11}$ ($n$ by $n$) and $M_{22}$ ($n-1$ by $n-1$) are diagonal matrices, 
$M_{12}$ ($n-1$ by $n$) is a nonnegative matrix whose non-diagonal elements are positive and diagonal elements equal to zero.
Here, the diagonal elements of $M_{12}$ refer to elements whose row and column indices are the same.
Generally, the inverse of $M$ does not have a closed form.
Define $m_{2n,i}=m_{i,2n}:= m_{i,i}-\sum_{j=1;j\neq i}^{2n-1} m_{i,j}$ for $i=1,\ldots, 2n-1$ and $m_{2n,2n}=\sum_{i=1}^{2n-1} m_{2n,i}$. Then $b_L \le m_{2n,i} \le b_U$ for $i=1,\ldots, n-1$, $m_{2n,i}=0$ for $i=n, n+1,\ldots, 2n-1$ and $m_{2n,2n}=\sum_{i=1}^n m_{i, 2n}=\sum_{i=1}^n m_{2n, i}$.
\cite{YLZ2016} proposed to use a simple matrix $\mathcal{S}(M)=(s_{ij})_{(2n-1)\times (2n-1)}$ to approximate the inverse of $M\in \mathcal{L}_n(b_L, b_M)$,
where $s_{ij}$ is defined as
\begin{equation*}
s_{i,j}=\left\{\begin{array}{ll}\frac{\delta_{i,j}}{m_{i,i}} + \frac{1}{m_{2n,2n}}, & i,j=1,\ldots,n, \\
-\frac{1}{m_{2n,2n}}, & i=1,\ldots, n,~~ j=n+1,\ldots,2n-1, \\
-\frac{1}{m_{2n,2n}}, & i=n+1,\ldots,2n-1,~~ j=1,\ldots,n, \\
\frac{\delta_{i,j}}{m_{i,i}}+\frac{1}{m_{2n,2n}}, & i,j=n+1,\ldots, 2n-1.
\end{array}
\right.
\end{equation*}
In the above equation, $\delta_{i,j}=1$ when $i=j$ and $\delta_{i,j}=0$ when $i\neq j$. 

Define $\|M\|_{\max}:=\max_{i,j} |m_{i,j}|$ 
as the maximum absolute entry-wise norm for any matrix $M=(a_{i,j}).$
\cite{Yan:Leng:Zhu:2016} proved that the upper bound of the approximation error has an order $n^{-2}$.

\begin{lemma} [Proposition 1 in \cite{YLZ2016}] \label{lem1}
If $M\in \mathcal{L}_n(b_L, b_U)$ with $b_U/b_L=o(n)$, then for large enough $n$,
\[
\| M^{-1}-\mathcal{S}(M) \|_{\max} \le \frac{c_1b_U^2}{b_L^3(n-1)^2},
\]
where $c_1$ is a constant that does not depend on $M$, $m$ and $n$.
\end{lemma}

Let $G(x): \R^n \to \R^n$ be a function vector on $x\in\R^n$. We say that a Jacobian matrix $G^{(1)}(x)$ with $x\in \R^n$ is Lipschitz continuous on a convex set $\mathcal{D}\subset\R^n$ if
for any $x,y\in \mathcal{D}$, there exists a constant $\lambda>0$, such that
for any vector $v\in \R^n$, the inequality
\begin{equation*}
\| \left\{ G^{(1)} (x) - G^{(1)} (y) \right\} v \|_{\infty} \le \lambda \| x - y \|_{\infty} \|v\|_{\infty}
\end{equation*}
holds.

We introduce an error bound in the Newton method by \cite{Kantorovich-Akilov1964}
under the Kantorovich conditions \citep{Kantorovich1948Functional}.

\begin{lemma}[Theorem 6 in \cite{Kantorovich-Akilov1964}]
\label{lem2}
Let $\mathcal{D}$ be an open convex subset of $\R^n$ and
$F:\mathcal{D} \to \R^n$ be Fr\'{e}chet differentiable.
Assume that, at some ${x}_0 \in \mathcal{D}$, $F^\prime({x}_0)$ is invertible and that
\begin{eqnarray}
\label{eq-kantororich-a}
\| F^\prime({x}_0)^{-1} ( F^\prime({x}) - F^\prime({y}))\| \le K\|{x}-{y}\|,~~ {x}, {y}\in \mathcal{D}, \\
\label{eq-kantororich-b}
\| F^\prime({x}_0)^{-1} F({x}_0) \| \le \eta,~~ h=K\eta \le 1/2, \\
\nonumber
\bar{S}({x}_0, t^*) \subseteq \mathcal{D},~~ t^*=2\eta/( 1+ \sqrt{ 1-2h}).
\end{eqnarray}
Then:
(1) The Newton iterates ${x}_{n+1} = {x}_n - \{ F^\prime ({x}_n)\}^{-1} F({x}_n)$, $n\ge0$, are well-defined,
lie in $\bar{S}({x}_0, t^*)$ and converge to a solution ${x}^*$ of $F({x})=0$. \\
(2) The solution ${x}^*$ is unique in $S({x}_0, t^{**})\cap \mathcal{D}$, $t^{**}=(1 + \sqrt{1-2h})/K$ if $2h<1$
and in $\bar{S}({x}_0, t^{**})$ if $2h=1$. \\
(3) $\| {x}^* - {x}_n \| \le t^*$ if $n=0$ and $\| {x}^* - {x}_n \| \le 2^{1-n} (2h)^{ 2^n -1 } \eta $ if $n\ge 1$.
\end{lemma}

The following lemma gives an exponential bound for 
the moment of a bounded random variable. 

\begin{lemma}\label{lemma3}
Let $Y$ be a random variable.
If $|Y|<C_0$ for some constant $C_0$ and $\mathbb{E}(Y)=0,$ then
for all sufficiently small $\upsilon>0,$
$$
\mathbb{E}\{\exp(\upsilon Y)\big\}\le \exp\big\{\upsilon^2\mathbb{E}(Y^2)\big\}.
$$
\end{lemma}
\begin{proof}
By Taylor's expansion, for a small $\upsilon>0$,
we have
\begin{align*}
\mathbb{E}\big\{\exp(\upsilon Y)\big\}=&1+\upsilon \mathbb{E}(Y)+\frac{1}{2}\upsilon^2\mathbb{E}(Y^2)+\sum_{j=3}^{\infty}\frac{\upsilon^j\mathbb{E}Y^j}{j!}\\
\le & 1+\frac{1}{2}\upsilon^2\mathbb{E}(Y^2)
+\frac{\upsilon^2\mathbb{E}Y^2}{2}\sum_{j=3}^{\infty}\frac{\upsilon^{j-2}C_0^{j-2}}{3^{j-2}}\\
=& 1+\frac{1}{2}\upsilon^2\mathbb{E}(Y^2)
+\frac{\upsilon^2\mathbb{E}Y^2}{2}\frac{\upsilon C_0/3}{1-\upsilon C_0/3}.
\end{align*}
By choosing $0<\upsilon<\min\{1,2C_0/3\},$ we obtain
\begin{align*}
1+\frac{1}{2}\upsilon^2\mathbb{E}(Y^2)
+\frac{\upsilon^2\mathbb{E}Y^2}{2}\frac{\upsilon C_0/3}{1-\upsilon C_0/3}
\le 1+\upsilon^2\mathbb{E}(Y^2)
\le \exp\Big\{\upsilon^2\mathbb{E}(Y^2)\Big\},
\end{align*}
where the last inequality holds for all small $\upsilon>0.$
This completes the proof.
\end{proof}

\section{Proofs of Theorems \ref{theorem:consistency}-\ref{theorem-central-degree}}
\label{section-proof-th15}
\subsection{Proof of Theorem \ref{theorem:consistency}}
\label{section-proof-con}

In this section, we present the proof of Theorem \ref{theorem:consistency}.
We first prove three
lemmas. The  first lemma is about 
 the upper bounds of $\|F^*(t)\|_\infty$ and $\|Q^*(t)\|_\infty$,
where $F^*(t)=F(\theta^*(t))$ and $Q^*(t)=Q(\theta^*(t))$.
For a given $\gamma(t)$, write
\[
F_{\gamma}(\eta(t))=(F_{1,\gamma}(\eta(t)),\dots,F_{2n-1,\gamma}(\eta(t)))^\top=F(\eta(t),\gamma(t)).
\]
Define $\alpha_{i,\gamma}^*(t)$ and $\beta_{j,\gamma}^*(t)$ as the solution to $\mathbb{E}\{\mathcal{M}_{ij}(t; \alpha_i(t),\beta_j(t),\gamma(t))\}=0~(1\le i\neq j \le n)$ with a given $\gamma(t).$ Further define $\eta^*_\gamma(t)=(\alpha_{1,\gamma}^*(t),\dots,\alpha_{n,\gamma}^*(t),
\beta_{1,\gamma}^*(t),\dots,\beta_{n,\gamma}^*(t))^\top.$


\begin{lemma}
\label{lem4}
Suppose Conditions \ref{condition:CB1}-\ref{condition:bandwidth5} holds,
and $\gamma(t)$ takes values in a compact set $\mathcal{J}$ of $\mathbb{R}^p.$
Then we have
\begin{eqnarray}
\label{ineq-bound-Ft}
\sup_{t\in[a,b]}\|F^*(t)\|_\infty & = & O_p\Big((q_n+1) e^{q_n}\sqrt{\frac{\log nh_1}{nh_1}}\Big), \\
\label{ineq-bound-Qt}
\sup_{t\in[a,b]}\|Q^*(t)\|_\infty & = & 
O_p\Big(\kappa_n(q_n+1) e^{q_n}\sqrt{\frac{\log Nh_2}{Nh_2}}\Big), \\
\label{ineq-bound-lemma4-3}
\sup_{t\in[a,b]}\sup_{\gamma(t)\in\mathcal{J}}\|F_{\gamma}(\eta_{\gamma}^*(t))\|_\infty& =&O_p\Big((q_n+1) e^{q_n}\sqrt{\frac{\log nh_1}{nh_1}}\Big).
\end{eqnarray}
\end{lemma}

\begin{proof}
For easy exposition, define 
\begin{eqnarray*}
\pi_{ij}^*(t)&=&\alpha_{i}^*(t)+\beta_j^*(t)+Z_{ij}(t)^\top\gamma^*(t),\\
\pi_{1,ij}^*(t) & = & \alpha_{i}^*(t)+Z_{ij}(t)^\top\gamma^*(t), \\
\pi_{2,ij}^*(t) & = & \beta_j^*(t)+Z_{ij}(t)^\top\gamma^*(t).
\end{eqnarray*}
Similarly, define
\begin{eqnarray*}
\pi_{ij}^*(s,t) & = & \alpha_{i}^*(t)+\beta_j^*(t)+Z_{ij}(s)^\top\gamma^*(t), \\
\pi_{1,ij}^*(s,t) & = & \alpha_{i}^*(t)+Z_{ij}(s)^\top\gamma^*(t), \\
\pi_{2,ij}^*(s,t) & = & \beta_j^*(t)+Z_{ij}(s)^\top\gamma^*(t).
\end{eqnarray*}
For $i=1,\ldots,n$,  we decompose $(n-1)F_i^*(t)$ into two parts:
\begin{equation}
\label{decomposition-Fit}
(n-1)F_i^*(t)=
\underbrace{\sum_{j\neq i}\int_{0}^{\tau}\mathcal{K}_{h_1}(s-t)d\mathcal{M}_{ij}(s)}_{B_{1i}(t)}
- \underbrace{\sum_{j\neq i}\int_0^{\tau}K_{h_1}(s-t)\big[e^{\pi_{1,ij}^*(s,t)}-e^{\pi_{1,ij}^*(s)}\big]ds}_{
B_{2i}(t) }.
\end{equation}
Note that if $g(t)$ is a generic function with second derivatives bounded above by a constant $c$,
we have (e.g., \citeauthor{E1988}, \citeyear{E1988}, p.128)
\begin{align*}
\Big|\int_{0}^\tau\mathcal{K}_{h_1}(s-t)[g(s)-g(t)]ds\Big |\le \frac{1}{2}ch_1^2,
\end{align*}
where $t\in[h_1,\tau-h_1]$ and $h_1$ satisfies $[-1,1] \subset [-t/h_1, (\tau-t)/h_1]$.
It follows from Conditions \ref{condition:CB1}, \ref{condition:PB2} and \ref{condition:kernel4} that
\begin{align*}
|\mathbb{E} B_{2i}(t)|=\bigg|\sum_{j\neq i}\mathbb{E}\Bigg\{\int_{0}^{\tau}\mathcal{K}_{h_1}(s-t)\big[e^{\pi_{ij}^*(s,t)}-e^{\pi_{ij}^*(s)}\big]ds\Bigg\}\bigg|
=O\big(e^{q_n}nh_1^2\big)
\end{align*}
and
\begin{align*}
\text{Var}(B_{2i}(t))\le\sum_{j\neq i}\mathbb{E}\Bigg[\int_{0}^{\tau}\mathcal{K}_{h_1}(s-t)\big[e^{\pi_{ij}^*(s,t)}ds-e^{\pi_{ij}(s)}\big]ds\Bigg]^2
=O\big(e^{2q_n}nh_1^4\big),
\end{align*}
where $f^{(l)}(t)$ denotes the $l$th-order derivative of any function $f(t)$.
By Chebyshev's inequality, we have
\begin{equation}\label{ineq-Bi2}
B_{2i}(t) = O_p(e^{q_n}nh_1^2).
\end{equation}
If $ne^{q_n}h_1^5\rightarrow 0$, then 
\[
B_{2i}(t) = o_p \left( e^{2q_n}\sqrt{n\log(nh_1)/h_1} \right).
\]

By the union bound and the triangle inequality, we have that for sufficiently large $n$,
\begin{align}
\nonumber
&\PP\Big(\sup_{t\in[a,b]}(n-1)\|F^*(t)\|_\infty \ge 4\sqrt{10}C_1\vartheta_n\sqrt{n\log(nh_1)/h_1} \; \Big) \\
\nonumber
\le&(n-1)\max_{1\le i\le n-1}\PP\Big(\sup_{t\in[a,b]}|B_{1i}(t)|
+\sup_{t\in[a,b]}|B_{2i}(t)\}|
\ge 4\sqrt{10}C_1\vartheta_n \sqrt{n\log(nh_1)/h_1} \; \Big)\\
\label{ineq-bB1i}
\le&(n-1)\max_{1\le i\le n-1}\PP\Big(\sup_{t\in[a,b]}|B_{1i}(t)|
\ge 2\sqrt{10}C_1\vartheta_n \sqrt{n\log(nh_1)/h_1} \; \Big),
\end{align}
where $C_1$ is some constant specified below and $\vartheta_n=e^{q_n}(q_n+1)>1$.

We now bound the probability in \eqref{ineq-bB1i}.
For any $\kappa>0$, partition $[a,b]$ into disjoint subsets $\bigcup_{k=1}^{M_n} \Theta_k$ such that
the distance between any two points in $\Theta_k$ does not
exceed $\kappa h_1^{2}$. Note that $M_n$ is not larger that $O(\tau/(\kappa h_1^2))$.
Then, we have
\begin{equation}
\label{eqA2}
\mathbb{P}\Big(\sup_{t\in[a,b]}|B_{1i}(t)|>\varepsilon\Big)
\le\sum_{k=1}^{M_n}\mathbb{P}\Big(|B_{1i}(t_k)|>\frac{\varepsilon}{2}\Big)
+ \sum_{k=1}^{M_n}\mathbb{P}\Big(\sup_{t\in\Theta_k}|B_{1i}(t)-B_{1i}(t_k)|>\frac{\varepsilon}{2}\Big).
\end{equation}
Let $Y_{1ij}(t)=\int_{0}^{\tau}\mathcal{K}((s-t)/h_1)d\mathcal{M}_{ij}(s).$
By Markov's inequality and Lemma \ref{lemma3}, we have
\begin{eqnarray*}
&&\mathbb{P}\Big(\sum_{j\neq i}Y_{1ij}(t_k)>h_1\varepsilon/2\Big) \\
& \le & \exp\{-\upsilon h_1\varepsilon/2 \}\mathbb{E}\Big[\exp\Big\{\upsilon\sum_{j\neq i}Y_{1ij}(t_k)\Big\}\Big]\\
& \le & \exp\{-\upsilon h_1\varepsilon/2 \}\prod_{j\neq i}\mathbb{E}\Big[\exp\big\{\upsilon Y_{1ij}(t_k)\big\}\Big]\\
& \le & \exp\{-\upsilon h_1\varepsilon/2 \}\prod_{j\neq i}\exp\Big(\upsilon^2 \text{Var}(Y_{1ij}(t_k))\Big)
\end{eqnarray*}
for some small $\upsilon>0$. 
Note that there exists some constant $C_1>0$ such that $|Y_{1ij}(t_k)|<C_1$ and $\text{Var}(Y_{1ij}(t_k))<C_1^2\vartheta_n^2h_1.$
Therefore, we have
$$
\mathbb{P}\Big(\sum_{j\neq i}Y_{1ij}(t_k)>h_1\varepsilon/2\Big)\le
\exp\Big(-\upsilon h_1\varepsilon/2+C_1^2\upsilon^2 n\vartheta_nh_1\Big).
$$
Then, by setting $\varepsilon=2\sqrt{10}C_1\vartheta_n\sqrt{n\log(nh_1)/h_1}$
and $\upsilon=\varepsilon/(4C_1^2\vartheta_n^2n),$
we have
\begin{align}\label{eqA3}
\mathbb{P}\Big(|B_{1i}(t_j)|>\varepsilon\Big)
\le O((nh_1)^{-5/2}). 
\end{align}
Because $B_{1i}(t)$ is uniformly continuous and the first-order derivative of $\mathcal{K}(x)$ is bounded, there exists a constant $C_2$ such that
$$
\frac{1}{n}|B_{1i}(t)-B_{1i}(t_j)|<(C_2/h_1^2)|t-t_j|<C_2\kappa.
$$
By setting  $\kappa<\varepsilon/(4nC_2)$,  
the second term in the right-hand side of \eqref{eqA2} vanishes, i.e.,
\begin{equation}
\label{eqA4}
\sum_{k=1}^{M_n}\mathbb{P}\Big(\sup_{t\in\Theta_k}|B_{1i}(t)-B_{1i}(t_k)|>\frac{\varepsilon}{2}\Big)=0.
\end{equation}
We now bound the first term in the right-hand side of \eqref{eqA2}.
Note that
$$
M_n=O(1/(\kappa h_1^2))=o((n/h_1^3)^{1/2}).
$$
In view of \eqref{eqA2}, \eqref{eqA3} and \ref{eqA4}), there exists  for some constant $C_3>0$
such that
\begin{align}
\label{ineq-prob-Bi1}
\mathbb{P}\Big(\sup_{t\in[a,b]}|B_{1i}(t)|>2\sqrt{10}C_1\vartheta_n\sqrt{n\log(nh_1)/h_1}\Big)\le \frac{C_3}{nh_1^2}.
\end{align}
Combining \eqref{ineq-Bi2}, \eqref{ineq-bB1i} and \eqref{ineq-prob-Bi1} yields
\begin{align*}
\max_{1\le i\le n}|F_i^*(t)|=O_p\Big((q_n+1) e^{q_n}\sqrt{\frac{\log nh_1}{nh_1}}\Big),
\end{align*}
uniformly in $t\in[a,b]$.
For $i=n+1,\dots 2n-1$, with the same arguments, we have 
\begin{equation*}
\max_{n+1\le i\le 2n-1}|F_i^*(t)|=O_p\Big((q_n+1)e^{q_n}\sqrt{\frac{\log nh_1}{nh_1}}\Big)
\end{equation*}
uniformly in $t\in[a,b]$. This shows \eqref{ineq-bound-Ft}.

Next, we show \eqref{ineq-bound-Qt}. We divide $Q(\eta^*(t))$ into two parts:
\begin{align*}
NQ(\eta^*(t))=& \underbrace{\sum_{i=1}^n \sum_{j=1,j\neq i}^n
\int_{0}^{\tau}Z_{ij}(s)\mathcal{K}_{h_2}(s-t)d\mathcal{M}_{ij}(s)}_{B_2(t)}\\
&- \underbrace{\sum_{i=1}^n \sum_{j\neq i}
\int_{0}^{\tau}Z_{ij}(s)\mathcal{K}_{h_2}(s-t)\big[e^{\pi_{ij}(s,t)}
-e^{\pi_{ij}(s)}\big]ds}_{B_3(t)}.
\end{align*}
Similar to the proof of $B_{2i}(t)$, we have
$$
|B_3(t)|=O_p((q_n+1)^2e^{q_n}\kappa_nh_2^2n^2),
$$
which is dominated by $O(\vartheta_n^2\kappa_nn\sqrt{\log(nh_2)/h_2})$ if $ne^{q_n}h_2^{5/2}\rightarrow 0$.
In addition, the variance of each term in the sum $B_{2}(t)$ is in the order of
$C_4\kappa_n^2\vartheta_ne^{q_n}h_2$ for some constant $C_4>0$.
It is less than $C_4\kappa_n^2\vartheta_n^2h_2$. 
With similar arguments as in the proof of \eqref{ineq-bound-Ft}, we have
that for some constant $C_5>0,$
\begin{align*}
\mathbb{P}\Big(\sup_{t\in[a,b]}\big|B_{2}(t)\big|>4\sqrt{10}C_4\kappa_n\vartheta_nn
\sqrt{\log(nh_2)h_2}\Big)\le \frac{C_5}{nh_2^2}.
\end{align*}
That is,
\begin{equation*}
\frac{1}{n^2}\|Q(\eta^*(t))\|_\infty=O_p\Big(\kappa_n(q_n+1)e^{q_n}\sqrt{\frac{\log n^2h_2}{n^2h_2}}\Big).
\end{equation*}

Lastly, we show \eqref{ineq-bound-lemma4-3}. Note that
\begin{align*}
\label{eqB11}
(n-1)F_i\big(\eta_{\gamma}^*(t)\big)&=
\underbrace{\sum_{j\neq i}\int_{0}^{\tau}\mathcal{K}_{h_1}(s-t)d\mathcal{M}_{ij}(s;\gamma)}_{\tilde B_{1i}}\\
&-\underbrace{\sum_{j\neq i}\int_0^{\tau}K_{h_1}(s-t)\big[e^{\alpha_{i,\gamma}^*(t)+\beta_{j,\gamma}^*(t)+Z_{ij}^\top(s)\gamma(t)}-e^{\alpha_{i,\gamma}^*(s)+\beta_{j,\gamma}^*(s)+Z_{ij}^\top(s)\gamma(s)}\big]ds}_{\tilde B_{2i}}, \tag{B.11}
\end{align*}
where $\mathcal{M}_{ij}(s;\gamma)=\mathcal{M}_{ij}(t; \alpha_{i,\gamma}(t),\beta_{j,\gamma}^*(t),\gamma(t)).$
By some arguments similar to $B_{2i}(t),$
we have $\tilde B_{2i}=o_p( e^{2q_n}\sqrt{\log(nh_1)/(nh_1)}).$
Next we partition $[a,b]$  into $M_n$ subintervals, denoted by $\{\Theta_{m}: 1\le m\le M_n\}.$ Further partition $\mathcal{J}$ into $\tilde M_n$ subintervals, denoted by $\{J_{m}: 1\le m\le \tilde M_n\}.$ 
Let the distance between any two points in $\Theta_M$ and $J_m$ does not exceed $\kappa h_1^{2}/2.$
Then, by arguments similar to the proofs of \eqref{eqA3} -\eqref{ineq-prob-Bi1}, we can show that 
\begin{align*}
\|\tilde B_{1i}\|_{\infty} = O_p\bigg((q_n+1) e^{q_n}\sqrt{\frac{\log(nh_1)}{nh_1}}\bigg).
\end{align*}
By \eqref{eqB11} and the fact $\tilde B_{2i}=o_p( e^{2q_n}\sqrt{n\log(nh_1)/h_1})$,
we have $\|F_{\gamma}\big(\eta_{\gamma}^*(t)\big)\|_{\infty}=O_p\big((q_n+1) e^{q_n}\sqrt{\log(nh_1)/(nh_1)}\big).$ This completes the proof.
\end{proof}

\begin{lemma}\label{lem5}
Let $\mathcal{D}$ be the set of twice continuously differentiable functions on $(0,\tau]$ such that $\eta_{\gamma}(t)\in\mathcal{B}$ defined in Condition \ref{condition:PB2}. The Jacobian matrix $F^{(1)}_{\gamma}(x(t))$ of $F_{\gamma}(x(t))$ on $\mathcal{D}$ satisfies
\begin{align*}
\|[F_{\gamma}^{(1)}(x(t)) - F_{\gamma}^{(1)}(y(t))]v\|_\infty &\le 3e^{q_n} \|x(t) - y(t)\|_\infty \|v\|_\infty, \\
\max_{i=1,\ldots,2n-1} \|F_{i,\gamma}^{(1)}(x(t)) - F_{i,\gamma}^{(1)}(y(t))\|_\infty &\le 3e^{q_n} \|x(t) - y(t)\|_\infty.
\end{align*}
\end{lemma}

\begin{proof}
Define 
\begin{align*}
F_{i,\gamma}^{(1)}(\eta(t)):=&\Big(\frac{\partial F_{i,\gamma}(\eta(t))}{\partial \alpha_1(t)}, \ldots,
\frac{\partial F_{i,\gamma}(\eta(t))}{\partial \alpha_{n}(t) },
\frac{\partial F_{i,\gamma}(\eta(t))}{\partial \beta_1(t) }, \ldots,
\frac{\partial F_{i,\gamma}(\eta(t))}{\partial \beta_{n-1}(t) }\Big)\\
=&(F_{i,1,\gamma}^{(1)}(\eta(t)), \ldots, F_{i,(2n-1),\gamma}^{(1)}(\eta(t))).
\end{align*}
The first-order partial derivatives of $F_{i,\gamma}(\eta(t))$ with regard to $\alpha_i(t)$ and $\beta_j(t)$
are given below:
\begin{align*}
-\frac{\partial F_{i,\gamma}(\eta(t))}{\partial \alpha_i(t)}
=&\frac{1}{n-1}\sum_{j\neq i}\int \mathcal{K}_{h_1}(s-t)\exp\{\alpha_i(t)+\beta_j(t)+Z_{ij}(s)^\top\gamma(t)\}ds\\
=& \frac{1}{n-1}\sum_{j\neq i}\exp\{\alpha_i(t)+\beta_j(t)+Z_{ij}(t)^\top\gamma(t)\} +O_p\big((e^{q_n}h^2)\big),\\
-\frac{\partial F_{i,\gamma}(\eta(t))}{\partial \beta_j(t)}=&\frac{1}{n-1}\int_0^{\tau}\mathcal{K}_{h_1}(s-t)\exp\{\alpha_i(t)+\beta_j(t)+Z_{ij}(s)^\top\gamma(t)\}ds\\
=& \frac{1}{n-1}\exp\{\alpha_i(t)+\beta_j(t)+Z_{ij}(t)^\top\gamma(t)\} +O_p\big((e^{q_n}h^2)\big)~~(j\neq i),\\
-\frac{\partial F_{i,\gamma}(\eta(t))}{\partial \alpha_j(t)}=&0~~(j\neq i),~~~
-\frac{\partial F_{i,\gamma}(\eta(t))}{\partial \beta_i(t)}=0.
\end{align*}
The second-order partial derivatives of  $F_{i,\gamma}(\eta(t))$ are calculated as
\begin{align*}
-\frac{\partial^2 F_{i,\gamma}(\eta(t)) }{ \partial^2 \alpha_i(t)}=&
\frac{1}{n-1}\sum_{j\neq i}\int_0^{\tau}\mathcal{K}_{h_1}(s-t)\exp\{\alpha_i(t)+\beta_j(t)+Z_{ij}(s)^\top\gamma(t)\}ds\\
=& \frac{1}{n-1}\sum_{j\neq i}\exp\{\alpha_i(t)+\beta_j(t)+Z_{ij}(t)^\top\gamma(t)\} +O_p\big((e^{q_n}h^2)\big),\\
-\frac{\partial^2 F_{i,\gamma}(\eta(t)) }{\partial\beta_j(t) \partial\alpha_i(t)}=&
\frac{1}{n-1}\int_0^{\tau}\mathcal{K}_{h_1}(s-t)\exp\{\alpha_i(t)+\beta_j(t)+Z_{ij}(s)^\top\gamma(t)\}ds\\
=& \frac{1}{n-1}\exp\{\alpha_i(t)+\beta_j(t)+Z_{ij}(t)^\top\gamma(t)\} +O_p\big((e^{q_n}h^2)\big),
~~j\in[n-1], j\neq i,\\
-\frac{\partial^2 F_{i,\gamma}(\eta(t)) }{\partial^2\beta_j(t)}=&
\frac{1}{n-1}\int_0^{\tau}\mathcal{K}_{h_1}(s-t)\exp\{\alpha_i(t)+\beta_j(t)+Z_{ij}(s)^\top\gamma(t)\}ds\\
=& \frac{1}{n-1}\exp\{\alpha_i(t)+\beta_j(t)+Z_{ij}(t)^\top\gamma(t)\} +O_p\big((e^{q_n}h^2)\big)~~j\in[n-1], j\neq i,\\
-\frac{\partial^2 F_{i,\gamma}(\eta(t)) }{ \partial \alpha_i(t)\partial \alpha_j(t)}=&0~~(j\neq i),~~
-\frac{\partial^2 F_{i,\gamma}(\eta(t))}{ \partial \beta_i(t)\partial \alpha_i(t)}=0,~~
-\frac{\partial^2 F_{i,\gamma}(\eta(t))}{ \partial \beta_j(t)\partial \beta_k(t)}=0,~~ j\neq k.
\end{align*}
By Condition \ref{condition:PB2}, we have
\begin{align*}
e^{-q_n}<\exp\{\alpha_i(t)+\beta_s(t)+Z_{ij}(t)^\top\gamma(t)\}\le e^{q_n}
\end{align*}
almost surely.
By the mean value theorem for vector-valued functions \citep{L1993}, we have
\[
F_{i,\gamma}^{(1)}(x(t)) - F_{i,\gamma}^{(1)}(y(t)) = J_i(t)\{x(t)-y(t)\},
\]
where $J_i(t)=(J_{i,sl}(t))\in\mathbb{R}^{(2n-1)\times (2n-1)}$ with
\[
J_{i,sl}(t) = \int_0^1 \frac{ \partial F_{i,s,\gamma}^{(1)}}{\partial \theta_l}(vx(t)+(1-v)y(t))dv,~~ s,l=1,\ldots, 2n-1.
\]
Because
\[
\max_s \sum_{l=1}^{2n-1} |J_{i,sl}(t)|\le 2e^{q_n}~~\text{and}~~\sum_{s,l}|J_{i,sl}(t)|\le e^{q_n}
\]
uniformly in $t\in[a,b],$
we have
\[
\| F_{i,\gamma}^{(1)}(x(t)) - F_{i,\gamma}^{(1)}(y(t)) \|_\infty \le \| J_i(t)\|_{\max} \|x(t) -y(t)\|_\infty \le 3e^{q_n},~~~i=1,\ldots, 2n-1,
\]
and for any vector $v\in \mathbb{R}^{2n-1}$,
\begin{align*}
\| [F_{\gamma}^{(1)}(x(t)) - F_{\gamma}^{(1)}(y(t))]v\|_\infty  & =  \max_i |\sum_{j=1}^{2n-1} ( F_{i,j}^{(1)}(x(t)) - F_{i,j}^{(1)}(y(t)) ) v_j | \\
& \le \|x(t)-y(t)\|_\infty \|v\|_\infty \sum_{s,l}|J_{i,sl}(t)|\\
&\le 3e^{q_n}\|x(t)-y(t)\|_\infty \|v\|_\infty
\end{align*}
uniformly in $t\in[a,b].$ This completes the proof.
\end{proof}

The following lemma characterizes the upper bound of the error between $\widehat{\eta}_{\gamma}(t)$ and $\eta_{\gamma}^*(t)$.

\begin{lemma}\label{lem6}
If $(q_n+1)e^{12q_n}\kappa_n = o( (nh_1)^{1/2}/(\log nh_1)^{1/2} ),$
then with probability tending to one, $\widehat{\eta}_{\gamma}(t)$
$(\gamma(t) \in \mathcal{D})$ exists and satisfies
\begin{align*}
\sup_{t\in[a,b]}\sup_{\gamma(t) \in \mathcal{D}}\| \widehat{\eta}_\gamma(t) - \eta_{\gamma}^*(t)\|_\infty = O_p\left((q_n+1)e^{6q_n}\kappa_n \sqrt{\frac{\log nh_1}{nh_1}} \right).
\end{align*}
\end{lemma}
\begin{proof}
Note that $\widehat{\eta}_{\gamma}(t)$ is the solution to
the equation $F_{\gamma}(\eta(t))=0.$
To prove this lemma, it is sufficient to show that the conditions in Lemma \ref{lem2} for $F_{\gamma}(\eta(t))$ hold.
In the Newton iterative step, we set $\eta_{\gamma}^{(0)}(t):=\eta_{\gamma}^*(t).$

In view of $(\eta_{\gamma}^*(t)^\top,\gamma(t)^\top)^\top \in \mathcal{B}$,
we have that 
$-F_{\gamma}(\eta_{\gamma}^*(t))\in \mathcal{L}_n(c_8e^{-q_n}, C_8e^{q_n})$, where $c_8$ and $C_8$ are absolute constants. 
Then, using Lemmas \ref{lem1} and \ref{lem4} and Lemma 1 of \cite{YLZ2016},
we have
\begin{eqnarray*}
r(t)&=&\| [F_{\gamma}^{(1)}(\eta_{\gamma}^*(t))]^{-1}F_{\gamma}(\eta_{\gamma}^*(t)) \|_{\infty}\\
& \le & \|[F_{\gamma}^{(1)}(\eta_{\gamma}^*(t))^{-1}-S_{\gamma}(t)]F_{\gamma}(\eta_{\gamma}^*(t)) \|_{\infty}+\|S_{\gamma}(t)F_{\gamma}(\eta_{\gamma}^*(t)) \|_{\infty} \\
& \le & (2n-1)\|F_{\gamma}^{(1)}(\eta_{\gamma}^*(t))^{-1}-S_{\gamma}(t)\|_{\max}\|F_{\gamma}(\eta_{\gamma}^*(t)) \|_{\infty}+\|S_{\gamma}(t)F_{\gamma}(\eta_{\gamma}^*(t)) \|_{\infty} \\
& \le & O_p\left((q_n+1)e^{6q_n}\sqrt{\frac{\log nh_1}{nh_1}}\right),
\end{eqnarray*}
where $S_{\gamma}(t)=\mathcal{S}(F_{\gamma}^{(1)}(\eta_{\gamma}^*(t))).$

This, together with Lemma \ref{lem5} and the condition $e^{12q_n} \kappa_n \sqrt{\log nh_1/(nh_1)}=o(1),$ implies that
\begin{align*}
\rho(t)r(t)= &O_p\big(e^{6q_n}\big)\times
O_p\left((q_n+1)e^{6q_n}\sqrt{\frac{\log nh_1}{nh_1}}\right) \\
= &O_p\left( (q_n+1)e^{12q_n} \kappa_n \sqrt{\frac{\log nh_1}{nh_1}} \right) =o_p(1).
\end{align*}
An application of Lemma \ref{lem2} yields
\begin{align*}
\sup_{t\in[a,b]}\sup_{\gamma(t)\in\mathcal{D}}\| \widehat{\eta}_\gamma(t) - \eta_{\gamma}^*(t) \|_\infty = O_p\left((q_n+1) e^{6q_n} \kappa_n \sqrt{\frac{\log nh_1}{nh_1}} \right)=o_p(1).
\end{align*}
It completes the proof.
\end{proof}

We now formally state the proof of Theorem \ref{theorem:consistency}.

\begin{proof}[Proof of Theorem \ref{theorem:consistency}]
Note that $\widehat Q_c(\gamma(t))=Q(\widehat{\eta}_\gamma(t), \gamma(t)).$
Define $Q_c(\gamma(t))=Q(\eta_{\gamma}^*(t),\gamma(t)).$
By Lemma \ref{lem4} and Taylor expansion of $\widehat Q_c(\gamma(t))$ at $\eta_{\gamma}^*(t),$ we have
\begin{align}\label{eqA5}
\|\widehat Q_c(\gamma(t))-Q_c(\gamma(t))\|
=O_p\left((q_n+1) e^{7q_n} \kappa_n^2 \sqrt{\frac{\log nh_1}{nh_1}} \right)
\end{align}
uniformly in $t\in[a,b]$ and $\gamma(t).$
By  the Donsker Theorem 
(e.g.  \citeauthor{V1998}, \citeyear{V1998}, Theorem 19.5),
$(n(n-1))^{-1}\sum_{i=1}^n\sum_{j\neq i}N_{ij}(t)=O_p(n^{-1})$
uniformly in $t\in[a,b].$
Note that
\begin{align*}
Q_c(\gamma(t))=\int_0^\tau&\mathcal{K}_{h_2}(s-t)d\Bigg[\frac{1}{N}\sum_{i=1}\sum_{j\neq i}Z_{ij}(s)N_{ij}(s)\Bigg]\\
&-\frac{1}{N}\sum_{i=1}\sum_{j\neq i}\int_0^{\tau}\mathcal{K}_{h_2}(s-t)Z_{ij}(s)\exp\{\alpha_{i,\gamma}^*(t)+\beta_{j,\gamma}^*(t)+Z_{ij}(s)^\top\gamma^*(t)\}ds.
\end{align*}
It follows that $Q_c(\gamma(t))=U(\gamma(t))+O_p((n^2h_2)^{-1})$ uniformly in $t\in[a,b]$, where 
\begin{align*}
U(\gamma(t))=&\lim_{N\rightarrow \infty}\frac{1}{N}\sum_{i=1}^n\sum_{j=1,j\neq i}^n
\mathbb{E}\Bigg\{Z_{ij}(t)\big(\exp\{\alpha_i^*(t)+\beta_j^*(t)+Z_{ij}(t)^\top\gamma^*(t)\}\\
&\hspace{1.6in}-\exp\{\alpha_{i,\gamma}^*(t)+\beta_{j,\gamma}^*(t)+Z_{ij}(t)^\top\gamma(t)\}\big)\Bigg\}.
\end{align*}
This, together with \eqref{eqA5}, gives that
\begin{align}
\nonumber
\|\widehat Q_c(\gamma(t))-U(\gamma(t))\|
\le& \|\widehat Q_c(\gamma(t))-Q_c(\gamma(t))\|+\|Q_c(\gamma(t))-U(\gamma(t))\|\\
\label{eqA6}
=&O_p\left((q_n+1) e^{7q_n} \kappa_n^2 \sqrt{\frac{\log nh_1}{nh_1}}+\frac{1}{n^2h_2} \right)
\end{align}
uniformly in $t\in[a,b]$ and $\gamma(t)$. 

Let $r_0$ be any positive constant such that $\|\gamma(t)-\gamma^*(t)\|\le r_0.$
For any $w\in \mathbb{R}^{p}$ satisfying $\|w\|=1$,
$w^\top U(\gamma^*(t)+wr)$ increases with $r$. 
It implies that for any $r \ge r_0>0$, 
$w^\top[U(\gamma^*(t)+wr)-U(\gamma^*(t))]\ge 0$. 
Then, we have
\begin{align*}
 & \|w\|\|U(\gamma^*(t)+wr)-U(\gamma^*(t))\| \\
\ge&|w^\top[U(\gamma^*(t)+wr)-U(\gamma^*(t))]|\\
\ge&|w^\top[U(\gamma^*(t)+wr,t)-U(\gamma^*(t))]|\\
\ge&w^\top H(\gamma^*(t)+w\bar r)w r>\bar\kappa r_0,
\end{align*}
where $\bar r\in[0,r_2]$ and $\bar\kappa>0$ is some constant.
Here the first inequality holds due to the Cauchy-Schwarz inequality,
and the last one follows from Condition \ref{condition:H3} and the continuity of $H(\gamma(t))$.
Therefore,
\begin{align}\label{eqA7}
\inf_{t\in[a,b]}\inf_{\|\gamma(t)-\gamma^*(t)\|>r_0}\|U(\gamma(t))-U(\gamma^*(t))\|>\bar\kappa r_0. 
\end{align}
Note that $\widehat Q_c(\widehat\gamma(t))=0$ almost surely and $U(\gamma^*(t))=0$. 
It then follows from \eqref{eqA6} that
\begin{align}
\nonumber
\|U(\widehat\gamma(t))-U(\gamma^*(t))\|=&\|\widehat Q_c(\widehat\gamma(t))-\{\widehat Q_c(\widehat\gamma(t))-U(\widehat\gamma(t))\}\|\\
\label{eqA8}
=&O_p\left((q_n+1) e^{7q_n} \kappa_n^2 \sqrt{\frac{\log nh_1}{nh_1}}+\frac{1}{n^2h_2} \right)=o_p(1),
\end{align}
and for sufficiently large $n$, 
$\|U(\widehat\gamma(t))\|<\bar\kappa r_0/2$
uniformly in $t\in[a,b]$.
Therefore, by \eqref{eqA7}, we obtain $\|\widehat\gamma(t)-\gamma^*(t)\|<r_0$ with probability tending to one.
A direct calculation yields
\begin{align*}
\frac{\partial U(\gamma(t))}{ \partial \gamma(t)^\top }=
\mathbb{E}\Big\{V_{\gamma,\gamma}(t)-V_{\gamma,\eta^*}(t)
V_{\eta^*,\eta^*}(t,\gamma(t))^{-1}V_{\eta^*,\gamma}(t)\Big\}.
\end{align*}
Taylor's expansion of $U(\widehat\gamma(t))$ at $\gamma^*(t)$ gives
\begin{align}
\nonumber
\sup_{t\in[a,b]}\|\widehat\gamma(t)-\gamma^*(t)\|
=&\sup_{t\in[a,b]}\|H_Q(\bar\gamma(t))^{-1}[U(\widehat\gamma(t))-U(\gamma^*(t))]\|\\
\nonumber
\le & \sup_{t\in[a,b]}\|H_Q(\bar\gamma(t))^{-1}\|\|U(\widehat\gamma(t))-U(\gamma^*(t))\|\\
\label{eqA9}
\le &\sup_{t\in[a,b]}\|H_Q(\bar\gamma(t))^{-1}\|\|U(\widehat\gamma(t))-U(\gamma^*(t))\|, 
\end{align}
where $\bar\gamma(t)$ is on the line segment between $\widehat\gamma(t)$ and $\gamma^*(t).$
By Condition \ref{condition:H3} and the continuity of $H_Q(\gamma(t))$,
there exists a positive constant $\varsigma$ such that
$\inf_{t\in[a,b]}\rho_{\min}\{H_Q(\bar\gamma(t))\}>\varsigma$. 
Thus, it follows from \eqref{eqA8} and \eqref{eqA9} that
$$
\sup_{t\in[a,b]}\|\widehat\gamma(t)-\gamma^*(t)\|=o_p(1).
$$


We next to show the consistency of $\widehat\eta(t)$. 
Let $\varepsilon_n=O((q_n+1) e^{7q_n} \kappa_n^2 \sqrt{\log(nh_1)/(nh_1)}+1/(n^2h_2))$.
By Lemma \ref{lem6}, it suffices to show
\begin{align}\label{eqA10}
\|\widehat\eta_{\gamma}(t)-\eta^*(t)\|_{\infty}=o_p(1)
\end{align}
uniformly in $t\in[a,b]$ and $\gamma(t)$ when $\|\gamma(t)-\gamma^*(t)\|<\varepsilon_n.$
Note that for $j=1,\dots,n-1,$
\begin{align*}
&(n-1)|F_i(\eta^*(t),\gamma^*(t))-F_i(\eta^*(t),\gamma(t))|\\
=&\bigg|\sum_{j\neq i}\int_{0}^{\tau}\exp\{\alpha_i^*(t)+\beta_j^*(t)\}\mathcal{K}_{h_1}(s-t)\Big[\exp\{Z_{ij}(s)^\top\gamma^*(t)\}ds -\exp\{Z_{ij}(s)^\top\gamma(t)\}\Big]\bigg|\\
=&\sum_{j\neq i}\int_{0}^{\tau}\mathcal{K}_{h_1}(s-t)\exp\{\alpha_i^*(t)+\beta_j^*(t)+Z_{ij}(s)^\top\tilde \gamma(t)\}|Z_{ij}(s)^\top\{\gamma^*(t)-\gamma(t)\}ds|\\
\le & 2ne^{q_n}\kappa_n\varepsilon_n.
\end{align*}
Similarly, for $j=1,\dots,n-1,$
we have $|F_{n+j}(\eta^*(t),\gamma^*(t))-F_{n+j}(\eta^*(t),\gamma(t))|\le e^{q_n}\kappa_n\varepsilon_n.$
Then, following the same line of the proof of Lemma \ref{lem6}, we get
\begin{eqnarray*}
r(t)&=&\| [F^{(1)}(\eta^*(t),\gamma(t))]^{-1}F(\eta^*(t),\gamma(t)) \|_{\infty}\\
& \le & \|[F^{(1)}(\eta^*(t),\gamma(t))^{-1}-\bar S(t)]F(\eta^*(t),\gamma(t)) \|_{\infty}+\|\bar S(t)F(\eta^*(t),\gamma(t)) \|_{\infty} \\
& \le & O_p\left((q_n+1)e^{13q_n}\kappa_n^3\sqrt{\frac{\log nh_1}{n h_1}}\right),
\end{eqnarray*}
and
\begin{align*}
\rho(t)r(t)= &O_p\big(e^{6q_n}\big)\times
O_p\left((q_n+1)e^{13q_n}\kappa_n^3\sqrt{\frac{\log nh_1}{n h_1}}\right) \\
= &O_p\left( (q_n+1)e^{19q_n} \kappa_n^3 \sqrt{\frac{\log nh_1}{nh_1}} \right),
\end{align*}
where $\bar S(t)=\mathcal{A}(F^{(1)}(\eta^*(t),\gamma(t)).$
Therefore, by Lemma \ref{lem2}, we have that (\ref{eqA10}) holds
if $(q_n+1)e^{19q_n} \kappa_n^3 \sqrt{\log(nh_1)/(nh_1)}=o(1).$
This completes the proof.
\end{proof}

\subsection{Proofs for Theorem \ref{theorem-central-gamma}}
\label{subsection-proof-th2}

The following lemma gives the asymptotic representation of 
$\widehat{\eta}_{\gamma^*}(t)$ that will be used the proof of 
Theorem \ref{theorem-central-gamma}. 

\begin{lemma}\label{lem7}
If $(q_n+1)e^{15q_n}\kappa_n= o( (nh_1)^{1/4}/(\log nh_1)^{1/2} )$,
then
\begin{equation}
\label{eq-lemma-widehat-eta}
\sqrt{nh_1}\big(\widehat{\eta}_{\gamma^*}(t)-\eta^*(t)\big)
= - [V_{\eta^*,\eta^*}(t,\gamma^*(t))]^{-1}\sqrt{\frac{h_1}{n}}\int_{0}^{\tau}\mathcal{K}_{h_1}(s-t)
d\widetilde{\mathcal{M}}(s)+o_p(1),
\end{equation}
where 
$\widetilde{\mathcal{M}}(s)=(\widetilde{M}_{1}(s),\dots,\widetilde{M}_{2n-1}(s))^\top$ and
\begin{align*}
\widetilde M_{i}(s)=
\begin{cases}
\sum_{k\ne i} \mathcal{M}_{ik}(s), & i=1,\dots, n, \\
\sum_{k\ne (i-n)} \mathcal{M}_{ki}(s), & i=n+1,\dots,2n-1.
\end{cases}
\end{align*}
Moreover, for a fixed $k$,
$\sqrt{nh_1}(\widehat{\eta}_{\gamma^*}(t)-\eta^*(t))_{1:k}$ converges in distribution to a
$k$-dimensional zero-mean normal random vector
with the covariance given by the upper-left $k\times k$ block of $\mu_{0}S^*(t)$,
where $\mu_{0}=\int \mathcal{K}^2(u)du$.
\end{lemma}

\begin{proof}
Let $V(t)=-V_{\eta^*,\eta^*}(t,\gamma^*(t))$
and $v_{ij}(t)$ be the $(i,j)$th element of $V(t)$. 
Define $v_{2n,2n}(t)=\sum_{i=1}^nv_{ii}(t)-\sum_{i=1}^n\sum_{j=1,j\neq i}^{2n-1}v_{ij}(t)$.
A second-order Taylor expansion of $F(\widehat{\eta}_{\gamma^*}(t),\gamma^*(t))$ gives
\begin{align*}
\widehat{\eta}_{\gamma^*}(t)-\eta^*(t)=& [V(t)]^{-1}
\underbrace{F(\eta^*(t),\gamma^*(t))}_{ B_{31}(t)} \\
&+ [V(t)]^{-1}\underbrace{\Bigg[\frac{1}{2n}\sum_{k=1}^{2n-1}\big(\widehat{\eta}_{k,\gamma^*}(t)-\eta_k^*(t)\big)
\frac{\partial F(\bar\eta(t),\gamma^*(t))}{\partial \eta_k(t)\partial \eta(t)^\top}\big(\widehat{\eta}_{\gamma^*}(t)-\eta^*(t)\big)\Bigg]}_{ B_{32}(t) },
\end{align*}
where $\bar\eta(t)$ is between $\widehat{\eta}_{\gamma^*}(t)$ and $\eta^*(t)$. 
Let $B_{32,l}(t)$ be the $l$th element of $B_{32}(t)$. 
Let $b_{l,ij}(t)=\partial F_{l}(\bar\eta(t),\gamma^*(t))/\partial \eta_i(t)\partial \eta_j(t).$
Note that for $1\le l\le n$, we have 
\begin{align*}
|B_{32,l}(t)|= & \frac{1}{2n}\Big|\big(\widehat{\eta}_{\gamma^*}(t)-\eta^*(t)\big)^\top
\frac{\partial F_{l}(\bar\eta(t),\gamma^*(t))}{\partial \eta(t)\partial \eta(t)^\top}\big(\widehat{\eta}_{\gamma^*}(t)-\eta^*(t)\big)\Big|\\
\le & \frac{1}{2n}\|\widehat{\eta}_{\gamma^*}(t)-\eta^*(t)\|_{\infty}^2
\sum_{i,j=1}^{2n-1}|b_{l,ij}(t)|,
\end{align*}
and 
\begin{align*}
b_{l,ij}(t) = \begin{cases}
-\sum_{j\neq l}\int_0^{\tau}\mathcal{K}_{h_1}(s-t)\exp\{\bar\alpha_l(t)+\bar\beta_j(t)+Z_{lj}(s)^\top\gamma^*(t)\}ds, & i=j= l, \\
-\int_0^{\tau}\mathcal{K}_{h_1}(s-t)\exp\{\bar\alpha_l(t)+\bar\beta_{j-n}(t)+Z_{ij}(s)^\top\gamma^*(t)\}ds,  & i=j\neq l, \\
 &\mbox{or~} i=l, j\neq i, \mbox{~or~} i\neq j, j=l, \\
0, & \mbox{otherwise}. 
\end{cases}
\end{align*}
Then, by Lemma \ref{lem6}, we have
\begin{align*}
|B_{32,l}(t)|\le & \frac{1}{2n}\|\widehat{\eta}_{\gamma^*}(t)-\eta^*(t)\|_{\infty}^2
\sum_{i,j=1}^{2n-1}|b_{l,ij}(t)|
\le 2 e^{q_n}\|\widehat{\eta}_{\gamma^*}(t)-\eta^*(t)\|_{\infty}^2\\
=&O_p\Big((q_n+1)^2e^{25q_n}\kappa_n^2 \frac{\log nh_1}{nh_1}\Big).
\end{align*}
In view of Lemmas \ref{lem1} and \ref{lem6}, for $1\le l\le n$, we have
\begin{align*}
|(V(t)^{-1}B_{32})_l|\le& |[(V(t)^{-1}-S(t))B_{32}]_l|+|(S(t)B_{32})_l|\\
\le&(2n-1)\|V(t)^{-1}-S(t)\|_{\max}\|B_{32}(t)\|_{\infty}
+\max_{1\le l\le 2n-1}\frac{|B_{32,l}(t)|}{v_{ll}(t)}\\
&+\|\widehat{\eta}_{\gamma^*}(t)-\eta^*(t)\|_{\infty}^2\times\frac{|\sum_{i=1}^{n-1}b_{l,in}(t)|}{nv_{2n,2n}(t)}\\
=&O_p\Big((q_n+1)^2e^{30q_n}\kappa_n^2\frac{\log nh_1}{nh_1}\Big)=o_p((nh_1)^{-1/2}).
\end{align*}
Similarly, we have $(V(t)^{-1}B_{32}(t))_l=o_p((nh_1)^{-1/2})$ for $n+1\le l\le 2n-1$.
This shows 
\begin{align*}
\sqrt{nh_1}\big(\widehat{\eta}_{\gamma^*}(t)-\eta^*(t)\big)
=&V(t)^{-1}B_{31}(t)+o_p(1).
\end{align*}
Now, we analyze $B_{31}(t)$ that is rewritten as
\begin{align*}
B_{31}(t)=\sqrt{\frac{h_1}{n}}\int_{0}^{\tau}\mathcal{K}_{h_1}(s-t) d \tilde{\mathcal{M}}(s) - B_{31,2}(t),
\end{align*}
where $B_{31,2}(t)=(B_{31,21}(t),\dots,B_{31,2(2n-1)}(t))^\top$ and 
\begin{align*}
B_{31,2i}(t)&=\frac{1}{n}\sum_{j\neq i}\int_{0}^{\tau}\mathcal{K}_{h_1}(s-t)\big[\exp\{\pi_{ij}^*(s,t)\}-\exp\{\pi_{ij}^*(s)\}\big]ds
~~\text{for}~~i\in[n],\\
B_{31,2(n+j)}(t)&=\frac{1}{n}\sum_{i\neq j}\int_{0}^{\tau}\mathcal{K}_{h_1}(s-t)\big[\exp\{\pi_{ij}^*(s,t)\}-\exp\{\pi_{ij}^*(s)\}\big]ds
~~\text{for}~~j\in[n-1].
\end{align*}
In order to prove \eqref{eq-lemma-widehat-eta}, it is sufficient to demonstrate 
\[
V(t)^{-1}B_{31,2}=o_p((nh_1)^{-1/2}).
\]
A direct calculation gives
\begin{align}\label{eqA11}
|(V(t)^{-1}B_{31,2})_l|
\le&\|V(t)^{-1}-S(t)\|_{\max}\sum_{i=1}^{2n-1}|B_{31,2i}(t)|
+\max_{1\le l\le 2n-1}\frac{|B_{31,2l}(t)|}{v_{ll}(t)}\\
&+\frac{|\sum_{i=1}^{n-1}b_{31,2 in}(t)|}{nv_{2n,2n}(t)},
\end{align}
where $b_{31,2 in}(t)=\int_{0}^{\tau}\mathcal{K}_{h_1}(s-t)[\exp\{\pi_{in}^*(s,t)\}-\exp\{\pi_{in}^*(s)\}]ds$.
By Conditions \ref{condition:CB1}-\ref{condition:bandwidth5}, we have
\begin{align*}
\mathbb{E}\{B_{31,2i}(t)\}=O(e^{q_n}h_1^2),\quad
\text{Var}(B_{31,2i}(t))=O(e^{2q_n}h_1^4/n).
\end{align*}
Therefore, by Lemma \ref{lem1} and Chebyshev's inequality, the first term in the right-hand side of \eqref{eqA11}
is of order $O_p(e^{6q_n}h_1^2)$.
Similarly, it can be shown that the last two terms  in the right-hand side of \eqref{eqA11} are of order 
$O_p( e^{2q_n} h_1^2 )$. 
These facts, together with $ne^{12q_n}h_1^5\rightarrow 0,$ imply
\begin{align*}
|(V(t)^{-1}B_{31,2})_l|=O_p(e^{6q_n}h_1^2)=o_p((nh_1)^{-1/2}).
\end{align*}
This completes the proof of \eqref{eq-lemma-widehat-eta}. 

Now, we prove the second part of Lemma \ref{lem7}.
With similar arguments as in the proof of \eqref{ineq-bound-Ft}, 
by Lemma \ref{lemma3}, we have
\begin{align}\label{eqA12}
\sup_{t\in [a,b]} \max_{1\le i\le n}\bigg|\frac{1}{n}\sum_{j\neq i}\int_0^{\tau}\mathcal{K}_{h_1}(s-t)[e^{\pi_{ij}^*(s,t)}-\mathbb{E}e^{\pi_{ij}^*(t)}]ds\bigg|&=O_p(e^{q_n}\sqrt{\log(nh_1)/(nh_1)}),\\
\label{eqA13}
\sup_{t\in[a,b]} \max_{1\le j\le n-1}\bigg|\frac{1}{n}\sum_{i\neq j}\int_0^{\tau}\mathcal{K}_{h_1}(s-t)[e^{\pi_{ij}^*(s,t)}
-\mathbb{E}e^{\pi_{ij}^*(t)}]ds\bigg|&=O_p(e^{q_n}\sqrt{\log(nh_1)/(nh_1)}).
\end{align}
Thus, $n^{-1}\sum_{j\neq i}\sum_{j\neq i}\int_0^{\tau}\mathcal{K}_{h_1}(s-t)e^{\pi_{ij}^*(s,t)}ds$ and $n^{-1}\sum_{i\neq j}\sum_{j\neq i}\int_0^{\tau}\mathcal{K}_{h_1}(s-t)e^{\pi_{ij}^*(s,t)}ds$
converge in probability to $\lim_{n\rightarrow \infty}n^{-1}\sum_{j\neq i}\mathbb{E}e^{\pi_{ij}^*(t)}$
and $\lim_{n\rightarrow \infty} n^{-1}\sum_{i\neq j}\mathbb{E}e^{\pi_{ij}^*(t)},$ respectively.
Then, we have
\begin{align*}
\sqrt{nh_1}\big(\widehat{\eta}_{\gamma^*}(t)-\eta^*(t)\big)
=&V(t)^{-1}\sqrt{\frac{h_1}{n}}\int_{0}^{\tau}\mathcal{K}_{h_1}(s-t)d\tilde{\mathcal{M}}(s)+o_p(1)\\
=& \underbrace{(V(t)^{-1}-S(t))\sqrt{\frac{h_1}{n}}\int_{0}^{\tau}\mathcal{K}_{h_1}(s-t)d\tilde{\mathcal{M}}(s)}_{ B_{41} } \\
& + \underbrace{(S(t)-S^*(t))\sqrt{\frac{h_1}{n}}\int_{0}^{\tau}\mathcal{K}_{h_1}(s-t)d\tilde{\mathcal{M}}(s)}_{ B_{42}} \\
& + \underbrace{S^*(t)\sqrt{\frac{h_1}{n}}\int_{0}^{\tau}\mathcal{K}_{h_1}(s-t)d\tilde{\mathcal{M}}(s)}_{ B_{43} } + o_p(1).
\end{align*}
By Lemma \ref{lem1}, a direct calculation yields that $\mathbb{E}(B_{41})=0$ and $\|\text{Cov}(B_{41})\|_{\max}=O(e^{11q_n}/n).$
Therefore, $B_{41}=o_p(1).$
For $B_{42},$ its $l$th ($1\le l\le n$) element can be written as
\begin{align*}
B_{42,l}=&\sum_{k=1}^{2n-1}(s_{lk}(t)-s_{lk}^*(t))\sqrt{\frac{h_1}{n}}\int_{0}^{\tau}\mathcal{K}_{h_1}(s-t)d \tilde M_{l}(s)\\
=&\Big(\frac{1}{v_{ll}(t)}-\frac{1}{v_{ll}^*(t)}\Big)\sqrt{\frac{h_1}{n}}\int_{0}^{\tau}\mathcal{K}_{h_1}(s-t)d \tilde M_{l}(s)\\
&+\sqrt{\frac{h_1}{n}}\Big(\frac{1}{v_{2n,2n}(t)}-\frac{1}{v_{2n,2n}^*(t)}\Big)
\sum_{i=1}^{n-1}\int_{0}^{\tau}\mathcal{K}_{h_1}(s-t)d \mathcal{M}_{in}(s),
\end{align*}
which is of order $O_p(e^{q_n}\sqrt{\log(n)/n})$ by (\ref{eqA12}).
This fact, together with $B_{41}=o_p(1),$ gives
\[\sqrt{nh_1}\big(\widehat{\eta}_{\gamma^*}(t)-\eta^*(t)\big)
=S^*(t)\sqrt{\frac{h_1}{n}}\int_{0}^{\tau}\mathcal{K}_{h_1}(s-t)d\tilde{\mathcal{M}}(s)+o_p(1).
\]
Note that the $i$th element $U_i(t)$ of
$\sqrt{h_1/n}S^*(t)\int_{0}^{\tau}\mathcal{K}_{h_1}(s-t)d\tilde{\mathcal{M}}(s)$
is
\begin{align*}
U_i(t)=&\sqrt{\frac{h_1}{n}}\frac{\sum_{j\neq i}\int_0^{\tau}\mathcal{K}_{h_1}(s-t)d\mathcal{M}_{ij}(s)}{v_{ii}^*(t)}
+\sqrt{\frac{h_1}{n}}\frac{\sum_{l=1}^{n-1}\int_0^{\tau}\mathcal{K}_{h_1}(s-t)d\mathcal{M}_{ln}(s)}{v_{2n,2n}^*(t)},
~i\in[n],\\
U_{n+j}(t)=&\sqrt{\frac{h_1}{n}}\frac{\sum_{i\neq j}\int_0^{\tau}\mathcal{K}_{h_1}(s-t)d\mathcal{M}_{ij}(s)}{v_{ii}^*(t)}
-\sqrt{\frac{h_1}{n}}\frac{\sum_{l=1}^{n-1}\int_0^{\tau}\mathcal{K}_{h_1}(s-t)d\mathcal{M}_{ln}(s)}{v_{2n,2n}^*(t)},
~j\in[n-1].
\end{align*}
Therefore, $U_i(t)~(i=1,\dots,2n-1)$ are sums of local square-integrable martingales with the quadratic variation process given by $S^*(t)\bar V(t,u)S^*(t),$ where $\bar V(t,u)=(\bar v_{ij}(t,u))$ and
\begin{align*}
\bar v_{i,i}(t,u)=&\frac{h_1}{n}\sum_{j\neq i}\int_0^{u}\mathcal{K}_{h_1}^2(s-t)e^{\pi_{ij}^*(s)}ds,~~i\in[n],\\
\bar v_{n+j,n+j}(t,u)=&\frac{h_1}{n}\sum_{i\neq j}\int_0^{u}\mathcal{K}_{h_1}^2(s-t)e^{\pi_{ij}^*(s)}ds,~~j\in[n-1],\\
\bar v_{i,n+j}(t,u)=&\bar v_{n+j,i}(t,u)=\frac{h_1}{n}\int_0^{u}\mathcal{K}_{h_1}^2(s-t)e^{\pi_{ij}^*(s)}ds,~~i\in[n],
~j\in[n-1],\\
\bar v_{i,j}(t,u)=&0,~~\text{otherwise}.
\end{align*}
It can be shown that $S^*(t)\bar V(t,\tau)S^*(t)\rightarrow \mu_{0}S^*(t)$
with probability tending to 1.
Specifically, by the uniform law of large numbers,
$\bar v_{ij}(t,\tau)$ converge in probability to its expectation, 
whose expression is 
\begin{align*}
\mathbb{E}\{\bar v_{i,i}(t)\}=&\frac{1}{n}\sum_{j\neq i}\mathbb{E}\{\mu_{0}e^{\pi_{ij}^*(t)}+\mu_{12}e^{\pi_{ij}^*(t)}\pi_{ij}^{*(1)}(t)h+O_p(q_ne^{q_n}h^2)\}\\
=&\mu_{0}v_{i,i}^*(t)+\mu_{12}h\tilde v_{i,i}^*(t)+O(q_ne^{q_n}h^2),~~i\in[n],\\
\mathbb{E}\{\bar v_{n+j,n+j}(t)\}=&\frac{1}{n}\sum_{i\neq j}\mathbb{E}\{\mu_{0}e^{\pi_{ij}^*(t)}+\mu_{12}e^{\pi_{ij}^*(t)}\pi_{ij}^{*(1)}(t)h+O_p(q_ne^{q_n}h^2)\}\\
=&\mu_{0}v_{n+j,n+j}^*(t)+\mu_{12}h\tilde v_{n+j,n+j}^*(t)+O(q_ne^{q_n}h^2),~~j\in [n-1],\\
\mathbb{E}\{\bar v_{i,j}(t)\}=&\mathbb{E}\{\bar v_{j,i}(t)\}=\frac{1}{n}\mathbb{E}\{\mu_{02}e^{\pi_{ij}^*(t)}
+\mu_{12}e^{\pi_{ij}^*(t)}\pi_{ij}^{*(1)}(t)h+O_p(q_ne^{q_n}h^2)\}\\
=&\mu_{0}v_{ij}^*(t)+h\mu_{12}\tilde v_{ij}^*(t)+O(q_ne^{q_n}h^2/n),~~i\in[n],~j\in n+[n-1],\\
\mathbb{E}\{\bar v_{i,j}(t)\}=&0, ~~~\text{otherwise}~~~(\text{Here, we set}~~ v_{ij}^*(t)=0~\text{and}~\tilde v_{ij}^*(t)=0).
\end{align*}
Therefore,
$$\mathbb{E} \bar V(t,\tau)=\mu_{0}V^*(t)+h\mu_{12}\widetilde V^*(t)+O(q_ne^{q_n}h^2),$$
where $V^*(t)=(v_{ij}^*(t))$ and $\widetilde V^*(t)=(\tilde v_{ij}^*(t)).$
Then, the $(i,j)$th element of $S^*(t)\widetilde V^*(t)$ satisfies 
\begin{align*}
|\sum_{k=1}^{2n-1}s_{ik}^*(t)\tilde v_{ki}^*(t)|\le &
\frac{|\tilde v_{ii}^*(t)|}{v_{ii}^*(t)}+\frac{|\tilde v_{ni}^*(t)|}{v_{2n,2n}^*(t)}\le q_ne^{2q_n}(1+1/n),~
i=j\in[n],\\
|\sum_{k=1}^{2n-1}s_{ik}^*(t)\tilde v_{ki}^*(t)|\le&
\frac{|\tilde v_{ii}^*(t)|}{v_{ii}(t)}\le q_ne^{2q_n},~i=j\in n+[n-1],\\
|\sum_{k=1}^{2n-1}s_{ik}^*(t)\tilde v_{kj}^*(t)|
\le&\frac{|\tilde v_{nj}^*(t)|}{v_{2n,2n}^*(t)}\le q_ne^{2q_n}/n,~i\neq j\in[n],\\
|\sum_{k=1}^{2n-1}s_{ik}^*(t)\tilde v_{kj}^*(t)|=&|\sum_{k=1}^{2n-1}s_{jk}^*(t)\tilde v_{ki}(t)|
\le\frac{|\tilde v_{i,n+i}^*(t)|}{v_{ii}^*(t)}+\frac{|\tilde v_{i,n+i}^*(t)|}{v_{2n,2n}^*(t)}\\
\le&q_ne^{2q_n}/n,~i\in[n],~j=\in n+[n-1]~(j\neq n+i),\\
|\sum_{k=1}^{2n-1}s_{ik}^*(t)\tilde v_{kj}^*(t)|=&0,~j=n+i~(i\in[n-1])~ \text{or}~ i=n+j~(j\in[n-1]).
\end{align*}
This implies
\begin{equation}
\label{eq-lemma-SV-diff}
\|S^*(t)\tilde V^*(t)-q_ne^{2q_n}\mathbb{I}\|_{\max}\le 2q_ne^{2q_n}/n,
\end{equation}
where $\mathbb{I}$ denotes a $(2n-1)\times (2n-1)$ matrix consisting of all ones. 
It follows that
\begin{eqnarray*}
 & & \|S^*(t)\tilde V^*(t)S^*(t)\|_{\max} \\
& \le &
\|(S^*(t)\tilde V^*(t)-q_ne^{2q_n}\mathbb{I})S^*(t)\|_{\max}+q_ne^{2q_n}\|S^*(t)\|_{\max}\\
& \le & \|(S_0(t)\tilde V(t)-q_ne^{2q_n}\mathbb{I})\|_{\max}\times \max_{1\le j\le 2n-1}\sum_{k=1}^{2n-1}|s_{kj}^*(t)|+q_ne^{2q_n}\|S^*(t)\|_{\max}\\
& = & O(q_ne^{3q_n}).
\end{eqnarray*}
That is, 
\[
\|S^*(t)h\mu_{12}\tilde V^*(t)S^*(t)\|_{\max}=O(hq_ne^{3q_n}).
\]
With similar arguments as in the proof of \eqref{eq-lemma-SV-diff}, 
we have $S^*(t)V^*(t)=\mathbb{I}+q_ne^{2q_n}(C_n-\mathbb{I})$
and $\|S^*(t)(C_n-\mathbb{I})\|_{\max}=q_ne^{3q_n}/n$. 
These facts imply
\begin{eqnarray*}
& & \|S^*(t)\mathbb{E}\{\bar V(t,\tau)\}S^*(t)-\mu_{0} S^*(t)\|_{\max} \\
& \le &
\|S^*(t)\mathbb{E}\{\bar V(t,\tau)\}S^*(t)-\mu_{0} S^*(t)V^*(t)S^*(t)\|_{\max}\\
& & +\mu_{0} \|S^*(t)V^*(t)S^*(t)- S^*(t)\|_{\max}\\
& = & O\big(hq_ne^{2q_n}+q_ne^{3q_n}/n\big)\rightarrow 0.
\end{eqnarray*}
Moreover, for any $\varepsilon>0$, because $\mathcal{K}_{h_1}(x)$ is bounded by $O(h_1^{-1}),$
\begin{align*}
I\Big(\sqrt{\frac{h_1}{n}}\mathcal{K}_{h_1}(s-t)/v_{ii}^*(t)>\varepsilon\Big)=I(O(e^{q_n}(nh_1)^{-1/2})>\varepsilon)\rightarrow 0,\\
I\Big(\sqrt{\frac{h_1}{n}}\mathcal{K}_{h_1}(s-t)/v_{2n,2n}^*(t)>\varepsilon\Big)=I(O(e^{q_n}(nh_1)^{-1/2})>\varepsilon)\rightarrow 0
\end{align*}
as $n\to \infty$. Thus, with probability tending to 1,
\begin{align*}
&\frac{h_1}{n}\sum_{j\neq i,j<n}\int_0^{\tau}\mathcal{K}_{h_1}^2(s-t)\frac{e^{\pi_{ij}^*(s)}}{[v_{ii}^*(t)]^2}
I\Big(\sqrt{\frac{h_1}{n}}\mathcal{K}_{h_1}(s-t)/v_{ii}(t)>\varepsilon\Big)ds\\
&+\frac{h_1}{n}\sum_{l=1}^{n-2}\int_0^{\tau}\mathcal{K}_{h_1}^2(s-t)\frac{e^{\pi_{ij}^*(s)}}{[v_{2n,2n}^*(t)]^2}
I\Big(\sqrt{\frac{h_1}{n}}\mathcal{K}_{h_1}(s-t)/v_{2n,2n}(t)>\varepsilon\Big)ds\\
&+\frac{h_1}{n}\int_0^{\tau}\mathcal{K}_{h_1}^2(s-t)e^{\pi_{in}^*(s)}\Big(\frac{1}{v_{ii}^*(t)}+\frac{1}{v_{2n,2n}^*(t)}\Big)^2
I\Big( \xi >\varepsilon\Big)ds\\
&\rightarrow 0,
\end{align*}
where
\[
\xi= \sqrt{\frac{h_1}{n}}\mathcal{K}_{h_1}(s-t)\frac{(v_{ii}^*(t)+v_{2n,2n}^*(t))}{v_{ii}^*(t)v_{2n,2n}^*(t)}.
\]
Thus, by Theorem 5.1.1 of \cite{FH2005}, we conclude that
for any fixed $k$,
$(U_1(t),\dots,U_k)^\top$ converges in distribution to a $k$-dimensional zero-mean normal random vector
with the covariance given by the upper-left $k\times k$ block of $\mu_{0}S^*(t)$.

\end{proof}

\begin{proof}[Proof of Theorem \ref{theorem-central-gamma}]
Note that $\widehat{\eta}(t)=\widehat{\eta}_{\widehat{\gamma}}(t).$
A mean value expansion gives
\[
\widehat Q_c(\widehat{\gamma}(t))
- \widehat Q_c(\gamma^*(t))
=   \frac{\partial \widehat Q_c(\bar\gamma(t))}{\partial \gamma(t)}(\widehat{\gamma}-\gamma^*),
\]
where $\bar{\gamma}(t)=s\gamma^*(t)+(1-s)\widehat{\gamma}(t)$ for some $s\in (0, 1)$.
By the definition of $\widehat{\gamma}(t),$ we get
\[
\sqrt{Nh_2}\{\widehat{\gamma}(t) - \gamma^*(t)\} = -
\Bigg[\frac{\partial \widehat Q_c(\bar\gamma(t))}{\partial \gamma(t)} \Bigg]^{-1}\Bigg[\sqrt{\frac{h_2}{N}} \sum_{i=1}^n\sum_{j=1,j\neq i}^n \xi_{ij}(\widehat{\eta}_{\gamma^*}(t),\gamma^*(t))\Bigg],
\]
where
\[\xi_{ij}(\widehat{\eta}_{\gamma^*}(t),\gamma^*(t))=
\int_{0}^{\tau}\mathcal{K}_{h_2}(s-t)Z_{ij}(s)\Big[dN_{ij}(s)- \exp\big\{\widehat{\pi}_{ij,\gamma^*}(t)\big\}ds\Big],
\]
and
\[
\widehat{\pi}_{ij,\gamma^*}(t)=
\widehat{\alpha}_{i,\gamma^*}(t)+\widehat{\beta}_{j,\gamma^*}(t)+Z_{ij}(s)^\top\gamma^*(t).
\]
A direct calculation yields
\begin{align*}
\frac{ \partial \widehat{Q}_c(\bar \gamma(t))}{ \partial \gamma(t)^\top }=
V_{\bar\gamma,\bar\gamma}(t)-V_{\bar\gamma,\hat\eta}(t)
\big[V_{\hat\eta,\hat\eta}(t,\bar\gamma(t))\big]^{-1}V_{\hat\eta,\bar\gamma}(t),
\end{align*}
whose limit is $H_Q(t).$
Therefore,
\begin{align}\label{eqA13}
\sqrt{Nh_2}\{\widehat{\gamma}(t)- \gamma^*(t)\}
= [H_Q(t)]^{-1}\Bigg[\sqrt{\frac{h_2}{N}} \sum_{i=1}^n\sum_{j\neq i} \xi_{ij}(\widehat{\eta}_{\gamma^*}(t),\gamma^*(t))\Bigg]+o_p(1).
\end{align}
A three-order Taylor expansion of $\xi_{ij}(\widehat{\eta}_{\gamma^*}(t),\gamma^*(t))$ at $\eta^*(t)$ gives
\begin{equation}
\sqrt{\frac{h_2}{N}}\sum_{i=1}^n \sum_{j\neq i} \xi_{ij}(\widehat{\eta}_{\gamma^*}(t),\gamma^*(t))
=B_{51}(t) + B_{52}(t) + B_{53}(t),
\end{equation}
where
\begin{eqnarray*}
B_{51}(t) & = & \sqrt{\frac{h_2}{N}}  \sum_{i=1}^n \sum_{j\neq i} \xi_{ij}(\eta^*(t),\gamma^*(t))\\
& &+ \sqrt{\frac{h_2}{N}}\sum_{i=1}^n \sum_{j\neq i}
\Big[\frac{\partial}{\partial \eta(t)^\top } \xi_{ij}(\eta^*(t),\gamma^*(t))\Big]\{\widehat{\eta}_{\gamma^*}(t) - \eta^*(t)\}, \\
B_{52}(t)  & =  &  \frac{1}{2}\sqrt{\frac{h_2}{N}} \sum_{k=1}^{2n-1} \Big[( \widehat{\eta}_{k,\gamma^*}(t) - \eta_{k}^*(t)) \sum_{i=1}^n \sum_{j\neq i}
\frac{\partial^2 \xi_{ij}(\eta^*(t),\gamma^*(t))}{ \partial \eta_k(t) \partial \eta(t)^\top }
\times \{\widehat{\eta}_{\gamma^*}(t) - \eta^*(t)\} \Big],  \\
B_{53}(t) & = & \frac{1}{6}\sqrt{\frac{h_2}{N}} \sum_{k=1}^{2n-1} \sum_{l=1}^{2n-1}\Bigg\{( \widehat{\eta}_{k,\gamma^*}(t) - \eta_{k}^*(t))
(\widehat{\eta}_{l,\gamma^*}^*(t) - \eta_{l}^*(t))\\
&&~~~~\times\Big[\sum_{i=1}^n \sum_{j\neq i}\frac{ \partial^3 \xi_{ij}(\bar{\eta}_{\gamma^*}(t),\gamma^*(t))}{ \partial \eta_k(t) \partial \eta_l(t) \partial \eta(t)^\top } \Big]
\{\widehat{\eta}_{\gamma^*}(t) - \eta^*(t)\}\Bigg\}.
\end{eqnarray*}
In the above equations, 
$\bar{\eta}_{\gamma^*}(t)=s\widehat{\eta}_{\gamma^*}(t)+(1-s){\eta}^*(t)$ for some $s\in(0,1)$,
and $\widehat\eta_{k,\gamma^*}(t)$ is the $k$th component of $\widehat\eta_{\gamma^*}(t).$
It is sufficient to demonstrate: (i) $B_{51}(t)$ converges in distribution to a normal distribution; (ii) 
$B_{52}(t)=b_*(t)+o_p(1)$, where $b_*(t)$ is given in \eqref{eq-bias}; (iii) $B_{53}(t)$ is an asymptotically negligible remainder term. 
These claims are shown in three steps in an inverse order. 

Step 1. We show $B_{53}(t)=o_p(1)$. 
Calculate $g^{ij}_{kls}(t)=\frac{ \partial^3 \xi_{ij}(\bar{\eta}_{\gamma^*}(t),\gamma^*(t))}{\partial \eta_k(t) \partial \eta_l(t) \partial \eta_s(t)}$ according to the indices $k,l,s$ as follows.
Note that $g^{ij}_{kls}(t)=0$ when $k,l,s\notin \{i,n+j\}.$ 
So there are only two cases in which  $g^{ij}_{kls}\neq 0$. \\
(1) Only two values among three indices $k, l, s$ are equal.
If $k=l=i$ and $s=n+j$, then
$g^{ij}_{kls}(t)=-Z_{ij}\exp\{\bar\alpha_{i,\gamma^*}(t)+\bar\beta_{j,\gamma^*}(t)+Z_{ij}^\top\gamma^*(t)\}.$
For other cases, the results are similar.\\
(2) If $k=l=s=i$ or $k=l=s=n+j,$
then $g^{ij}_{kls}(t)=-Z_{ij}\exp\{\alpha_i(t)+\beta_j(t)+Z_{ij}^\top\gamma^*(t)\}.$ \\
Therefore, we have
\begin{align*}
B_{53}(t)=&\frac{1}{6}\sqrt{\frac{h_2}{N}} \sum_{k,l,s=1}^{2n-1}g_{kls}^{ij}(t)( \widehat{\eta}_{k,\gamma^*}(t) - \eta_{k}^*(t))
(\widehat{\eta}_{l,\gamma^*}^*(t) - \eta_{l}^*(t))( \widehat{\eta}_{s,\gamma^*}(t) - \eta_{s}^*(t))\\
=&\frac{1}{6}\sqrt{\frac{h_2}{N}}\sum_{i=1}^n\sum_{j=1,j\neq i}^n g_{iii}^{ij}(t)( \widehat{\eta}_{i,\gamma^*}(t) - \eta_{i}^*(t))^3
\\
&+ \frac{1}{6}\sqrt{\frac{h_2}{N}}\sum_{i=1}^n\sum_{j=1,j\neq i}^n g_{jjj}^{ij}(t)( \widehat{\eta}_{n+j,\gamma^*}(t) - \eta_{n+j}^*(t))^3\\
&+\frac{1}{2}\sqrt{\frac{h_2}{N}}\sum_{i=1}^n\sum_{j=1,j\neq i}^n g_{iij}^{ij}(t)( \widehat{\eta}_{i,\gamma^*}(t) - \eta_{i}^*(t))^2(\widehat{\eta}_{n+j,\gamma^*}(t) - \eta_{n+j}^*(t))\\
&+\frac{1}{2}\sqrt{\frac{h_2}{N}}\sum_{i=1}^n\sum_{j=1,j\neq i}^n g_{jji}^{ij}(t)( \widehat{\eta}_{i,\gamma^*}(t) - \eta_{i}^*(t))( \widehat{\eta}_{n+j,\gamma^*}(t) - \eta_{n+j}^*(t))^2
\end{align*}
This, together with the definition of $g_{kls}^{ij}(t)$, implies that
\begin{align*}
\|B_{53}(t)\|_\infty \le2\sqrt{Nh_2}\max_{i,j}
\left\{ \exp\{ \bar{\pi}_{ij, \gamma^*} \|Z_{ij}(t)\|_\infty \right\}
\times \| \widehat{\eta}_{\gamma^*}(t) - \eta^*(t)\|_\infty^3,
\end{align*}
where
\[
\bar{\pi}_{ij, \gamma^*}=\bar \alpha_{i,\gamma^*}(t)+\bar\beta_{j,\gamma^*}(t)+Z_{ij}(t)^\top\gamma^*(t)\}. 
\]
By Lemma \ref{lem6}  and Conditions \ref{condition:CB1}-\ref{condition:PB2}, we have
\begin{align}\label{eqA14}
\|B_{53}(t)\|_\infty =  O_p\Big(\frac{\kappa_n^4 (q_n+1)^3e^{19q_n} (\log nh_1)^{3/2}h_2^{1/2}}{n^{1/2}h_1^{3/2}}\Big)=o_p(1).
\end{align}

Step 2. We show claim (ii). Note that
\begin{align*}
B_{52}(t)  =  &  \frac{1}{2}\sqrt{\frac{h_2}{N}} \sum_{k=1}^{2n-1} \Big[( \widehat{\eta}_{k,\gamma^*}(t) - \eta_{k}^*(t)) \sum_{i=1}^n \sum_{j=1,j\neq i}^n
\frac{\partial^2 \xi_{ij}(\eta^*(t),\gamma^*(t))}{ \partial \eta_k(t) \partial \eta(t)^\top }
\times \{\widehat{\eta}_{\gamma^*}(t) - \eta^*(t)\} \Big]\\
=&\frac{1}{2}\sqrt{\frac{h_2}{N}} \sum_{i=1}^n \sum_{j=1,j\neq i}^n  \frac{\partial^2 \xi_{ij}(\eta^*(t),\gamma^*(t))}
{ \partial \eta_i^2(t)}[\widehat{\eta}_{i,\gamma^*}(t) - \eta_{i}^*(t)]^2\\
&+\frac{1}{2}\sqrt{\frac{h_2}{N}} \sum_{j=1}^n \sum_{i=1,i\neq j}^n  \frac{\partial^2 \xi_{ij}(\eta^*(t),\gamma^*(t))}
{ \partial \eta_{n+j}^2(t)}[\widehat{\eta}_{n+j,\gamma^*}(t) - \eta_{n+j}^*(t)]^2\\
&+\sqrt{\frac{h_2}{N}} \sum_{i=1}^n \sum_{j=1,j\neq i}^n  \frac{\partial^2 \xi_{ij}(\eta^*(t),\gamma^*(t))}
{ \partial \eta_i(t)\partial \eta_{n+j}(t)}[\widehat{\eta}_{i,\gamma^*}(t) - \eta_{i}^*(t)]
[\widehat{\eta}_{n+j,\gamma^*}(t) - \eta_{n+j}^*(t)].
\end{align*}
Then, similar to the calculation in the derivation of the asymptotic bias in Theorem 4 in \cite{G2017}, 
we have
\begin{align}\label{eqA15}
B_{52}(t)=b_*(t)+o_p(1), 
\end{align}
where
\begin{equation}
\label{eq-bias}
b_*(t)=\frac{\mu_{0}}{2Nh_1}\bigg[ \sum_{i=1}^n \frac{  \sum_{j\neq i}\mathbb{E}(Z_{ij}(t) \exp\{\pi_{ij}^*(t)\})}
{\sum_{j\neq i}\mathbb{E}(\exp\{\pi_{ij}^*(t)\})}+\sum_{j=1}^n \frac{  \sum_{i\neq j}\mathbb{E}(Z_{ij}(t) \exp\{\pi_{ij}^*(t)\})}
{\sum_{i\neq j}\mathbb{E}(\exp\{\pi_{ij}^*(t)\})}\bigg],
\end{equation}
and $\pi_{ij}^*=\alpha_i^*(t) + \beta_j^*(t) +Z_{ij}(t)^\top \gamma^*(t)$.

Step 3. We show claim (i). 
Let $\iota_{ij}$ be a $(2n-1)$-dimensional vector
with the $i$th and $(n+j)$th elements being one and others being zero.
For $B_{51}(t),$ we have
\begin{align*}\label{eqA16}
B_{51}(t) = \sqrt{\frac{h_2}{N}}\sum_{i=1}^n\sum_{j=1,j\neq i}^n \tilde{\xi}_{ij} (\eta^*(t), \gamma^*(t)) + o_p(1),
\tag{A.16}
\end{align*}
where
$$
\tilde{\xi}_{ij} (\eta^*(t), \gamma^*(t))=\int_{0}^{\tau}Z_{ij}(s)\mathcal{K}_{h_2}(s-t)d\mathcal{M}_{ij}(s)
-V_{\gamma^*\eta^*}(t)S^*(t)\int_{0}^{\tau}\mathcal{K}_{h_2}(s-t)d\mathcal{M}_{ij}(s)\iota_{ij}.
$$
For any fixed $t\in[a,b]$, 
with similar arguments as in the proof of Lemma \ref{lem7},
we have that the distribution of $B_{51}(t)$
is approximately normal with mean $0$ and covariance matrix $\mu_{0}\tilde\Sigma(t)$,
where
\begin{align*}
\tilde\Sigma(t)=\frac{1}{N}\sum_{i=1}^n\sum_{j\neq i}\mathbb{E}\bigg[\Big(Z_{ij}(t)
-V_{\gamma^*\eta^*}(t)S^*(t)\iota_{ij}\Big)^{\otimes 2}e^{\pi_{ij}^*(t)}\bigg].
\end{align*}
Here, for any vector $a$ with $a^{\otimes2}=aa^{\top}$,
substituting \eqref{eqA14}-\eqref{eqA16} into \eqref{eqA13} gives
\[
\sqrt{Nh_2}\big(\widehat{\gamma}(t)- \gamma^*(t)- H_Q(t)^{-1} b(t)\big)
= H_Q(t)^{-1}\sqrt{\frac{h_2}{N}}\sum_{i=1}^n \sum_{j\neq i}
\tilde{\xi}_{ij} (\eta^*(t), \gamma^*(t)) + o_p(1),
\]
which converges in distribution to a $p$-dimensional multivariate normal random vector with mean 0
and covariance matrix $\mu_{0}H_Q(t)^{-1}\Sigma(t)H_Q(t).$
It completes the proof.
\end{proof}

\subsection{Proof of Theorem \ref{theorem-central-degree}}
\label{section-proof-th3}
In this section, we present the proof of Theorem \ref{theorem-central-degree}.

\begin{proof}[Proof of Theorem \ref{theorem-central-degree}]
To simplify notations, write $\widehat{\pi}_{ij}(t)=\widehat{\alpha}_i(t)+\widehat{\beta}_j(t)+Z_{ij}(t)^\top \widehat{\gamma}(t)$.
Let $\widehat{\theta}_{ij}(t)=(\widehat{\alpha}_i(t), \widehat{\beta}_j(t), \widehat{\gamma}(t)^\top)^\top$
and ${\theta}_{ij}^*(t)=(\alpha_i^*(t), \beta_j^*(t),  \gamma^*(t)^\top  )^\top.$
By Taylor's expansion, we have
\begin{align}
\nonumber
&F(\widehat{\eta}(t),\widehat{\gamma}(t)) - F(\eta^*(t),\gamma^*(t))\\
\label{eqA17}
 =& V_{\eta^*,\eta^*}(t,\gamma^*(t))\{\widehat{\eta}(t) - \eta^*(t)\}-
 V_{\eta^*,\gamma^*}(t)\{\widehat{\gamma}(t)-\gamma^*(t)\}+n^{-1}g(t),
\end{align}
where 
$g(t)=(g_1(t), \ldots, g_{2n-1}(t))^\top$, $g_i(t)=\sum_{j\neq i}g_{ij}(t)~(1\le i\le n)$, $g_{n+j}(t)=\sum_{i\neq j}g_{ij}(t)~(j\in[n-1])$,
and $g_{ij}(t)=\int_0^{\tau}\mathcal{K}_{h_1}(s-t)\bar g_{ij}(s,t)ds$ with
\[
\bar g_{ij}(s,t)=  (\widehat{\theta}_{ij}(t)-\theta_{ij}^*(t))^\top
\begin{pmatrix}
 e^{\tilde{\pi}_{ij}(s,t)} & e^{\tilde{\pi}_{ij}(s,t)}  & e^{\tilde{\pi}_{ij}(s,t)}Z_{ij}(s)^\top \\
 e^{\tilde{\pi}_{ij}(s,t)} & e^{\tilde{\pi}_{ij}(s,t)}  & e^{\tilde{\pi}_{ij}(s,t)}Z_{ij}(s)^\top \\
 e^{\tilde{\pi}_{ij}(s,t)}Z_{ij}(s) & e^{\tilde{\pi}_{ij}(s,t)} Z_{ij}(s) & e^{\tilde{\pi}_{ij}(s,t)}Z_{ij}(s)Z_{ij}(s)^\top
\end{pmatrix}(\widehat{\theta}_{ij}(t)-\theta_{ij}^*(t)).
\]
Here, $\tilde{\pi}_{ij}(t)$ lies between $\pi_{ij}^*(t)$ and $\widehat{\pi}_{ij}(t)$.
Recall that $V(t)=V_{\eta^*,\eta^*}(t,\gamma^*(t))$.
Because $F(\widehat{\eta}(t),\widehat{\gamma}(t))=0$,
\eqref{eqA17} is equivalent to
\begin{eqnarray*}
\widehat{\eta}(t) - \eta^*(t)
  =   -[V(t)]^{-1} F(\eta^*(t),\gamma^*(t)) +  [V(t)]^{-1} V_{\eta^*,\gamma^*}(t)\{\widehat{\gamma}(t)-\gamma^*(t)\}
 - \frac{1}{n}[V(t)]^{-1}g(t).
\end{eqnarray*}
The remainder of the proof is to show that the first term in the right-hand side of the above 
equation is asymptotically normal, and the second and third terms are asymptotically 
negligible. These claims are shown in the following three steps. 

Step 1. We show 
\begin{equation}\label{eq-proof-th3-2}
[V(t)]^{-1}g(t)=o_p((nh_1)^{-1/2}).
\end{equation}
Note that $|g_{ij}(t)-g_{ij}^*(t)|=O(\kappa_ne^{q_n}h_1^2)$, 
where 
\[
\bar g_{ij}^*(t)=  (\widehat{\theta}_{ij}(t)-\theta_{ij}^*(t))^\top
\begin{pmatrix}
 e^{\tilde{\pi}_{ij}(t)} & e^{\tilde{\pi}_{ij}(t)}  & e^{\tilde{\pi}_{ij}(t)}Z_{ij}(t)^\top \\
 e^{\tilde{\pi}_{ij}(t)} & e^{\tilde{\pi}_{ij}(t)}  & e^{\tilde{\pi}_{ij}(t)}Z_{ij}(t)^\top \\
 e^{\tilde{\pi}_{ij}(t)}Z_{ij}(t) & e^{\tilde{\pi}_{ij}(t)} Z_{ij}(t) & e^{\tilde{\pi}_{ij}(t)}Z_{ij}(t)Z_{ij}(t)^\top
\end{pmatrix}(\widehat{\theta}_{ij}(t)-\theta_{ij}^*(t)).\]
A direct calculation gives
\begin{eqnarray*}
g_{ij}^*(t) & = &  e^{\tilde{\pi}_{ij}(t)} [(\widehat{\alpha}_i(t)-\alpha_i^*(t))^2 +  (\widehat{\beta}_j(t)-\beta_j^*(t))^2 + 2(\widehat{\alpha}_i(t)-\alpha_i^*(t))(\widehat{\beta}_j(t)-\beta_j^*(t))]
\\
&& + 2e^{\tilde{\pi}_{ij}(t)}Z_{ij}(t)^\top ( \widehat{\gamma}(t) - \gamma(t)) (\widehat{\alpha}_i(t)-\alpha_i^*(t)+\widehat{\beta}_j(t)-\beta_j^*(t))\\
&&+e^{\tilde{\pi}_{ij}(t)}( \widehat{\gamma}(t) - \gamma(t))^\top Z_{ij}(t)Z_{ij}(t)^\top( \widehat{\gamma}(t) - \gamma^*(t)).
\end{eqnarray*}
Note that $\kappa_n := \max_{i,j} \|Z_{ij}(t) \|$ and $e^{\tilde{\pi}_{ij}(t)}\le e^{2q_n}$. By Theorem \ref{theorem:consistency} and Condition \ref{condition:PB2},
we have
\begin{align}
\nonumber
|g_{ij}(t)| \le & 4e^{2q_n} \| \widehat{\eta}(t)- \eta^*(t)\|_\infty^2
+ 4\kappa_n e^{2q_n}\| \widehat{\eta}(t) - \eta^*(t)\|_\infty \| \widehat{\gamma}(t)-\gamma^*(t) \|\\
\nonumber
&+ e^{2q_n}\kappa_n^2 \| \widehat{\gamma}(t)-\gamma^*(t) \|^2 \\
\label{eqA18}
\le & 2e^{2q_n}[4 \| \widehat{\eta}(t) - \eta^*(t)\|_\infty^2+  \kappa_n^2\| \widehat{\gamma}(t)-\gamma^*(t) \|^2].
\end{align}
Write $( V(t)^{-1} g(t))_i=(S(t) g(t))_i+(W(t)g(t))_i$,
where $W(t)=V(t)^{-1}- S(t)$.
For $i\in[n]$, a direct calculation yields 
\begin{align*}
n^{-1}(S(t) g(t))_i=\frac{g_i(t)}{v_{ii}(t)}+\frac{\sum_{j=1}^{n-1} g_{jn}(t)}{nv_{2n,2n}(t)}.
\end{align*}
By \eqref{eqA18} and Theorem \ref{theorem:consistency},
we have 
\begin{align*}
|n^{-1}(S(t) g(t))_i|\le &2e^{3q_n}\{4 \| \widehat{\eta}(t) - \eta^*(t)\|_\infty^2+  \kappa_n^2\| \widehat{\gamma}(t)-\gamma^*(t) \|^2\}\\
=&O_p\Big((q_n+1)^2e^{41q_n}\kappa_n^6\frac{\log nh_1}{nh_1}\Big).
\end{align*}
In addition, by Lemma \ref{lem1}, we have
\begin{align*}
\|n^{-1}W(t)g(t)\|_\infty \le  \|W(t) \|_{\max} \|g(t)\|_\infty=O_p\Big((q_n+1)^2e^{45q_n}\kappa_n^6\frac{\log nh_1}{nh_1}\Big).
\end{align*}
Therefore, if $(q_n+1)^2e^{45q_n}\kappa_n^6\log nh_1/\sqrt{nh_1}=o(1),$ then
\begin{equation*}
\|n^{-1}W(t)g(t)\|_\infty=o_p\big((nh_1)^{-1/2}\big).
\end{equation*}
This shows \eqref{eq-proof-th3-2}. 

Step 2. We show
\begin{equation}
\label{eq-proof-th3-3}
[V(t)]^{-1}V_{\eta^*\gamma^*}(t)\{\widehat{\gamma}(t)-\gamma^*(t)\}=o_p((nh_1)^{-1/2}).
\end{equation}
Note that
\[
V_{\eta^*,\gamma^*}(t) = \left(\begin{array}{c}
\frac{1}{n}\sum_{j\neq 1} \int_0^{\tau}\mathcal{K}_{h_1}(s-t)e^{\pi_{1j}^*(s,t)}Z_{1j}(s)^\top ds, \\
\vdots \\
\frac{1}{n}\sum_{j\neq n}\int_0^{\tau}\mathcal{K}_{h_1}(s-t) e^{\pi_{nj}^*(s,t)}Z_{nj}(s)^\top ds,\\
\frac{1}{n}\sum_{j\neq 1}\int_0^{\tau}\mathcal{K}_{h_1}(s-t) e^{\pi_{j1}^*(s,t)}Z_{j1}(s)^\top ds, \\
\vdots \\
\frac{1}{n}\sum_{j\neq n-1}\int_0^{\tau}\mathcal{K}_{h_1}(s-t) e^{\pi_{j,n-1}^*(s,t)}Z_{j,n-1}(s)^\top ds,\\
\end{array}
\right).
\]
We have
\[
\|V_{\eta^*,\gamma^*}(t)(\widehat{\gamma}(t)-\gamma^*(t)) \|_\infty \le (2-1/n) e^{q_n}\kappa_n\|\widehat{\gamma}(t)-\gamma^*(t)\|.
\]
This, together with Theorem \ref{theorem-central-gamma}, yields
\begin{align*}
\|S(t)V_{\eta^*,\gamma^*}(t)(\widehat{\gamma}(t)-\gamma^*(t))\|_\infty\le & \max_{i} \frac{1}{v_{ii}(t)} \|V_{\eta^*\gamma^*}(t) (\widehat{\gamma}(t)-\gamma^*(t))\|_\infty\\
& +\frac{|\sum_{j=1}^{n-1}e^{\pi_{jn}(t)}Z_{jn}(t)^\top(\widehat{\gamma}(t)-\gamma^*(t))|}{v_{2n,2n}(t)}\\
= & O_p\left(\frac{e^{2q_n}\kappa_n}{n\sqrt{h_2}}\right).
\end{align*}
Furthermore, we have
\begin{align*}
\|W(t)V_{\eta^*\gamma^*}(t)(\widehat{\gamma}(t)-\gamma^*(t))\|_\infty
\le&(2n-1)\|W(t)\|_{\max} \|V_{\eta^*\gamma^*}(t) (\widehat{\gamma}(t)-\gamma^*(t))\|_\infty \\
=&O_p\left(\frac{e^{6q_n}\kappa_n}{n\sqrt{h_2}}\right).
\end{align*}
This, together with $h_2=O(h_1^2)$ and $e^{6q_n}\kappa_n/\sqrt{n}=o(1)$, gives
\eqref{eq-proof-th3-3}.

Step 3 is a combination step. 
By \eqref{eq-proof-th3-2} and \eqref{eq-proof-th3-3}, we have
\begin{align*}
\sqrt{nh_1}\big\{\widehat{\eta}_i(t)-\eta_i^*(t)\big\}=
\Bigg\{S^*(t)\sqrt{\frac{h_1}{n}}\int_{0}^{\tau}\mathcal{K}_{h_1}(s-t)d\tilde{\mathcal{M}}(s)\Bigg\}_i+o_p(1).
\end{align*}
By Lemma \ref{lem7},
we have that for any fixed $k$,
$\big[(\mu_{0}S^*(t))^{-1/2}\{\widehat{\eta}_i(t)-\eta^*(t)\}\big]_{1:L}$ converges in distribution to a
$k$-dimensional standard normal random vector.
It completes the proof.
\end{proof}

\begin{figure*}[h]
	\centering
	\subfloat[Subfigure 1 list of figures text][$t=0.6$]{
		\includegraphics[width=0.4\textwidth]{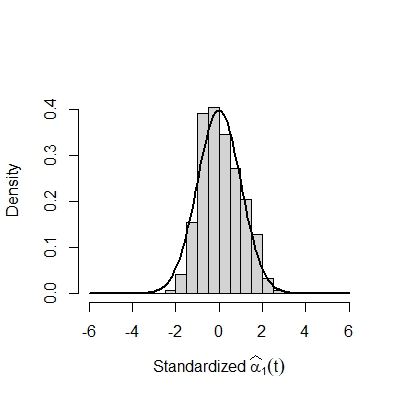}}
	\subfloat[Subfigure 3 list of figures text][$t=0.8$]{
		\includegraphics[width=0.4\textwidth]{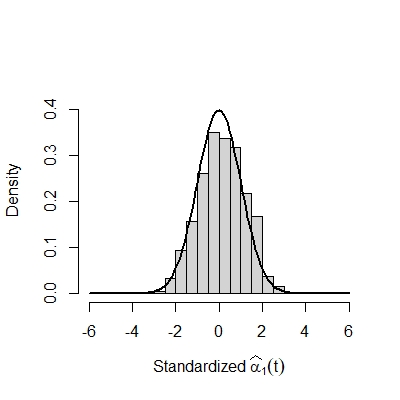}}\\
  	\subfloat[Subfigure 1 list of figures text][$t=0.6$]{
		\includegraphics[width=0.4\textwidth]{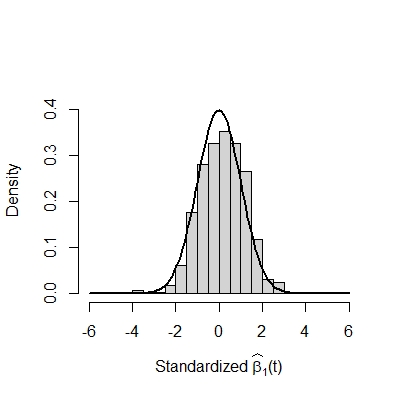}}
	\subfloat[Subfigure 3 list of figures text][$t=0.8$]{
		\includegraphics[width=0.4\textwidth]{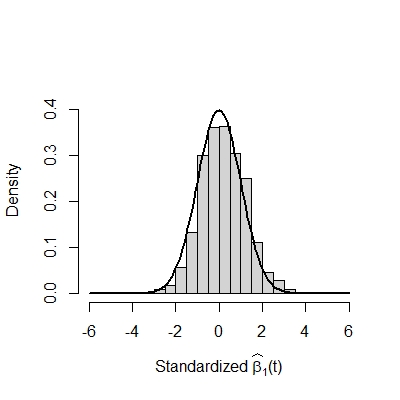}}\\
  	\subfloat[Subfigure 1 list of figures text][$t=0.6$]{
		\includegraphics[width=0.4\textwidth]{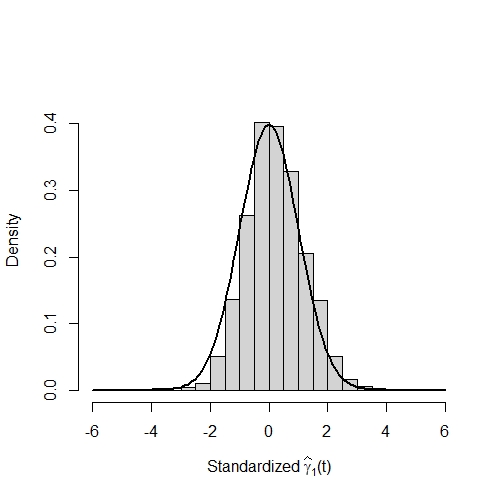}}
	\subfloat[Subfigure 3 list of figures text][$t=0.8$]{
		\includegraphics[width=0.4\textwidth]{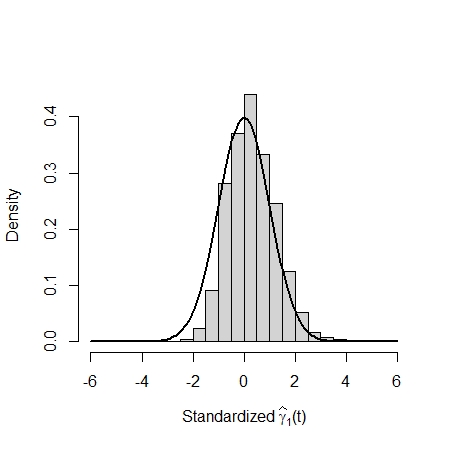}}
	\qquad
	\caption*{Figure S1. Simulation results: Asymptotic normality for standardized $\widehat\alpha_1(t)$, $\widehat\beta_1(t)$ and $\widehat\gamma_1(t)~(t=0.6~\text{and}~0.8)$ with $n=500.$}
	\label{fig:asym}
\end{figure*}

\begin{figure*}[h]
	\centering
	\subfloat[Subfigure 1 list of figures text][In-degree]{
		\includegraphics[width=0.4\textwidth]{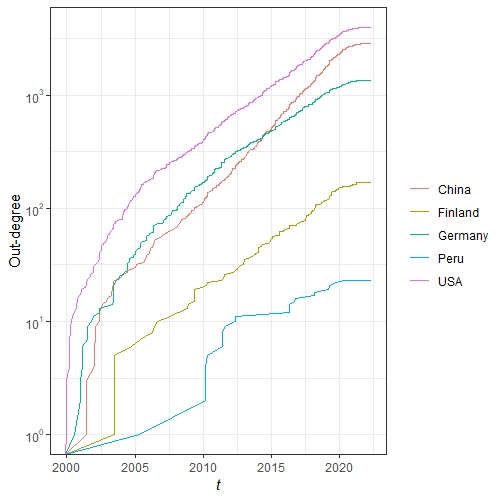}}
	\subfloat[Subfigure 3 list of figures text][Out-degree]{
		\includegraphics[width=0.4\textwidth]{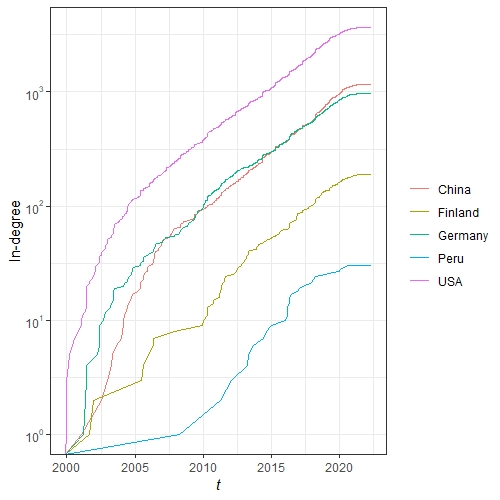}}
	\caption*{Figure S2. Real data analysis: The curves of the in- and out-degrees of 5 selected countries.}
\end{figure*}

\begin{figure*}[h]
	\centering
	\subfloat[Subfigure 1 list of figures text][In-degree]{
		\includegraphics[width=0.4\textwidth]{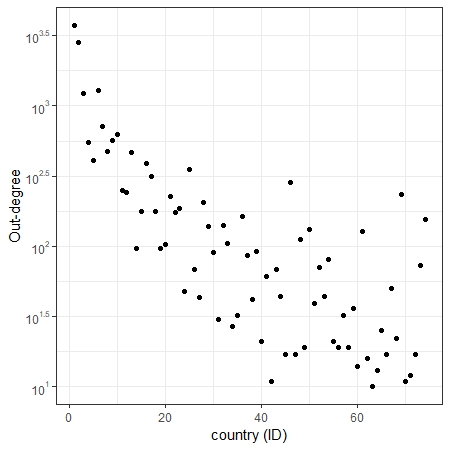}}
	\subfloat[Subfigure 3 list of figures text][Out-degree]{
		\includegraphics[width=0.4\textwidth]{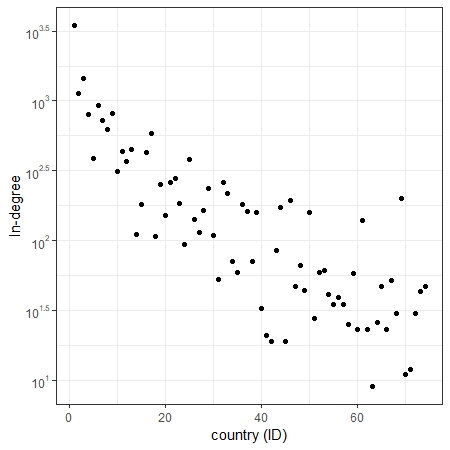}}
	\caption*{Figure S3. Real data analysis: The total in- and out-degrees of 74 countries from Jan. 2000 to Apr. 2022.}
\end{figure*}

{}

\end{document}